\documentclass[12pt]{amsart}
\usepackage{graphicx, amssymb, amsthm, amsfonts, latexsym, tabularx, epsfig, amsmath, amscd, amsbsy}
\usepackage[usenames]{color}
\usepackage[dvipsnames]{xcolor}
\usepackage{tikz}
\usepackage{enumerate}
\usepackage{epstopdf}
\usepackage{adjustbox}
\usepackage{hyperref}

\usetikzlibrary{decorations.pathmorphing}

\theoremstyle{plain}
\newtheorem{theorem}{Theorem}[section]
\newtheorem{lem}[theorem]{Lemma}
\newtheorem{cor}[theorem]{Corollary}
\newtheorem{prop}[theorem]{Proposition}
\theoremstyle{definition}

\newtheorem{df}[theorem]{Definition}
\newtheorem{rem}[theorem]{Remark}

\newtheorem*{claim*}{Claim}

\def\R{{\mathbb R}}
\def\N{{\mathbb N}}
\def\Z{{\mathbb Z}}
\def\D{{\mathbb D}}
\def\T{{\mathbb T}}

\def\d{{\mathcal D}}
\def\C{{\mathcal C}}
\def\K{{\mathcal K}}
\def\O{{\mathcal O}}

\def\A{{\mathcal A}}
\def\M{{\mathcal M}}

\def\H{{\mathcal H}}

\def\WY{\mathcal{WY}}

\def\eps{\varepsilon}
\def\phi{\varphi}
\def\mf{\mathfrak}

\def\mfc{\mathfrak{C}}
\def\mfk{\mathfrak{K}}

\def\Int{\mathop\mathrm{Int}}
\def\Cl{\mathop\mathrm{Cl}}

\def\length{\mathop\mathrm{length}}

\newcommand{\diam}{\operatorname{diam}}

\newcommand{\invlim}{\displaystyle\lim_{\longleftarrow}}

\renewcommand{\setminus}{\smallsetminus}
\def\nin{\notin}
\renewcommand{\hat}{\widehat}
\renewcommand{\tilde}{\widetilde}

\def\bar{\overline}
\newcommand{\la}{\overleftarrow}
\newcommand{\ra}{\overrightarrow}

\newcommand{\h}{\hspace{-.1cm}}

\newcommand{\tl}{\triangleleft}

\def\rpinf{+^{\hspace{-.1cm}\infty}}
\def\lpinf{\hspace{.2cm}+^{\hspace{-.5cm}\infty}\hspace{.2cm}}

\begin{document}
	
	\title[The pruning front conjecture, folding patterns and a classification]{The pruning front conjecture, folding patterns and classification of H\'enon maps in the presence of strange attractors}
	
	\author{Jan P. Boro\'nski}\thanks{JB: Supported by the National Science Centre, Poland (NCN), grant no. 2019/34/E/ST1/00237: ``Topological and 
		Dynamical Properties in Parameterized Families of Non-Hyperbolic Attractors: the inverse limit approach''.}
	\author{Sonja \v{S}timac}\thanks{S\v{S}: Supported in part by the Croatian Science Foundation grant IP-2022-10-9820 GLODS, and by 
		 the program Excellence Initiative – Research University at the Jagiellonian University in Kraków. A part of this work was done while S\v{S} was a Guest Professor of the College of Natural Sciences at Seoul National University, and another part while she was a Jagiellonian Scholar at the Faculty of Mathematics and Computer Science at the Jagiellonian University in Kraków. The hospitality of the Department of Mathematical Sciences, Seoul National University, and the Faculty of Mathematics and Computer Science, Jagiellonian University in Kraków, is gratefully acknowledged.}\thanks{JB and S\v{S} are grateful to Banff International Research Station, and its sponsors, for holding the workshop ``Dynamics of H\'enon maps: real, complex and beyond'' in April of 2023, and providing a very productive environment for our work.}
	
%	\date{\today}
	
	\maketitle
	
	\begin{abstract}
		We study the topological dynamics of the H\'enon maps. For a parameter set generalizing the Benedicks-Carleson parameters (the Wang-Young parameter set) we obtain the following:
		\begin{enumerate} 
			\item The pruning front conjecture (due to Cvitanovi\'c);
			\item A kneading theory (realizing a conjecture by Benedicks and Carleson); 
			\item A classification: two H\'enon maps are conjugate on their strange attractors if and only if their sets of kneading sequences coincide, if and only if their folding patterns coincide. 
		\end{enumerate}
		The folding pattern is a single sequence of $0$s and $1$s, which allows to distinguish two nonconjugate H\'enon attractors in finitely many steps. The classification result relies on further development of the authors' recent inverse limit description of H\'enon attractors in terms of densely branching trees.
	\end{abstract}	
	
	
	
	\baselineskip=18pt
	
	\section{Introduction}\label{sec:intro}
	
	\subsection{H\'enon attractors} 
	In 1976 M.\ H\'enon introduced the two-parameter family of smooth diffeomorphisms on the plane $F_{a,b} : \R^2 \to \R^2$, $F_{a,b}(x,y) = (1 + y - ax^2, bx)$, to show that numerical experiments for $a = 1.4$, $b = 0.3$ indicate the existence of a strange attractor, \cite{H}. Since then, this family, known today as the H\'enon family of maps, has been the subject of intense research, both theoretical and numerical. However, its dynamics is still far from being completely understood. The H\'enon attractor is the prototypical strange attractor and it is one of the most studied examples of dynamical systems exhibiting chaotic behavior. Understanding the H\'enon-type strange attractors is a central problem of nonlinear dynamics because such attractors are believed to be generic for low-dimensional dynamical systems. The first proof of the existence of the H\'enon strange attractors was given by M.\ Benedicks and L.\ Carleson in 1991, \cite{BC}. They proved that if $b > 0$ is small enough, then for a positive measure set of values $a$ near $a^* = 2$, the corresponding H\'enon map $F_{a,b}$ exhibits a strange attractor. In the introduction of this remarkable paper, they also stated four problems that they believed were possible to solve by further development of their ideas. The fourth problem, the only one unsolved until today, is the following one: 
	
	\noindent
	{\bf Problem} \cite[p.\ 74, Problem (d)]{BC} {\it Develop a partial theory of kneading sequences for the H\'enon maps.} 
	
	After the first breakthrough, several generalizations of the Benedicks-Carleson result appeared. In 1993, L.\ Mora and M.\ Viana generalized the result to the so-called H\'enon-like families of maps, \cite{MV}. In 2001, Q.\ Wang and L.-S.\ Young in \cite{WY} gave simple conditions for strongly dissipative maps that guarantee the existence of strange attractors. They developed a dynamical picture for the attractors in this class, including the geometry of fractal critical sets, nonuniform hyperbolic behavior, symbolic coding of orbits, formulas for topological entropy, and many statistical properties associated with chaos. Their results hold for the H\'enon family of maps for a positive measure set of parameters $(a, b)$, arbitrarily near $(a^*, 0)$, where $a^* \in [1.5, 2]$ is such that $q_{a^*}(x) = 1 - a^*x^2$ is a Misiurewicz map and $b$ can be both, positive and negative. We call this set of parameters the \emph{Wang-Young parameter set} and denote it by $\WY$. Note that the Benedicks-Carleson results in \cite{BC} is a version of the case for $a^* = 2$ and $b > 0$. There is a vast literature on the H\'enon family and strange attractors, as well as their generalizations, see e.g. \cite{BenedicksYoung93}, \cite{BenedicksYoung00}, \cite{BenedicksViana}, \cite{BenedicksViana2}, \cite{dCLM}, \cite{HLM}, and \cite{WYAnn}, including some of the most recent advances \cite{BMP}, \cite{BP}, \cite{BergerJEMS19},\cite{BergerA19},\cite{CP}, \cite{CPNotices}, \cite{CroPujTre}, and \cite{HMT}.
	
	\subsection{Kneading theory in dimension 1}
	A very powerful tool in one-dimensional dynamics is the Milnor-Thurston kneading theory, \cite{MT}. Its central object for unimodal maps on an interval is the kneading sequence. This symbol sequence is defined as the itinerary of the critical value, and it is an invariant of the topological conjugacy classes of the unimodal maps with negative Schwartzian derivative and without periodic attractors. The kneading sequence completely characterizes the set of all possible itineraries of such unimodal maps. It also has a key role in various studies, like in proving the continuity and monotonicity of the topological entropy for the quadratic family (and the universality of the family).
	
	\subsection{Pruning front conjecture}	
	In 1988 in \cite{CGP}, P.\ Cvitanovi\'c, G.H.\ Gunaatne, and I.\ Procaccia proposed a symbolic dynamics model for what they called H\'enon-type maps, that is, for maps on the plane that fold the plane back into itself exactly once, and they especially addressed the case of the H\'enon and Lozi maps. They stated that every allowed orbit has a unique binary label and that a binary tree that contains all allowed orbits is not complete, since disallowed orbits are `pruned' from it. They called a boundary that separates the allowed and the forbidden orbits a `pruning front'. They write: 
	
	``A pruning front arises from the fact that the two-dimensional attractor is not a single curve, but is multisheeted. There is no single `maximal point' as in the one-dimensional case; each sheet defines a locally highest allowed itinerary. The pruning front is computed therefore by determining the `primary' sequence of homoclinic tangencies. The homoclinic tangencies are points where the unstable manifold and the stable manifold are tangent. `Primary' tangencies are those tangencies that lie closest to the $y$ axis [...] We conjecture that the pruning front specifies the allowed symbolic dynamics (the union of all periodic points) fully; there are no orbits that are pruned out by other mechanisms. All the other disallowed regions of the symbol plane are obtained by backward and forward iterations of the primary pruned patch [...]''
	
	This conjecture is known as the pruning front conjecture and has received a lot of attention in the last 35 years (see for example \cite{dC}, \cite{dCH2}, \cite{DN}, \cite{Me}). Note that at the time when the aforementioned paper was published, a rigorous mathematical proof of the existence of the H\'enon attractors did not exist. The pruning front conjecture as introduced in \cite{CGP} is described intuitively and relies on numerical evidence to validate the models. A version of its precise mathematical statement (for the H\'enon maps) is given in 2002 by A.\ de Carvalho and T.\ Hall in \cite{dCH}, where they suggested a possible approach to a proof of the conjecture. This approach is based on the so-called braid type conjecture, which states that the periodic orbits of H\'enon maps are of the same topological type as the periodic orbits of the horseshoe. 
	
	\subsection{Statement of results}
	
	The aim of this paper is as follows:
	\begin{itemize}
		\item To prove the pruning front conjecture;
		\item To develop the kneading theory for the H\'enon maps within the Wang-Young set of parameters $\WY$;
		\item To classify (up to topological conjugacy) the H\'enon maps on their strange attractors within $\WY$ (in terms of the kneading theory).
	\end{itemize} 
	
	\subsubsection{Motivation for the approach and the construction of a critical locus.}		Our approach to the pruning front conjecture is motivated in part by the piecewise affine case. For the Lozi family, the first proof of the conjecture was given by Y.\ Ishii in 1997, \cite{I}. His approach is adapted to the piecewise affine nature of the Lozi family and is quite different from the one in \cite{dCH}, and also from the one adopted here. Our approach is motivated by the kneading theory for the Lozi maps developed in 2016 by M.\ Misiurewicz and the second author of this paper, \cite{MS}. In 2018 the same authors generalized the kneading theory to the Lozi-like maps, \cite{MS2}. Although in \cite{MS} the authors work only with the Lozi maps, they avoid explicit use of the formulas defining those maps and take the geometric approach. That allows us to follow the main ideas in \cite{MS} to prove the pruning front conjecture for the H\'enon maps and to develop the kneading theory for them in this paper. However, since there are some major differences between the families of the H\'enon and Lozi maps, we have to develop some new methods and tools to address these differences. 
	
	One of the main distinctions in developing symbolic dynamics for the H\'enon and Lozi maps is the following: The idea of using the partition into the left and right half-planes as the base for the symbolic dynamics for the Lozi maps is natural and has been used for a long time. On the other hand, for the H\'enon maps, as P.\ Cvitanovi\'c noticed in \cite{CGP}, ``a partition of the plane by the $y$-axis does not determine the periodic points uniquely. Numerically it appears that a line joining the homoclinic tangencies does in fact partition the plane as needed. [...] we conjecture that this line partitions space so that the unstable periodic points are uniquely described by symbolic sequences.'' 
	
	The first question is the existence of a such line that contains the mentioned homoclinic tangencies, that is, the set of all critical points, that we denote by $\C$. Note that in \cite{BC} the authors write:
	``One might for example believe (although this is doubtful), that $\C$ is located on a smooth curve cutting through the [unstable manifold of the fixed point] $W^u$ [...]. We can however only prove this for a $C^{l/2}$-curve and this is based on the fact that $W^u$ never intersects itself.'' Since we could not find any such proof, for completeness we provide in Section \ref{sec:cl} a proof of the existence of an arc $\K$ which contains the critical set and call it the \emph{critical locus}. 
	
	In Section \ref{sec:coa}, we use $\K$ to partition a forward invariant disc $D$, that contains the H\'enon attractor, into the left and right regions as the base for the symbolic dynamics for the H\'enon maps. Note that in \cite[Theorem 1.6]{WY}, it is proved that the critical set partitions the attractor into finitely many disjoint sets (two for the H\'enon maps), and that there exists a subshift of a full shift on finitely many symbols (two for the H\'enon maps) which provides coding of orbits on the attractor that is given by the partition. In \cite[Corollary 1.1]{WY}, it is mentioned that kneading sequences are the itineraries of images of the critical points. However, any relationship between kneading sequences and the itineraries of the other points is not studied. Our method in Section \ref{sec:coa} gives an alternative proof of the existence of coding.
	
	\subsubsection{Turning points and the kneading set}
	The question now is what should be the pruning front and/or what should replace the kneading sequence of unimodal maps? In our approach, the set of images of the basic critical points $F(\mfc)$ has a key role. We denote by $\mfc$ the set of all critical points that lie on $W^u$, $\mfc = \C \cap W^u$, and for simplicity, we write $F = F_{a,b}$, assuming $(a,b) \in \WY$. We call the points of $\mfc$ the \emph{basic critical points}, the points of $F(\mfc)$ the \emph{turning points}, and the itineraries of the turning points we call the \emph{kneading sequences}. By $\mfk_F$ we denote the set of all kneading sequences and we call it the {\it kneading set}. Note that $\mfk_F$ is a countable infinite set. Every kneading sequence, denoted by $\bar k$, has the following form: $$\bar k = \lpinf w \pm \cdot \ra k_{\h 0},$$ where $\lpinf = \dots + + +$, $w = w_0 \dots w_m$, for some $m \in \N_0$, $\ra k_{\h 0} = k_0 k_1 k_2 \dots$, $w_0 = -$, $k_0 = +$, $w_i, k_j \in \{ -, + \}$ for $i = 1, \dots m$ and $j \in \N$, for $\pm$ we can substitute any of $+$ and $-$, and the dot shows where the 0th coordinate is. In Section \ref{sec:pfc} we prove the pruning front conjecture by proving that $\mfk_F$ characterizes all itineraries. The proof is given in two steps, by two theorems.
	
	First, we prove that $\mfk_F$ characterizes all itineraries of points of the unstable manifold $W^u$ of the fixed point $X$ (what we call $W^u$-admissibility). Itineraries of all points of $W^u$ start with $\lpinf = \dots +++$. Moreover, $W^u$ is invariant for $F$, so the set of all $W^u$-admissible sequences is invariant for the shift map $\sigma$. This means that apart from the sequence that consists of all $+$s (which is  $W^u$-admissible, because it is the itinerary of $X$), we only need a way of checking $W^u$-admissibility of sequences of the form $\lpinf \cdot p_0p_1p_2 \dots = \lpinf \cdot \ra p_{\h 0}$, such that $p_0=-$. This tool is given by the following theorem:
	
	\begin{theorem}\label{admis}
		A sequence $\lpinf \cdot \ra p_{\h 0}$, such that $p_0 = -$, is $W^u$-admissible if and only if for every kneading sequence $\lpinf w \pm \cdot\ra k_{\h 0}$, such that $w = p_0p_1 \dots p_m$ for some $m$, we have $\sigma^{m+2}(\ra p_{\h 0})\preceq \ra k_{\h 0}$, where $\preceq$ is the parity-lexicographical ordering.
	\end{theorem}
	Now that we know which sequences are  $W^u$-admissible, and since the symbolic space is equipped with the product topology, the following theorem holds.
	\begin{theorem}\label{esadmis}
		A sequence $\bar p = \dots p_{-2} p_{-1} \cdot p_0 p_1 \dots$ is admissible if and only if for every positive integer $n$ there is a $W^u$-admissible sequence $\bar q =  \dots q_{-2} q_{-1} \cdot q_0 q_1 \dots$ such that $p_{-n} \dots p_n = q_{-n} \dots q_n$. 
	\end{theorem}
	Note that the kneading set $\mfk_F$ is indeed the pruning front and that Theorems \ref{admis} and \ref{esadmis} prove that all the other disallowed regions of the symbol plane are obtained by backward and forward iterations of the primary pruned region, which is the statement of the pruning front conjecture in \cite{CGP}. At the end of Section \ref{sec:pfc} we prove that the equality of two kneading sets implies topological conjugacy of the corresponding H\'enon maps.
	\begin{theorem}\label{oneway}
		Within $\WY$, the H\'enon maps $F_1$ and $F_2$ are topologically conjugate on their strange attractors if their sets of kneading sequences coincide; i.e., $\mfk_{F_1} = \mfk_{F_2}$.
	\end{theorem}
	
	The H\'enon maps are maps on the plane that fold the plane back into itself exactly once, and as such, they could be considered as a two-dimensional analogue of the unimodal maps of an interval. To characterize itineraries of all points for a unimodal map we need only one sequence of two symbols -  the kneading sequence - and as we have already mentioned, this sequence is an invariant of the topological conjugacy classes of the unimodal maps with negative Schwartzian derivative and without periodic attractors. The question is: Is there such a good sequence of two symbols for the H\'enon  maps?
	
	\subsubsection{The folding pattern and pruned trees}
	In Section \ref{sec:fp} we introduce a new notion for the H\'enon maps, where the whole information about the symbolic system is maximally compressed. This new notion we call the \emph{folding pattern}. It is a bi-infinite sequence of two symbols and, as we prove in this paper, it is an invariant of the topological conjugacy classes of the H\'enon maps within the Wang-Young parameter set, and it characterizes the set of itineraries of such H\'enon maps; see Theorem \ref{equivalent}. The same holds also for the Lozi maps within the Misiurewicz set of parameters (the folding pattern is introduced in \cite{MS} for the Lozi maps). The folding patterns are introduced for the first time in \cite{S} by the second author, for the classification of the inverse limit spaces of tent maps with finite critical orbit. They also played a key role in the proof of the Ingram conjecture (topological classification of the inverse limit spaces of tent maps) by M.\ Barge, H.\ Bruin, and the second author, \cite{BBS}.  
	
	In addition, in Section \ref{sec:fp} we construct pruned trees for the H\'enon maps, mentioned in \cite{CGP}, and prove that they are equivalent to the sets of kneading sequences and to the folding patterns, that is, given one of them, we can recover the other two, see Theorem \ref{equivalent-fp-pt}.
	
	\subsubsection{The classification of H\'enon attractors}    
	In Section \ref{sec:cla} we classify (up to topological conjugacy) the H\'enon maps on their strange attractors within Wang-Young's set of parameters.
	
	Our proof exploits the notion of mild dissipation introduced by Crovisier and Pujals \cite{CP} in  2017, where they showed that for any $a \in (1,2)$ and  $b \in (-\frac{1}{4},0) \cup (0,\frac{1}{4})$ the H\'enon map $F_{a,b}$ is mildly dissipative on the surface $\D = \{ (x, y) : |x| < 1/2 + 1/a, |y| < 1/2 - a/4 \}$ and has a 1-dimensional structure in the following sense: using 1-dimensional pieces of stable manifolds of $\mu$-almost every point in the surface of dissipation, for any ergodic measure $\mu$ not supported on a hyperbolic sink, a reduced 1-dimensional dynamics is obtained as a continuous non-invertible map acting on a tree. Our classification result relies on an extension of that result from our recent work in \cite{BS}, where we showed that for each orientation reversing H\'enon map $F$ there exists a densely branching tree $\T$, such that $F$ on its strange attractor is conjugate to the shift homeomorphism on the inverse limit of $\T$. In the current paper, we first argue that the same result holds for $b < 0$, and for that case, we slightly modify the invariant disc $D$ to include the other fixed point of $F$, denoted by $Y$, to the boundary of $D$. Then we isolate two special sets $E_X$ and $B_X$, which are subsets of the set of all endpoints and all branch points of $\T$, respectively. We show that the aforementioned conjugacy determines: (1) For $b > 0$, a bijection between the elements of $E_X \cup B_X$ and the orbits of basic critical points; (2) For $b < 0$, a bijection between the elements of $E_X \cup B_X$ and the orbits of quasi-critical points, that are the points of $\K \cap W^u_Y$. That, together with Theorem \ref{oneway} and the results from Section \ref{sec:fp}, leads to the classification result in terms of the kneading sequences and the folding pattern.
	
	\begin{theorem}\label{thm:iff}
		Let $F_i : D_i \to D_i$, $i = 1,2$, be two H\'enon maps with parameters in $\WY$. If $F_1$ and $F_2$ are topologically conjugate, then their sets of kneading sequences coincide, and their folding patterns coincide. Conversely, $F_1$ and $F_2$ are topologically conjugate on their strange attractors if their sets of kneading sequences coincide, or equivalently if their folding patterns coincide.  
	\end{theorem}
	
	\subsubsection{Generalizations to the piecewise affine case}
	In Section \ref{sec:lm} we discuss how the same approach can be used to classify the Lozi maps for the Misiurewicz set of parameters, by proving that two such Lozi maps are conjugate if and only if their sets of kneading sequences coincide, if and only if their folding patterns coincide. 
	
	\smallskip
	{\bf Acknowledgments.} We are grateful to Michael Benedicks and Liviana Palmisano for making us aware of the work of Wang and Young and pointing out the necessity of proving the existence of the critical locus $\K$, as well as helpful feedback during the first author's visit to KTH Stockholm in December of 2022. We thank Sylvain Crovisier whose suggestion to look at the notion of mild dissipation in \cite{CP} started our work on H\'enon maps in \cite{BS}. We thank P. Boyland, A. de Carvalho, and T. Hall for their comments during Simons Semester in Warsaw ``Topological, smooth and holomorphic dynamics, ergodic theory, fractals'' in the Spring of 2023. We are also thankful to Marco Martens for his feedback on folding patterns and inverse limits during his visit in Krak\'ow in May of 2023.

	\section{Preliminaries}\label{sec:prelim}
	
	\subsection{Geometry of critical regions and critical set: a recap}\label{ss:prelim1}
	
	In 2001, Wang and Young proved in \cite{WY}, among many results, a few theorems that we use in our construction of the arc $\K$, which contains the set of critical points $\C$, and which is given in the next section. Therefore, for completeness and for readers' convenience, we include these results here.
	\begin{theorem}[{\cite[Corollary 1.3 \& Appendix A]{WY}}]\label{thm:WY}
		Let $q_{a}$ be a quadratic map, $q_{a}(x) = 1 - ax^2$. Let $I$, $I_0$ be closed intervals such that $q_a(I) \subset \Int I_0 \subset I_0 \subset \Int I$, and let $J_1$, $J_2$ be the two components of $I \setminus I_0$. Let $b_0 << |J_1|, |J_2|$. For every $a^* \in [1.5, 2]$ for which $q_{a^*}$ is a Misiurewicz map the following holds:
		\begin{enumerate}[$(I)$]
			\item \cite[Theorem 1.8 (i)]{WY} There exist $\kappa > 0$ and a rectangle $[a_0, a_1] \times (0, b_0]$ arbitrarily near $(a^*, 0)$ such that for each $(a, b) \in [a_0, a_1] \times (0, b_0]$, $F_{a,b}$ maps $R = I \times [-\kappa b, \kappa b]$ strictly into $I_0 \times [-\kappa b, \kappa b]$, definig an attractor $\Lambda_{F_{a,b}} = \bigcap_{n \in \N_0} F_{a,b}^nR$. (To ensure that $q_a(I) \subset \Int I$ for some $I$ in the case $a^* = 2$, consider $a$ slightly less than $2$.)
			
			Let now replace $\partial I \times [-\kappa b, \kappa b]$ by two curves $\omega_1$ and $\omega_2$ so that each $\omega_i \subset J_i \times [-\kappa b, \kappa b]$ joins the top and bottom boundaries of $R$, and lies on the stable manifold of a periodic orbit. Let $R_0$ be the subregion of $R$ bounded by $\omega_1$ and $\omega_2$. Let $\partial_h R_0$ denote the horizontal boundary of $R_0$. Let $R_k = F_{a,b}^k(R_0)$ and $\partial_h R_k = F_{a,b}^k(\partial_h R_0)$. Note that the left and right boundaries shrink exponentially as we iterate.  
			
			The constants $\kappa, \alpha, \delta, c > 0$, $0 < \rho < 1$, below are system constants, and $b << \alpha, \delta, \rho, e^{-c}$ for all $(a, b) \in [a_0, a_1] \times (0, b_0]$.
			
			\item \cite[Theorem 1.8 (ii)]{WY} There is a positive measure set $\WY \subset [a_0, a_1] \times (0, b_0]$ such that for all $(a, b) \in \WY$ and for $F_{a,b}|_{R_0}$ the following holds:
			\begin{enumerate}[$(1)$]
				\item \cite[Theorem 1.1 (1)]{WY} There is a Cantor set $\C \subset \Lambda_{F_{a,b}}$ called the \emph{critical set} given by $\C = \bigcap_{k \in \N_0} \C^{(k)}$, where the $\C^{(k)}$ are a decreasing sequence of neighborhoods of $\C$ called critical regions. More precisely,
				\begin{enumerate}[$(i)$]
					\item $\C^{(0)} = \{ (x, y) \in R_0 : d(x, c^*) < \delta \}$, where $c^*$ is the critical point of $q_{a^*}$.
					\item $\C^{(k)}$ has a finite number of components called $Q^{(k)}$ each one of which is diffeomorphic to a rectangle. The boundary of $Q^{(k)}$ is made up of two $C^2(b)$ segments of $\partial_h R_k$ connected by two vertical lines. The horizontal boundaries are $\approx \min (2\delta, \rho^k)$ in length, and the Hausdorff distance between them is $\O (b^{k/2})$.
					\item $\C^{(k)}$ is related to $\C^{(k-1)}$ as follows: $Q^{(k-1)} \cap R_k$ has at most two components, each one of which lies between two $C^2(b)$ segments of $\partial_h R_k$ that stretch across $Q^{(k-1)}$. Each component of $Q^{(k-1)} \cap R_k$ contains exactly one component of $\C^{(k)}$.
				\end{enumerate}
				\item \cite[Theorem 1.1 (2)]{WY} On each horizontal boundary $\eta$ of each component $Q^{(k)}$ of $\C^{(k)}$, $k \in \N_0$, there is a unique point $z$ characterized by the following two properties:
				\begin{enumerate}[$(i)$]
					\item $\| DF_{a,b}^j(z)({0 \atop 1}) \| \ge \kappa^{-1}e^{cj}$ for all $j \in \N$.
					\item If $\tau$ is a unit tangent vector to $\eta$ at $z$, then $\| DF_{a,b}^n(z) \tau \| < (\kappa b)^n$ for all $n \in \N$.
				\end{enumerate}
				Let $\H$ be the set of all these points, and let $d_{\C}(\cdot)$ be the notion of ``distance to the critical set'' defined below. Then $z \in \H$ also satisfies
				\begin{enumerate}
					\item[(iii)] $d_{\C}(z^j) \ge \kappa^{-1}e^{-\alpha j}$ for all $j \in \N$, where $z^j = F_{a,b}^j(z)$.
				\end{enumerate}
				Finally, since the critical set $\C$ is the accumulation set of $\H$, properties $(i)$ and $(iii)$ of $\H$ are passed on to $\C$.
			\end{enumerate}
			The set $\WY$ has the property that for all sufficiently small $b$, the set $\{ a : (a,b) \in \WY \}$ has positive 1-dimensional Lebesgue measure.  
			These results are valid for both $b > 0$ and $b < 0$.
		\end{enumerate}
	\end{theorem}
	\noindent
	{\bf Distance to the critical set.} For $z \in R_0$, $d_{\C}(z)$ is defined as follows: For $z \in \C^{(0)} \setminus \C$, let $k$ be the largest number with $z \in \C^{(k)}$, and let $d_{\C}(z)$ be the horizontal distance between $z$ and the midpoint of the component of $\C^{(k)}$ containing $z$. For $z \nin \C^{(0)}$, use  $\C^{(0)}$.
	
	\begin{rem}
		It is assumed that periodic orbits, such that the curves $\omega_1$, $\omega_2$ above lie on their stable manifolds, stay outside of $\C^{(0)}$ (see \cite[Appendix A]{WY}).
		
		A curve in $R_0$ is called a $C^2(b)$-curve if the slopes of its tangent vectors are $\O(b)$ and its curvature is everywhere $\O(b)$ (see \cite[Subsection 1.2.]{WY}).
	\end{rem}

	From now on we denote by $F$ a H\'enon map with a pair of parameters in $\WY$, and by $\Lambda_F$  the strange attractor of $F$. Recall that $F$ has two fixed points: the one that lies in the strange attractor $\Lambda_F$ we denote by $X$, and the other one, that does not lie in $\Lambda_F$, we denote by $Y$. The unstable and stable manifolds of $X$ we denote by $W^u$ and $W^s$ respectively, and in general, the unstable and stable manifolds of a point $P$, we denote by $W^u_P$ and $W^s_P$ respectively, if they exist. It is convenient to distinguish two cases, depending on whether the map $F$ is orientation preserving, $b < 0$, or orientation reversing, $b > 0$ (the latter case is somewhat better known, see for example \cite{BC} and \cite{BenedicksYoung93}).
	
	Let $D$ be a closed disc. For $b > 0$, $D$ is defined (and denoted) as in \cite[Lemma 4.4, Figure 2]{BC}, or as in \cite[Proposition 4.1, Figure 5]{MV}, see Figure \ref{fig:D}. The boundary of $D$ consists of an arc of $W^u$, denoted by $\partial^u D$, and a straight line segment, denoted by $l$, so $\partial D = \partial^u D \cup l$. For $b < 0$, $D$ is defined (and denoted) as in \cite[Proposition 4.1, Figure 4]{MV}. The boundary of $D$ consists of an arc of $W^u_Y$, also denoted by $\partial^u D$, and a straight line segment, denoted by $l$, so again $\partial D = \partial^u D \cup l$, see Figure \ref{fig:D2}. The disc $D$ is forward invariant, $F(D) \subset D$, it contains a fixed point $X$ and $\Cl W^u \subset D$, \cite[Proposition 4.1]{MV}. Also $\Lambda_F = \Cl W^u \subset D$; that is proved in \cite[Proposition 4.1]{MV} for $b > 0$, and follows by \cite[Proposition 4.1]{MV} and \cite[(TB) on p.\ 432]{BenedicksViana} for $b < 0$. 
	
	\begin{figure}[h]
		\resizebox{1\textwidth}{!}{
			\begin{tikzpicture}
				
				\fill[green!5!white] (-6.65,3.43)--(-4.74,-3.72) -- (-6.65,3.43) .. controls (10.75,0.9) and (10.75,-0.7) .. (-4.74,-3.75) -- cycle;
				
				\tikzstyle{every node}=[draw, circle, fill=white, minimum size=2pt, inner sep=0pt]
				\tikzstyle{dot}=[circle, fill=white, minimum size=0pt, inner sep=0pt, outer sep=-1pt]
				\node (n1) at (5*4/3,0) {};
				\node (n2) at (-6.65,5*2/3)  {};
				\node (n3) at (-4.74,-3.72)  {};
				\node (n4) at (5*5/9,0)  {};
				\node (n5) at (-4.59,-1.76) {};
				\node (n6) at (-3.98,-2.1)  {};
				\node (n7) at (-1.61,2.08)  {};
				\node (n8) at (5*61/51,0)  {};
				\node (n9) at (5*65/75,0)  {};
				
				\node (n10) at (5*4/9,1.87)  {};
				
				\node (n21) at (0.04,2.32)  {};
				\node (n22) at (-0.05,1.06)  {};
				\node (n23) at (0.05,-0.74)  {};
				\node (n24) at (-0.05,-1.28)  {};
				\node (n25) at (0.05,-2.08)  {};
				\node (n26) at (0,-2.74)  {};
				
				\node (n27) at (0,1.87)  {};
				\node (n28) at (-0.01,1.46)  {};
				
				\draw (-6.65,5*2/3)--(-4.74,-3.72);
				
				\draw (3,1.7) .. controls (7.7,0.7) and  (7.7,-0.6).. (4,-1.7);	
				\draw (3,1.7) .. controls (-9,4.25) and (-9,3.7) .. (-1.5,1.5);
				\draw (0,-2.1) .. controls (-6.9,-4.1) and (-6.9,-4.6) .. (4,-1.7);
				\draw (2.8,1.2) .. controls (7.44,0) and (7.44,-0.1) .. (0,-2.1);
				\draw (2.8,1.2) .. controls (-3,2.7) and (-3,2.2) .. (2.3,0.8);
				\draw (2.3,0.8) .. controls (5.63,-0.2) and (5.63,0.1) .. (-5*19/24+0.5,-5*5/12);
				\draw (-1.5,1.5) .. controls (4.55,-0.4) and (4.55,0.3) .. (-5*5/6+0.5,-5*1/3+0.15);
				\draw (-5*5/6+0.5,-5*1/3+0.15) .. controls (-5*5/6-0.7,-5*1/3-0.08) and (-5*5/6-0.7,-5*1/3-0.18) .. (-5*5/6+0.2,-5*1/3-0.05);
				\draw (-5*19/24+0.2,-5*5/12+0.1) .. controls (-5*19/24-0.15,-5*5/12) and (-5*19/24-0.15,-5*5/12-0.1) .. (-5*19/24+0.5,-5*5/12);
				
				\node[dot, draw=none, label=above: $z^1_0$] at (6.9,0) {};
				\node[dot, draw=none, label=above: $z_0$] at (0,2.8) {};
				\node[dot, draw=none, label=above: $z^2_0$] at (-6.8,5*2/3) {};
				\node[dot, draw=none, label=above: $z_1$] at (0,1.05) {};
				\node[dot, draw=none, label=above: $z^1_{-1}$] at (5*5/9,0) {};
				\node[dot, draw=none, label=above: $z_2$] at (0,-0.2) {};
				\node[dot, draw=none, label=above: $z^4_0$] at (-4.7,-1.05) {};
				\node[dot, draw=none, label=above: $z^5_0$] at (-5*19/24,-5*5/12) {};
				\node[dot, draw=none, label=above: $z_{-5}$] at (0,-1.2) {};
				\node[dot, draw=none, label=above: $z_{-3}$] at (0,2.4) {};
				\node[dot, draw=none, label=above: $z_{-4}$] at (0,1.6) {};
				\node[dot, draw=none, label=above: $z^1_2$] at (5*65/75,0) {};
				\node[dot, draw=none, label=above: $z^2_{-1}$] at (-2,2.57) {};
				\node[dot, draw=none, label=above: $z^1_1$] at (6.1,0) {};
				\node[dot, draw=none, label=above: $z^3_0$] at (-4.8,-3.7) {};
				\node[dot, draw=none, label=above: $z_{-1}$] at (0,-2.7) {};
				\node[dot, draw=none, label=above: $z_{-2}$] at (0,-1.53) {};
				\node[dot, draw=none, label=above: $X$] at (2.25,2.5) {};
				\node[dot, draw=none, label=above: $l$] at (-6.1,0.5) {};
				\node[dot, draw=none, label=above: $W^u$] at (5,2) {};
				\node[dot, draw=none, label=above: $D$] at (-3,0.5) {};
				\node[dot, draw=none, label=above: $\partial^uD$] at (5,-1.3) {};
				
			\end{tikzpicture}
		}
		\caption{Disc $D$ for $b > 0$ and several basic critical and post-critical points.}
		\label{fig:D}
	\end{figure}
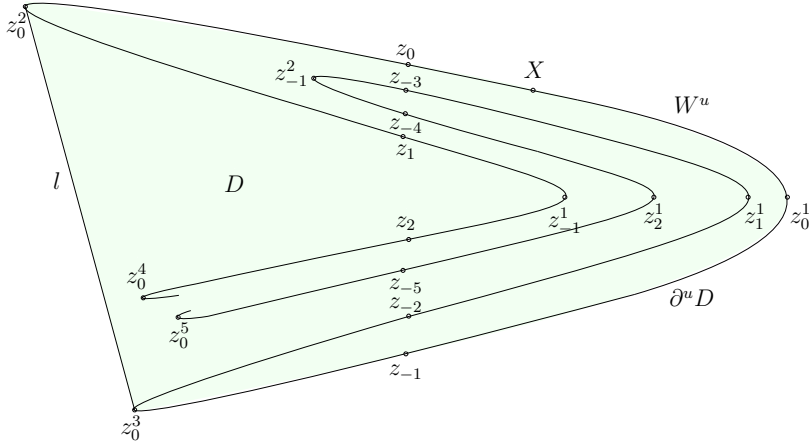
	
	In this paper, $\N$ denotes the set of all positive integers and $\N_0$ is the set of all nonnegative integers. An arc is always a homeomorphic copy of $[0, 1]$, that is, an arc does not intersect itself. For two points $Q_1$, $Q_2$, we denote by $\bar{Q_1Q_2}$ a straight line segment with the endpoints $Q_1$ and $Q_2$, and by  $[Q_1, Q_2]$ an arc with the endpoints $Q_1$ and $Q_2$. For $Q_1, Q_2 \in W^u$ we denote by $[Q_1, Q_2]^u \subset W^u$ the arc in $W^u$ that connects $Q_1$ and $Q_2$, and in general, for $Q_1, Q_2 \in W^u_P$ we denote by $[Q_1, Q_2]^u_P \subset W^u_P$ the arc in $W^u_P$ that connects $Q_1$ and $Q_2$. Analogously, $[Q_1, Q_2]^s \subset W^s$ and $[Q_1, Q_2]^s_P \subset W^s_P$. If $A = [P, Q]$ is an arc, we denote $\mathring{A} = (P, Q) = [P, Q] \setminus \{ P, Q \}$, and analogously $[P, Q)$ and $(P, Q]$, and sometimes we also call them arcs. For a point $P$, we let $P^i = F^i(P)$ for $i \in \Z$, and thus $P = P^0$. By $d(P, Q)$ we denote the (Euclidean) distance between points $P$ and $Q$ and by $\length [P, Q]$ the length of the arc $[P, Q]$. For a set $U$, by $\Int U$, $\Cl U$, $\partial U$ we denote the interior, closure and boundary of $U$ respectively.
	
	\begin{figure}[h]
		
		\resizebox{1\textwidth}{!}{
			\begin{tikzpicture}
				
				\hspace{1.5cm}
				
				\fill[green!5!white] (-6,6.75)--(0,6)--(0,6) .. controls (15,4) and  (15,-2.8).. (0,-6)--(0,-6) .. controls (-6,-7) and  (-6,-5.75).. (-6,-5.2)-- cycle;		
				
				\tikzstyle{every node}=[draw, circle, fill=white, minimum size=2pt, inner sep=0pt]
				\tikzstyle{dot}=[circle, fill=white, minimum size=0pt, inner sep=0pt, outer sep=-1pt]
				
				\node[label=right: $z_2$] at (0.05,3.81)  {};
				\node[label=right: $z_3$] at (-0.01,3.52)  {};
				\node[label=right: $z_{6}$] at (0.04,2.72)  {};
				\node[label=right: $z_{-1}$] at (0.04,2.32)  {};
				\node[label=right: $z_{-2}$] at (0,1.55)  {};
				\node[label=right: $z_{7}$] at (0,1.18)  {};					
				
				\node[label=right: $z_{-3}$] at (0.05,-1.91)  {};
				\node[label=right: $z_0$] at (0,-2.8)  {};
				\node[label=right: $z_5$] at (0.05,-3.34)  {};
				\node[label=right: $z_4$] at (-0.05,-4.29)  {};
				\node[label=right: $z_1$] at (0,-4.77)  {};
				\node[label=right: $z_{-4}$] at (0,-5.57)  {};							
				
				\node[label=right: $X$] at (1.41,-2.4)  {}; 
				\node[label=below: $Y$] at (-7,6.88)  {}; 		
				
				\node[label=left: $z'_0$] at (0,6)  {};	
				\node[label=left: $z'_8$] at (0,4.05)  {};	
				\node[label=left: $z'_7$] at (0,3.3)  {};
				\node[label=left: $z'_4$] at (0,2.9)  {};
				\node[label=left: $z'_{11}$] at (0,2.17)  {};	
				\node[label=left: $z'_{12}$] at (0,1.7)  {};	
				\node[label=left: $z'_3$] at (0,0.9)  {};
				\node[label=left: $z'_2$] at (0,-1.25)  {};
				\node[label=left: $z'_{13}$] at (0,-2.1)  {};	
				\node[label=left: $z'_{10}$] at (0,-2.6)  {};	
				\node[label=left: $z'_5$] at (0,-3.6)  {};	
				\node[label=left: $z'_6$] at (0,-4.1)  {};	
				\node[label=left: $z'_9$] at (0,-5)  {};	
				\node[label=left: $z'_{14}$] at (0,-5.4)  {};	
				\node[label=left: $z'_1$] at (0,-6)  {};	
				
				\node[dot, draw=none, label=above: $l$] at (-6.3,0) {};
				\node[dot, draw=none, label=above: $W^u_Y$] at (7,5.2) {};
				\node[dot, draw=none, label=below right: $W^u$] at (2,-3.2) {};
				\node[dot, draw=none, label=above: $D$] at (-3,0) {};
				\node[dot, draw=none, label=above: $\partial^uD$] at (7,-3.9) {};
				
				\draw[OliveGreen,thick] (2,4.5) .. controls (1,4.7) and (1,4.3) .. (2,4.1);								
				
				\draw[OliveGreen,thick] (0,4.05) .. controls (-1.2,4.2) and (-1.2,3.7) .. (0,3.3);	
				\draw[blue] (2.45,3.3) .. controls (-1.5,4.3) and (-1.5,3.8) .. (2.4,2.8);	
				
				\draw[OliveGreen,thick] (0,0.9) .. controls (-4.8,2.1) and (-4.8,3.8) .. (0,2.9);	
				\draw[blue] (2.3,2.23) .. controls (-4.8,4) and (-4.8,2.1) .. (2,0.7);
				\draw[blue] (2.2,1.87) .. controls (-4,3.3) and (-4,2.4) .. (2.1,1.05);
				\draw[OliveGreen,thick] (0,2.17) .. controls (-2.4,2.6) and (-2.4,2.35) .. (0,1.7);

				\draw[OliveGreen,thick] (0,-1.25) .. controls (-8,-3.9) and (-8,-7.7) .. (0,-6);	
				\draw[blue] (0.72,-5.4) .. controls (-7.5,-7.4) and (-7.5,-3.9) .. (1.63,-1.45);
				\draw[OliveGreen,thick] (0,-2.1) .. controls (-6.8,-4.4) and (-6.8,-6.8) .. (0,-5.4);
				\draw[blue] (1,-4.5) .. controls (-6,-6.5) and (-6,-4.4) .. (1.4,-2.4);
				\draw[OliveGreen,thick] (0,-2.6) .. controls (-6,-4.5) and (-6,-6.4) .. (0,-5);
				\draw[blue] (1,-4) .. controls (-4.5,-5.6) and (-4.5,-4.5) .. (1.3,-3);		
				\draw[OliveGreen,thick] (0,-3.6) .. controls (-3.5,-4.6) and (-3.5,-5) .. (0,-4.1);

				\draw[OliveGreen,thick] (0,0.9) .. controls (4,-0.1) and (4,-0.1) .. (0,-1.25);
				\draw[blue] (2,0.7) .. controls (4.2,0.1) and (4.2,-0.1) .. (1.69,-1);
				\draw[blue] (2.1,1.05) .. controls (5.2,0.3) and (5.2,-0.3) .. (1.63,-1.45);
				\draw[OliveGreen,thick] (0,1.7) .. controls (6.3,0.3) and (6.3,-0.3) .. (0,-2.1);
				\draw[blue] (2.2,1.87) .. controls (7.5,0.5) and  (7.5,-0.5).. (1.4,-2.4);
				\draw[OliveGreen,thick] (0,2.17) .. controls (7.7,0.5) and  (7.7,-0.3).. (0,-2.6);
				\draw[blue] (2.3,2.23) .. controls (9,0.5) and  (9,-0.5).. (1.3,-3);
				\draw[OliveGreen,thick] (0,2.9) .. controls (10,0.8) and  (10,-0.7).. (0,-3.6);	
				\draw[OliveGreen,thick] (0,3.3) .. controls (11,0.4) and  (11,-0.8).. (0,-4.1);	
				\draw[blue] (2.4,2.8) .. controls (10.8,0.5) and  (10.8,-1).. (1,-4);
				\draw[blue] (2.45,3.3) .. controls (12,0.9) and  (12.4,-1.2).. (1,-4.5);
				\draw[OliveGreen,thick] (0,4.05) .. controls (13.2,1.32) and  (13.2,-1.4).. (0,-5);
				\draw[OliveGreen,thick] (2,4.1) .. controls (13.5,1.6) and  (13.5,-1.9).. (0,-5.4);
				\draw[OliveGreen,thick] (0,6) .. controls (15,4) and  (15,-2.8).. (0,-6);
				
				\draw[OliveGreen,thick](-8,7)--(0,6);	
				\draw[OliveGreen,thick](2,4.5)--(4,4.05);	 
				
				\draw[blue](1.69,-1)--(0.7,-1.3);
				\draw[blue](0.72,-5.4)--(1.8,-5.1);								
				
				\draw (-6,6.75)--(-6,-5.2);
				
			\end{tikzpicture}
			
		}	
		\vspace{-0.5cm}
		\caption{Disc $D$ for $b < 0$ and several basic critical and quasi-critical points.}
		\label{fig:D2}
		
	\end{figure}
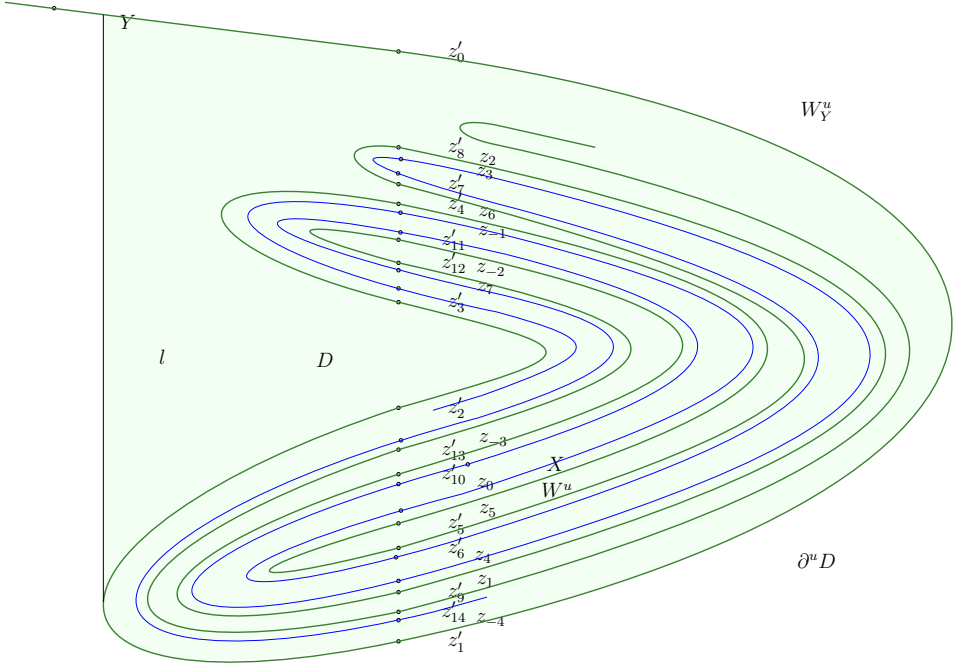
	
	Recall that $\mfc = \C \cap W^u$ and we call the points of $\mfc$ the basic critical points. The set of basic critical points is countable and we index them in the following way: Let $z_0$ be the first basic critical point on $W^u$ on the left of the fixed point $X$ (and $X$ is between $z_0$ and $z_0^1$). Let $\phi : \mathbb{R} \to W^u$ be a parametrization of $W^u$ such that $\phi(0) = X$ and $z_0 \in \phi((-\infty, 0))$. Denote $W^{u-} = \phi((-\infty, 0])$ and $W^{u+} = \phi([0, \infty ))$. Now we index all the other basic critical points such that $z_k \in W^{u-}$ and $z_{-k} \in W^{u+}$ for all $k \in \N$, and $z_i$ and $z_j$ are consecutive (there are no other basic critical points in $W^u$ between them) if and only if $|i - j| = 1$ for all $i, j \in \Z$. The points in $\bigcup_{n \in \N} F^n(\C)$ are called the post-critical points, in $\mfc^+ = \bigcup_{n \in \N} F^n(\mfc)$ the basic post-critical points, and in $\mfc^- = \bigcup_{n \in \N} F^{-n}(\mfc)$ the basic pre-critical points. 
	There is a natural order $\le$ on $W^u$ given by $$P < Q \textrm{ if and only if }\phi^{-1}(P) < \phi^{-1}(Q).$$ In particular, $z_0 < X$. In general, on any arc $[Q_1, Q_2]$ we can put an ordering $\le$ defined as follows: If $P_1, P_2 \in [Q_1, Q_2]$, then $P_1 \le P_2$ if and only if $\length[Q_1, P_1] \le \length[Q_1, P_2]$.

	\section{The critical locus}\label{sec:cl}
	
	In this section, we show that there exists an arc $\K$ such that $\Lambda_F \cap \K = \C$.

	{\bf (1) Modification of $R_0$ for the orientation preserving case.} Let $b < 0$. Let $R$ be as in \cite[Theorem 1.8 (i)]{WY} (see also Theorem \ref{thm:WY} (I)). The region $R_0$ is defined in \cite[Theorem 1.8 (i)]{WY} as a subregion of $R$ bounded by two curves $\omega_1$ and $\omega_2$. We modify slightly $R_0$, while still keeping all its properties as in \cite[Theorems 1.8 and 1.1]{WY} ($R_0$ will be similar to the disc $D$, only with an additional requirement on the line $l$): Let $\omega_1$ be as in \cite[Appendix A.1.]{WY}, that is, it joins the top and bottom boundary of $R$ and lies on the stable manifold of a periodic orbit. Let additionally $\omega_1 \cap W^u_Y$ consist of at least two and at most three points, $P_1, P_2$ (and $P_3) \in \omega_1$, $P_1$ above $P_2$ (above $P_3$). It is possible to choose such $\omega_1$ since the stable manifolds of the periodic points are dense in $D$, and for every point $Q_1 \in W^u_Y$ there exist $\eps > 0$ and a point $Q_2 \in W^u_Y$, $\length [Q_1, Q_2]^u_Y = \eps$, such that the $\eps$-neighborhood of $[Y, Q_1]^u$ does not intersect the attractor and the part of $W^u_Y$ that starts at $Q_2$ and accumulates on the attractor. Let $R_0$ be a subregion of $R$ bounded by the arcs $[P_1, P_2] \subset \omega_1$ and $[P_1, P_2]^u_Y \subset W^u_Y$. We let $\partial_h R_0 = [P_1, P_2]^u_Y$ and call it the horizontal boundary of $R_0$. Let also $R_k = F^k(R_0)$ and $\partial_h R_k = F^k(\partial_h R_0)$. Obviously, $R_0$ is forward invariant and $\Lambda_F = \bigcap_{k \in \N_0} R_k$.
	
	With this modification we obtain that each horizontal boundary $\eta$ of each component $Q^{(k)}$ of $\C^{(k)}$, $k \in \N_0$ (see Theorem \ref{thm:WY} (II) (2)), lies in $W^u_Y$, $\eta \subset W^u_Y$, and consequently $\H \subset W^u_Y$, where $\H$ is as in Theorem \ref{thm:WY} (II)  (2). We call the points of $\H$ the \emph{quasi-critical points}. The set of quasi-critical points is countable infinite and we index them in the following way: Let $z'_0$ be the first quasi-critical point on $W^u_Y$ on the right of the fixed point $Y$. Let $\psi : \mathbb{R} \to W^u_Y$ be a parametrization of $W^u_Y$ such that $\psi(0) = Y$ and $z'_0 \in \phi((0, \infty))$. Denote $W^{u+} = \psi([0, \infty))$ and note that $\H \subset W^{u+}_Y$. Now we index all the other quasi-critical points $z'_k \in \H$, $k \in \N_0$, such that $z'_i$ and $z'_j$ are consecutive (there are no other quasi-critical points in $W^u_Y$ between them) if and only if $|i - j| = 1$ for all $i, j \in \N_0$. The points in $\H^+ = \bigcup_{n \in \N} F^n(\H)$ are called the quasi-post-critical points, and in $\H^- = \bigcup_{n \in \N} F^{-n}(\H)$ the quasi-pre-critical points.
	
	For convenience, we distinguish the upper and lower horizontal boundary of $R_k$, $k \in \N_0$: $\partial_h R_0 = [P_1, {z'_0}^1]^u_Y \cup [{z'_0}^1, P_2]^u_Y$, where $[P_1, {z'_0}^1]^u_Y$ is the upper horizontal boundary of $R_0$ and $[{z'_0}^1, P_2]^u_Y$ is the lower horizontal boundary of $R_0$, and analogously for $\partial_h R_k = F^k(\partial_h R_0)$.
	
	{\bf (2) Open sets $U_k$.} Let $b > 0$. For every $k \in \N$, $k \ge 4$, let $Y_k \in W^u_Y$ be the point such that $\alpha_k = [Y, Y_k]^u_Y$ is the longest arc with the property that for every $P \in \mathring\alpha_k$, $d(P, z_0^k) > d(Y_k, z_0^k)$. Note that $Y_k \in D$ and $\alpha_k \subset \alpha_{k+1}$. For every $k \ge 4$, let us denote by $\beta_k$ the arc with endpoints $z_0^k$ and $l_k \in l$ that is ``parallel'' with $\alpha_k$; i.e. for every $P \in \beta_k$, $d(P, \alpha_k) = d(z_0^k, Y_k)$. Let $l_2, l_3$ be the endpoints of $l$, where $l_2$ lies in the second quadrant, and $l_3$ lies in the third quadrant.
	
	For $b < 0$, for every $k \ge 3$ we construct the points $l_k \in l$ and the arcs $\beta_k$ similarly, although not identically, since in this case points in  $W^u$ are not accessible from the complement of the attractor. Let $l_1, l_2$ be the endpoints of $l$, where $l_1$ lies in the second quadrant, and $l_2$ lies in the third quadrant. We pick $l_3\in \bar{l_1l_2}$ and let $\beta_3$ be an arc with endpoints $z_0^3$ and $l_3$. We continue inductively, having defined $\beta_k$ and $l_k$, we choose $l_{k+1}\in\bar{l_1l_k}$ and let $\beta_{k+1}$ be an arc connecting $l_{k+1}$ and $z_0^{k+1}$, such that $\beta_{k+1}\cap \beta_{k-i}=\emptyset$ for all $i=0,...,k-3$, and $\beta_{k+1}\cap W^u_Y=\{z_0^{k+1}\}$ (one can make the arcs $\beta_k$ ``parallel'' to the appropriate initial segments of $W^{u+}_Y$), see Figure \ref{fig:connectors2}. 
	
	\begin{figure}[h]
		
		\resizebox{1.2\textwidth}{!}{
			\begin{tikzpicture}
				
				\fill[teal!10!white] 
				(-6,3.1)--(-3.58,2.8)--(-3.58,2.8) .. controls (-3.4,1.68) and (-1,1.2) .. (0,0.9)--(0,0.9) .. controls (4,-0.1) and (4,-0.1) .. (0,-1.25)--(0,-1.25) .. controls (-2.2,-2.1) and (-5.5,-3.15) .. (-6,-5.2)--(-6,-5.2)--(-6,3.1)-- cycle;	
				
				\fill[blue!5!white]	
				(-6,3.1)--(-3.58,2.8)--(-3.58,2.8) .. controls (-3.3,3.3) and (-1,3.2) .. (0,2.9)--(0,2.9) .. controls (10,0.8) and  (10,-0.7).. (0,-3.6)--(0,-3.6) .. controls (-3.5,-4.6) and (-3.5,-5) .. (0,-4.1)--(0,-4.1) .. controls (11,-0.8) and  (11,0.4).. (0,3.3)--(0,3.3) .. controls (-0.3,3.35) and (-0.8,3.5) .. (-0.9,3.9)--(-0.9,3.9)--(-6,4.6)-- cycle;	
				
				\fill[pink!30!white]
				(-6,5.48)--(1.23,4.48)--(1.23,4.48) .. controls (1.3,4.4) and (1,4.3) .. (2,4.1)--(2,4.1) .. controls (13.5,1.6) and  (13.5,-1.9).. (0,-5.4)--(0,-5.4) .. controls (-6.8,-6.8) and (-6.8,-4.4) .. (0,-2.1)--(0,-2.1) .. controls (6.3,-0.3) and (6.3,0.3) .. (0,1.7)--(0,1.7) .. controls (-2.4,2.35) and (-2.4,2.6) .. (0,2.17)--(0,2.17) .. controls (7.7,0.5) and  (7.7,-0.3).. (0,-2.6)--(0,-2.6) .. controls (-6,-4.5) and (-6,-6.4) .. (0,-5)--(0,-5) .. controls (13.2,-1.4) and  (13.2,1.32).. (0,4.05)--(0,4.05) .. controls (-0.3,4.1) and (-0.7,4.18) .. (-0.9,3.9)--(-0.9,3.9)--(-6,4.6)--cycle;				
				
				\tikzstyle{every node}=[draw, circle, fill=white, minimum size=2pt, inner sep=0pt]
				\tikzstyle{dot}=[circle, fill=white, minimum size=0pt, inner sep=0pt, outer sep=-1pt]
				
				\node[label=right: $z_2$] at (0.05,3.81)  {};
				\node[label=right: $z_3$] at (-0.01,3.52)  {};
				\node[label=right: $z_{6}$] at (0.04,2.72)  {};
				\node[label=right: $z_{-1}$] at (0.04,2.32)  {};
				\node[label=right: $z_{-2}$] at (0,1.55)  {};
				\node[label=right: $z_{7}$] at (0,1.18)  {};					
				
				\node[label=right: $z_{-3}$] at (0.05,-1.91)  {};
				\node[label=right: $z_0$] at (0,-2.8)  {};
				\node[label=right: $z_5$] at (0.05,-3.34)  {};
				\node[label=right: $z_4$] at (-0.05,-4.29)  {};
				\node[label=right: $z_1$] at (0,-4.77)  {};
				\node[label=right: $z_{-4}$] at (0,-5.57)  {};					
				
				\node[label=right: $X$] at (1.41,-2.4)  {}; 
				\node[label=below: $Y$] at (-7,6.88)  {}; 		
				\node[label=left: $z'_0$] at (0,6)  {};	
				\node[label=left: $z'_8$] at (0,4.05)  {};	
				\node[label=left: $z'_7$] at (0,3.3)  {};
				\node[label=left: $z'_4$] at (0,2.9)  {};
				\node[label=left: $z'_{11}$] at (0,2.17)  {};	
				\node[label=left: $z'_{12}$] at (0,1.7)  {};	
				\node[label=left: $z'_3$] at (0,0.9)  {};
				\node[label=left: $z'_2$] at (0,-1.25)  {};
				\node[label=left: $z'_{13}$] at (0,-2.1)  {};	
				\node[label=left: $z'_{10}$] at (0,-2.6)  {};	
				\node[label=left: $z'_5$] at (0,-3.6)  {};	
				\node[label=left: $z'_6$] at (0,-4.1)  {};	
				\node[label=left: $z'_9$] at (0,-5)  {};	
				\node[label=left: $z'_{14}$] at (0,-5.4)  {};	
				\node[label=left: $z'_1$] at (0,-6)  {};	
				
				\node[label=below left: $l_1$] at (-6,6.75)  {};
				\node[label=left: $l_2$] at (-6,-5.2)  {};
				\node[label=below left: $l_3$] at (-6,3.1)  {};
				\node[label=below left: $l_4$] at (-6,4.6)  {};
				\node[label=below left: $l_5$] at (-6,5.48)  {};
				\node[label=right: $z^{'2}_0$] at (-5.98,-5.4)  {};
				\node[label=below left: $z^{'3}_0$] at (-3.58,2.8)  {};
				\node[label=below left: $z^{'4}_0$] at (-0.9,3.9)  {};
				\node[label=below left: $z^{'5}_0$] at (1.26,4.48)  {};	
				
				\node[label=right: $z^{'1}_0$] at (11.25,0.5)  {};
				\node[label=left: $z^{'1}_1$] at (3,-0.1)  {};
				\node[label=right: $z^{'2}_1$] at (-2.62,-4.56)  {};
				\node[label=left: $z^{'3}_1$] at (-1.79,2.35)  {};	
				
				\node[dot, draw=none, label=above: $l$] at (-6.3,0) {};
				\node[dot, draw=none, label=above: $W^u_Y$] at (7,5.2) {};
				\node[dot, draw=none, label=below right: $W^u$] at (2,-3.2) {};
				\node[dot, draw=none, label=above: $U_1$] at (-3,0) {};
				\node[dot, draw=none, label=above: $U_2$] at (7.83,0) {};
				\node[dot, draw=none, label=above: $U_3$] at (10.1,0) {};
				\node[dot, draw=none, label=above: $\partial^uD$] at (7,-3.9) {};
				
				\draw[TealBlue,thick] (0,6)--(0,4.05);
				\draw[TealBlue,thick] (0,3.3)--(0,2.9);
				\draw[TealBlue,thick] (0,2.17)--(0,1.7);
				\draw[TealBlue,thick] (0,0.9)--(0,-1.25);
				\draw[TealBlue,thick] (0,-2.1)--(0,-2.6);
				\draw[TealBlue,thick] (0,-3.6)--(0,-4.1);
				\draw[TealBlue,thick] (0,-5)--(0,-5.4);
				
				\draw[OliveGreen,thick] (2,4.5) .. controls (1,4.7) and (1,4.3) .. (2,4.1);				
				
				\draw[OliveGreen,thick] (0,4.05) .. controls (-1.2,4.2) and (-1.2,3.7) .. (0,3.3);	
				\draw[blue] (2.45,3.3) .. controls (-1.5,4.3) and (-1.5,3.8) .. (2.4,2.8);	
				
				\draw[OliveGreen,thick] (0,0.9) .. controls (-4.8,2.1) and (-4.8,3.8) .. (0,2.9);	
				\draw[blue] (2.3,2.23) .. controls (-4.8,4) and (-4.8,2.1) .. (2,0.7);
				\draw[blue] (2.2,1.87) .. controls (-4,3.3) and (-4,2.4) .. (2.1,1.05);
				\draw[OliveGreen,thick] (0,2.17) .. controls (-2.4,2.6) and (-2.4,2.35) .. (0,1.7);		
				
				\draw[OliveGreen,thick] (0,-1.25) .. controls (-8,-3.9) and (-8,-7.7) .. (0,-6);	
				\draw[blue] (0.72,-5.4) .. controls (-7.5,-7.4) and (-7.5,-3.9) .. (1.63,-1.45);
				\draw[OliveGreen,thick] (0,-2.1) .. controls (-6.8,-4.4) and (-6.8,-6.8) .. (0,-5.4);
				\draw[blue] (1,-4.5) .. controls (-6,-6.5) and (-6,-4.4) .. (1.4,-2.4);
				\draw[OliveGreen,thick] (0,-2.6) .. controls (-6,-4.5) and (-6,-6.4) .. (0,-5);
				\draw[blue] (1,-4) .. controls (-4.5,-5.6) and (-4.5,-4.5) .. (1.3,-3);		
				\draw[OliveGreen,thick] (0,-3.6) .. controls (-3.5,-4.6) and (-3.5,-5) .. (0,-4.1);
				
				\draw[OliveGreen,thick] (0,0.9) .. controls (4,-0.1) and (4,-0.1) .. (0,-1.25);
				\draw[blue] (2,0.7) .. controls (4.2,0.1) and (4.2,-0.1) .. (1.69,-1);
				\draw[blue] (2.1,1.05) .. controls (5.2,0.3) and (5.2,-0.3) .. (1.63,-1.45);
				\draw[OliveGreen,thick] (0,1.7) .. controls (6.3,0.3) and (6.3,-0.3) .. (0,-2.1);
				\draw[blue] (2.2,1.87) .. controls (7.5,0.5) and  (7.5,-0.5).. (1.4,-2.4);
				\draw[OliveGreen,thick] (0,2.17) .. controls (7.7,0.5) and  (7.7,-0.3).. (0,-2.6);
				\draw[blue] (2.3,2.23) .. controls (9,0.5) and  (9,-0.5).. (1.3,-3);
				\draw[OliveGreen,thick] (0,2.9) .. controls (10,0.8) and  (10,-0.7).. (0,-3.6);	
				\draw[OliveGreen,thick] (0,3.3) .. controls (11,0.4) and  (11,-0.8).. (0,-4.1);	
				\draw[blue] (2.4,2.8) .. controls (10.8,0.5) and  (10.8,-1).. (1,-4);
				\draw[blue] (2.45,3.3) .. controls (12,0.9) and  (12.4,-1.2).. (1,-4.5);
				\draw[OliveGreen,thick] (0,4.05) .. controls (13.2,1.32) and  (13.2,-1.4).. (0,-5);
				\draw[OliveGreen,thick] (2,4.1) .. controls (13.5,1.6) and  (13.5,-1.9).. (0,-5.4);
				\draw[OliveGreen,thick] (0,6) .. controls (15,4) and  (15,-2.8).. (0,-6);
				
				\draw[OliveGreen,thick](-8,7)--(0,6);	
				\draw[OliveGreen,thick](2,4.5)--(4,4.05);	 
				
				\draw[blue](1.69,-1)--(0.7,-1.3);
				\draw[blue](0.72,-5.4)--(1.8,-5.1);
				
				\draw (-6,6.75)--(-6,-5.2);
				
				\draw (-6,5.48)--(1.23,4.48);
				\draw (-6,4.6)--(-0.9,3.9);
				\draw (-6,3.1)--(-3.58,2.8);
				
			\end{tikzpicture}
			
		}	
		
		\caption{For $b < 0$ connectors of the critical locus $\K$ are in teal.}
		\label{fig:connectors2}
		
	\end{figure}
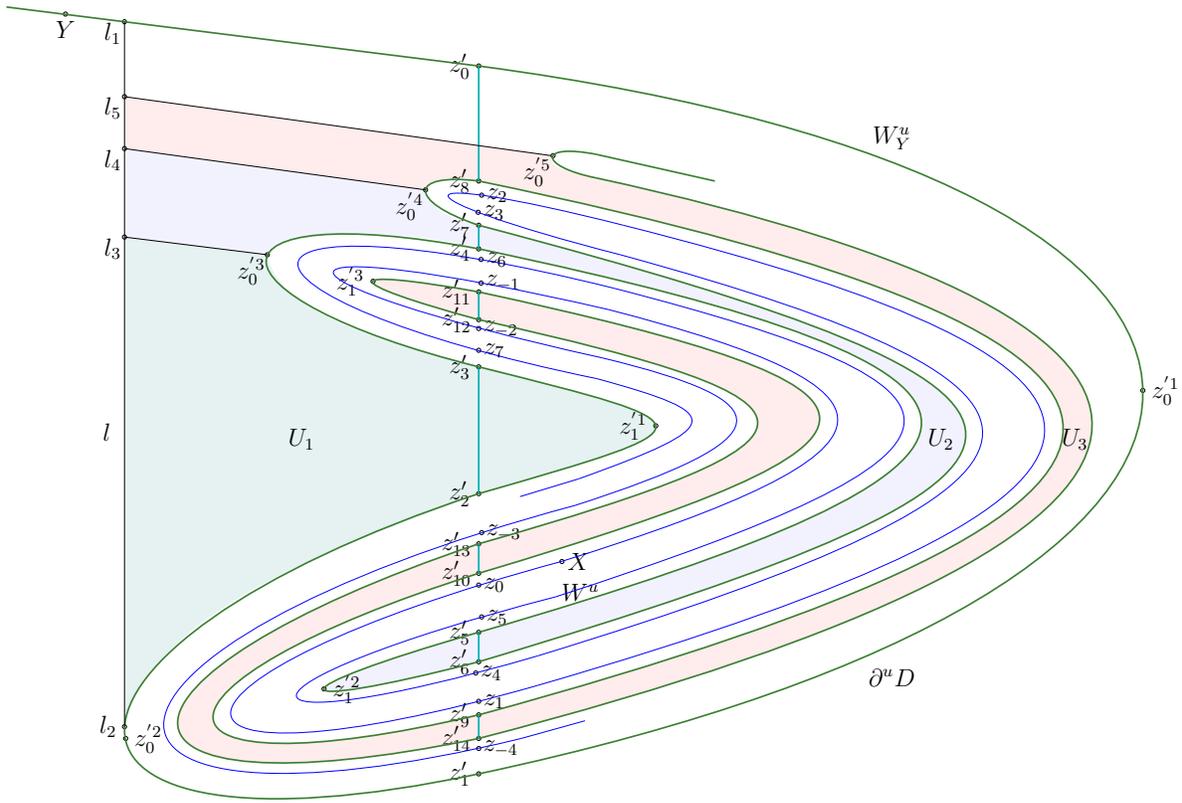
	
	We define the open regions $U_k$, $k \in \N$, as follows:
	
	Case 1. Let $b > 0$. The region $U_1$ is bounded by the arcs $\partial^u U_1 = [l_2, z_0^4]^u$, $\beta_4$ and $\bar{l_2l_4}$, $U_2$ is bounded by $\partial^u U_2 = [l_3, z_0^5]^u$, $\beta_5$ and $\bar{l_3l_5}$, and for every $k \in \N$, $k \ge 3$, $U_k$ is bounded by the arcs $\partial^u U_k = [z_0^{k+1}, z_0^{k+3}]^u$, $\beta_{k+1}$, $\beta_{k+3}$ and $\bar{l_{k+1}l_{k+3}}$. 
	
	Case 2. Let $b < 0$. The region $U_1$ is bounded by the arcs $\partial^u U_1 = [l_2, {z'_0}^3]^u_Y$, $\beta_3$ and $\bar{l_2l_3}$, for every $k \in \N$, $k \ge 2$, $U_k$ is bounded by the arcs $\partial^u U_k = [{z'_0}^{k+1}, {z'_0}^{k+2}]^u_Y$, $\beta_{k+1}$, $\beta_{k+2}$ and $\bar{l_{k+1}l_{k+2}}$. 
	
	It is easy to see that the open sets $U_k$, $k \in \N$, satisfy the following properties:
	\begin{enumerate}[(i)]
		\item $F(\partial^u U_k) = \partial^u U_{k+1}$ for every $k \ge 3$ if $b > 0$, and for every $k \ge 2$ if $b < 0$.
		\item For $b > 0$, $\partial^u U_k \cap \partial^u U_{k+2} = \{ z_0^{k+3} \}$ and $\partial^u U_i \cap \partial^u U_j = \emptyset$ if $|i - j| \notin\{0, 2\}$.\\
		For $b < 0$, $\partial^u U_k \cap \partial^u U_{k+1} = \{ {z'_0}^{k+2} \}$ and $\partial^u U_i \cap \partial^u U_j = \emptyset$ if $|i - j| \notin\{0, 1\}$.	
		\item $\partial^u D \cup \bigcup_{i \in \N} \partial^u U_i = W^u$ if $b > 0$, and $\partial^u D \cup \bigcup_{i \in \N} \partial^u U_i = W^{u+}_Y\setminus [Y,l_1)^u_Y$ if $b < 0$.
		\item $U_i \cap U_j = \emptyset$ for $i \ne j$,
		\item $U_i \cap (\Lambda_F \cup W^u_Y) = \emptyset$ for every $i \in \N$.
		\item $U_0 = \bigcup_{i \in \N} U_i$ is an open set such that $\Cl U_0 = D$.
	\end{enumerate} 
	
	\begin{figure}[h!]
		
		\resizebox{1\textwidth}{!}{
			\hspace*{-1.8cm}
			\begin{tikzpicture}
				
				\fill[teal!10!white]
				(-6.65,5*2/3)--(-5.25,-1.9) -- (-5.25,-1.9)--(-4.59,-1.76) -- (-4.59,-1.76) -- (-5*5/6+0.5,-5*1/3+0.15) -- (-1.5,1.5) .. controls (4.55,-0.4) and (4.55,0.2) .. (-5*5/6+0.5,-5*1/3+0.15) --(-1.5,1.5) .. controls (-6,5*2/3-0.5) .. (-6.65,5*2/3) -- cycle;	
				\fill[blue!5!white]
				(2.8,1.2) .. controls (-6.7,3.6) and (-6.7,3.2) .. (2.3,0.9) -- cycle;	 
				\fill[blue!5!white] 
				(0,-2.1) .. controls (7.3,-0.1) and (7.3,0) .. (2.8,1.2) -- (2.3,0.9) .. controls (6,-0.1) and (6,-0.2) .. (-5*19/24+0.5,-2.5) -- cycle;				 	
				\fill[blue!5!white] 
				(-4.75,-3.71).. controls (-4.65,-3.5)..(0,-2.1)--(-5*19/24+0.5,-2.5)--(-4.97,-2.85) -- cycle;
				\fill[pink!30!white]	
				(3,1.55) .. controls (7.5,0.5) and  (7.5,-0.4).. (4,-1.5) -- (4,-1.23) .. controls (7.1,-0.3) and (7.1,0.4) .. (3,1.33) -- cycle;					 
				\fill[pink!30!white]
				(4,-1.23) .. controls (-4.26,-3.4) and (-4.0,-3.7) .. (4,-1.5) -- cycle;					 
				\fill[pink!30!white]
				(3,1.33) .. controls (-6.7,3.8) and (-6.7,3.1) .. (-1.5,1.75) -- (-1.5,1.6) .. controls (-8,3.5) and (-8,4) .. (3,1.55) -- cycle;			
				\fill[pink!30!white]
				(-5*19/24+1,-1.9) .. controls (6,0.2) and (6,-0.3) ..(-1.5,1.75) -- (-1.5,1.6) .. controls (5.0,-0.3) and (5.0,0.1) .. (-5*5/6+0.2,-5*1/3-0.05) -- cycle;					 
				\fill[pink!30!white]  (-5*5/6+0.2,-5*1/3-0.05)..controls (-5*5/6-0.02,-5*1/3-0.05)..(-5.25,-1.9) -- (-5.1,-2.43) -- (-5*19/24+1,-1.9) -- cycle;
				
				\tikzstyle{every node}=[draw, circle, fill=white, minimum size=2pt, inner sep=0pt]
				\tikzstyle{dot}=[circle, fill=white, minimum size=0pt, inner sep=0pt, outer sep=-1pt]
				\node (n1) at (5*4/3,0) {};
				\node (n2) at (-6.65,5*2/3)  {};
				\node (n3) at (-4.75,-3.71)  {};
				\node (n4) at (5*5/9+0.01,0)  {};
				\node (n5) at (-4.59,-1.76) {};
				\node (n6) at (-3.32,-2.03)  {};
				\node (n7) at (-4.38,2.81)  {};
				\node (n8) at (5*61/51-0.1,0)  {};
				\node (n9) at (5*65/75+0.26,0)  {};
				
				\node (n10) at (5*4/9,1.87)  {};
				
				\node (n21) at (0.04,2.32)  {};
				\node (n22) at (-0.05,1.06)  {};
				\node (n23) at (0.05,-0.74)  {};
				\node (n24) at (-0.05,-1.24)  {};
				\node (n25) at (0.05,-2.08)  {};
				\node (n26) at (0,-2.74)  {};
				\node (n27) at (-2.16,-3.02)  {};
				
				\node (n30) at (0.04,1.88) {};
				\node (n31) at (-0.05,1.52)  {};
				\node (n32) at (0.04,-0.92)  {};
				\node (n33) at (-0.05,-1.67)  {};
				\node (n34) at (0.04,-2.32)  {};
				\node (n35) at (0,-2.57)  {};
				\node (n36) at (0.04,2.2)  {};
				\node (n37) at (0,2.07)  {};
				\node (n38) at (0.01,1.35)  {};
				\node (n39) at (0,1.17)  {};
				
				\node (n40) at (-3.96,-2.59)  {};
				
				\node (n41) at (-6,-2.9)  {};
				
				\node (n42) at (-4.39,-2.5)  {};
				\node (n43) at (-4.01,-2.41)  {};
				\node (n44) at (-3.27,-2.23)  {};

				\draw[TealBlue,thick](0.04,1.88)--(-0.05,1.52);
				\draw[TealBlue,thick](-0.05,1.06)--(0.05,-0.74);
				\draw[TealBlue,thick](0.05,-0.91)--(-0.05,-1.25);
				\draw[TealBlue,thick](-0.05,-1.67)--(0.05,-2.08);
				\draw[TealBlue,thick](0.05,-2.31)--(0,-2.57);
				\draw[TealBlue,thick](n36)--(n37);
				\draw[TealBlue,thick](n38)--(n39);
				
				\draw[OliveGreen,thick](-6.5,-3.02)--(-0.5,-1.55);
				\draw[Gray](-6.65,5*2/3)--(-4.75,-3.71);
				\draw[Gray](-5.25,-1.9)--(-4.59,-1.76);			\draw[Gray](-5.1,-2.43)--(-3.32,-2.03);
				\draw[Gray](-4.97,-2.85)--(-3.96,-2.59);
				
				\draw (3,1.7) .. controls (7.7,0.7) and  (7.7,-0.6).. (4,-1.7);	
				\draw (3,1.7) .. controls (-9,4.25) and (-9,3.7) .. (-1.5,1.5);
				
				\draw (0,-2.1) .. controls (-6.9,-4.1) and (-6.9,-4.6) .. (4,-1.7);
				\draw (2.8,1.2) .. controls (7.3,0) and (7.3,-0.1) .. (0,-2.1);
				\draw (2.8,1.2) .. controls (-6.7,3.6) and (-6.7,3.2) .. (2.3,0.9);
				\draw (-1.5,1.5) .. controls (4.55,-0.4) and (4.55,0.3) .. (-5*5/6+0.5,-5*1/3+0.15);
				\draw (-5*5/6+0.5,-5*1/3+0.15) .. controls (-5*5/6-0.7,-5*1/3-0.08) and (-5*5/6-0.7,-5*1/3-0.18) .. (-5*5/6+0.2,-5*1/3-0.05);
				\draw (-5*19/24+1,-5*5/12+0.18) .. controls (-5*19/24+0.5,-5*5/12+0.05) and (-5*19/24+0.5,-5*5/12) .. (-5*19/24+1,-5*5/12+0.08);
				
				\draw (2.3,0.9) .. controls (6,-0.1) and (6,-0.2) .. (-5*19/24+0.5,-2.5);
				\draw (-1.5,1.75) .. controls (6,-0.3) and (6,0.2) .. (-5*19/24+1,-1.9);
				\draw (-1.5,1.6) .. controls (5.0,-0.3) and (5.0,0.1) .. (-5*5/6+0.2,-5*1/3-0.05);
				\draw (-5*19/24+0.3,-5*5/12+0.05-0.41) .. controls (-5*19/24-0.15,-5*5/12-0.5) and (-5*19/24-0.15,-5*5/12-0.1-0.5) .. (-5*19/24+0.5,-5*5/12-0.41);
				
				\draw (3,1.55) .. controls (-8,4.0) and (-8,3.5) .. (-1.5,1.6);
				\draw (3,1.33) .. controls (-6.7,3.8) and (-6.7,3.1) .. (-1.5,1.75);
				\draw (3,1.55) .. controls (7.5,0.5) and  (7.5,-0.4).. (4,-1.5);	
				\draw (3,1.33) .. controls (7.1,0.3) and  (7.1,-0.2).. (4,-1.25);	
				\draw (4,-1.25) .. controls (-4.23,-3.43) and (-4.23,-3.7) .. (4,-1.5);	
				
				\node[dot, draw=none, label=above: $z^1_0$] at (6.9,0) {};
				\node[dot, draw=none, label=above: $z_0$] at (0,2.8) {};
				\node[dot, draw=none, label=above: $z^2_0$] at (-6.8,5*2/3) {};
				\node[dot, draw=none, label=above: $z_1$] at (0,1.05) {};
				\node[dot, draw=none, label=above: $z^1_{-1}$] at (5*5/9,0) {};
				\node[dot, draw=none, label=above: $z_2$] at (0,-0.2) {};
				\node[dot, draw=none, label=above: $z^4_0$] at (-4.7,-1.05) {};
				\node[dot, draw=none, label=above: $z^5_0$] at (-4,-2.55) {};
				\node[dot, draw=none, label=above: $z^6_0$] at (-3.6,-1.7) {};
				\node[dot, draw=none, label=above: $z_{-3}$] at (0,2.4) {};
				\node[dot, draw=none, label=above: $z_{-4}$] at (0,1.7) {};
				\node[dot, draw=none, label=above: $z_{-5}$] at (0,-1.2) {};
				\node[dot, draw=none, label=above: $z^1_2$] at (5*65/75,0) {};
				\node[dot, draw=none, label=above: $z^2_{-1}$] at (-3.8,3.2) {};
				\node[dot, draw=none, label=above: $z^1_1$] at (6,0) {};
				\node[dot, draw=none, label=above: $z^3_0$] at (-4.8,-3.7) {};
				\node[dot, draw=none, label=above: $z^3_{-1}$] at(-2.45,-2.68) {};
				\node[dot, draw=none, label=above: $z_{-1}$] at (0,-2.7) {};
				\node[dot, draw=none, label=above: $z_{-2}$] at (0,-1.6) {};
				\node[dot, draw=none, label=above: $X$] at (2.25,2.5) {};
				\node[dot, draw=none, label=above: $Y$] at (-6,-3) {};
				\node[dot, draw=none, label=above: $W^u_Y$] at (-6.5,-2.2) {};
				\node[dot, draw=none, label=above: $l$] at (-6,0) {};
				\node[dot, draw=none, label=above: $W^u$] at (5,2) {};
				\node[dot, draw=none, label=above: $D$] at (5,-1.5) {};
				\node[dot, draw=none, label=above: $U_1$] at (-2,0.3) {};
				\node[dot, draw=none, label=above: $U_3$] at (3.4,0.3) {};
				\node[dot, draw=none, label=above: $U_2$] at (5.2,0.3) {};
				
			\end{tikzpicture}
			
		}
		\caption{For $b > 0$ connectors of the critical locus $\K$ are in teal.}
		\label{fig:connectors}
	\end{figure}
	
	{\bf (3) Regions $\hat Q^{(k)}$ for the orientation reversing case.} Let $k \in \N$ and $Q^{(k)}$ be a component of a critical region $C^{(k)}$ from \cite[Theorem 1.1]{WY}. Recall, every $Q^{(k)}$ is diffeomorphic to a rectangle and the boundary of $Q^{(k)}$ is made up of two segments of $\partial_h R_k$ connected by two vertical lines, and $\Lambda_F \subset R_k$. By the choice of $\delta$, $\rho$ and $\alpha$, $Q^{(k)}$ does not contain any basic post-turning point of level less than $k + k_0$, for some large $k_0$ (see \cite[Subsection 3.2 \& Inductive Assumption (IA2) in Subsection 3.3]{WY}). Here, $k_0 > 6$ is fine. Also, $Q^{(k-1)} \cap R_k$ has at most two components, each one of which lies between two segments of $\partial_h R_k$ that stretch across $Q^{(k-1)}$. Each component of $Q^{(k-1)} \cap R_k$ contains exactly one component of $\C^{(k)}$. Without loss of generality we consider only $Q^{(k)}$ with $\Lambda_F \cap Q^{(k)} \ne \emptyset$.
	
	Let us now slightly modify regions $Q^{(k)}$. We define $\hat Q^{(k)} \subset Q^{(k)}$ such that the vertical boundary of $Q^{(k)}$ and $\hat Q^{(k)}$ coincides, and the horizontal boundary of $\hat Q^{(k)}$ consists of two segments of $W^u$ such that $\hat Q^{(k)} \cap \Lambda_F = Q^{(k)} \cap \Lambda_F$. We denote this horizontal boundary by $\partial^u \hat Q^{(k)}$. 
	
	Let us prove that $\hat Q^{(k)}$ is well defined. 
	
	(i) For $k = 0$, $Q^{(0)} = \C^{(0)}$, the closest component (in the Hausdorff metric) of $Q^{(0)} \cap \Lambda_F$ to the upper horizontal boundary of $Q^{(0)}$ is the arc $\theta_0 = Q^{(0)} \cap [z_0^2, z_0^1]^u$, and the closest component of $Q^{(0)} \cap \Lambda_F$ to the lower horizontal boundary of $Q^{(0)}$ is the arc $\hat \theta_0 = Q^{(0)} \cap [z_0^3, z_0^1]^u$. Therefore, if we let $\partial^u \hat Q^{(0)} = \theta_0 \cup \hat \theta_0$, then $\hat Q^{(0)} \cap \Lambda_F = Q^{(0)} \cap \Lambda_F$ as required. Note that $z_0 \in \theta_0$ and $z_{-1} \in \hat \theta_0$. 
	
	(ii) Let $k \in \N$. If $Q^{(k-1)}$ contains two components of $\C^{(k)}$ and $Q^{(k)}$ is any of them, then the two components of the horizontal boundary of $Q^{(k)}$ are segments of different components of the horizontal boundary of $R_k$. There are two cases to consider.
	
	(a) First, let us suppose that both arcs, $[z_0^{k+2}, z_0^{k+1}]^u$ and $[z_0^{k+3}, z_0^{k+1}]^u$, stretch across $Q^{(k)}$. Let $\partial^u \hat Q^{(k)} \subseteq \theta_k \cup \hat \theta_k$, where $\theta_k  = Q^{(k-1)} \cap [z_0^{k+2}, z_0^{k+1}]^u$ and $\hat \theta_k = Q^{(k-1)} \cap [z_0^{k+3}, z_0^{k+1}]^u$. Then $\hat Q^{(k)} \cap \Lambda_F = Q^{(k)} \cap \Lambda_F$ as required. Note that each of $\theta_k, \hat \theta_k$ contains a basic critical point. Since $[z_0^{k+2}, z_0^{k+1}]^u = [z_0^{k+2}, z_0^{k}]^u \cup [z_0^{k}, z_0^{k+1}]^u$, the basic critical point in $\theta_k$ is visited in the previous step when we considered the components $Q^{(k-1)}$ of $\C^{(k-1)}$. The basic critical point in $\hat \theta_k$ is not visited in any previous step.
	
	(b) Second, let us suppose that the arcs $[z_0^{k+2}, z_0^{k+1}]^u$ and $\beta_{k+3}$ stretch across $Q^{(k)}$, but the arc $[z_0^{k+3}, z_0^{k+1}]^u$ does not intersect $Q^{(k)}$ at all (if it intersected $Q^{(k)}$, it would have stretched across $Q^{(k)}$, and this is the case (a)). Let $\theta_k  = Q^{(k-1)} \cap [z_0^{k+2}, z_0^{k+1}]^u$, $\hat \theta_k = Q^{(k-1)} \cap [z_0^{k+2}, z_0^{k+4}]^u$ and $\partial^u \hat Q^{(k)} \subseteq \theta_k \cup \hat \theta_k$. Then $\hat Q^{(k)} \cap \Lambda_F = Q^{(k)} \cap \Lambda_F$ as required. Note that both $\theta_k$ and $\hat \theta_k$ contain a basic critical point, and the basic critical point in $\hat \theta_k$ is not visited in any previous step.
	
	(iii) Now let us assume that $Q^{(k-1)}$ contains only one component of $\C^{(k)}$ and denote it by $Q^{(k)}$. 
	
	(a) If the two components of the horizontal boundary of $Q^{(k)}$ are segments of distinct components of the horizontal boundary of $R_k$, then both arcs, $[z_0^{k+2}, z_0^{k+1}]^u$, $[z_0^{k+3}, z_0^{k+1}]^u$, stretch across $Q^{(k)}$, so let $\theta_k  = Q^{(k-1)} \cap [z_0^{k+2}, z_0^{k+1}]^u$, $\hat \theta_k = Q^{(k-1)} \cap [z_0^{k+3}, z_0^{k+1}]^u$ and $\partial^u \hat Q^{(k)} \subseteq \theta_k \cup \hat \theta_k$. Then $\hat Q^{(k)} \cap \Lambda_F = Q^{(k)} \cap \Lambda_F$ as required. Since $Q^{(k)}$ is the only component of $\C^{(k)}$ contained in $Q^{(k-1)}$, we have that $\beta_{k+3}$ stretches across $Q^{(k-1)}$ and hence $\partial^u \hat Q^{(k)} \subseteq \partial^u \hat Q^{(k-1)}$.
	
	(b) If both components of the horizontal boundary of $Q^{(k)}$ are segments of the same component of the horizontal boundary of $R_k$, then $Q^{(k-1)} \cap [z_0^{k+2}, z_0^{k+1}]^u$ contains two arc-components, denote them by $\theta_k, \hat \theta_k$, and $Q^{(k-1)} \cap [z_0^{k+3}, z_0^{k+1}]^u = \emptyset$. If we let $\partial^u \hat Q^{(k)} \subseteq \theta_k \cup \hat \theta_k$, then $\hat Q^{(k)} \cap \Lambda_F = Q^{(k)} \cap \Lambda_F$ as required. Note that $[z_0^{k+2}, z_0^{k+1}]^u = [z_0^{k+2}, z_0^{k}]^u \cup [z_0^{k}, z_0^{k+1}]^u$ and again $\partial^u \hat Q^{(k)} \subseteq \partial^u \hat Q^{(k-1)}$. 
	
	Since we have exhausted all cases, this completes the proof that $\hat Q^{(k)}$ is well defined.
	
	Recall that in the orientation preserving case, by modifying $R_0$ we obtained that each horizontal boundary $\eta$ of each component $Q^{(k)}$ of $\C^{(k)}$, $k \in \N_0$, lies in $W^u_Y$, $\eta \subset W^u_Y$, and it contains a quasi-critical point $z'$.
	
	{\bf (4) An order on $\C$.} For every pair of points $u, v \in \C$, $u \ne v$, we define that $u$ is \emph{above} $v$, or equivalently, that $v$ is \emph{below} $u$, in the following way: Let $k \in \N$ be such that $u$ and $v$ lie in different components of $\C^{(k)}$, that is, $u \in Q^{(k)}, v \in Q'^{(k)}$ and $Q^{(k)} \ne Q'^{(k)}$. Then $u$ is \emph{above} $v$ if $Q^{(k)}$ is above $Q'^{(k)}$. 
	
	Let us show that this relation is well-defined. Since every component $Q^{(k-1)}$ of $\C^{(k-1)}$ contains at most two components $Q^{(k)}$ of $\C^{(k)}$, and each one of them is contained in a component of $Q^{(k-1)} \cap R_k$ that stretches across $Q^{(k-1)}$, we know which one is above the other one (or equivalently which one is below the other one). Also, if $Q^{(k-1)}$ is above $Q'^{(k-1)}$, then every $Q^{(k)} \subset Q^{(k-1)}$ is above every $Q'^{(k)} \subset Q'^{(k-1)}$. Therefore, for every two different components $Q^{(k)}, Q'^{(k)}$ of $\C^{(k)}$ we know which one is above (or below) the other one, and the above-defined order on $\C$ is well defined.
	
	Note that in the orientation preserving case, we can define in the same way the same order on $\C \cup \H$.
	
	\begin{theorem}
		There exists an arc $\K$ such that $\Lambda_F \cap \K = \C$. If $b < 0$ then additionally $W^u_Y \cap \K = \H$. The endpoints of $\K$ are $z_0$ and $z_{-1}$ for $b > 0$, and $z'_0$ and $z'_1$ for $b < 0$.
	\end{theorem}
	\begin{proof}		
		We construct $\K$ inductively.
		
		Case 1. For $b > 0$ in the $k$th step we consider all the basic critical points of $\partial^u U_k$ and connect certain pairs of them with arcs, that we call \emph{connectors}, as follows:
		
		Let $k = 1$. Recall that $Q^{(0)} = \C^{(0)}$, $[z_0^2, z_0^3]^u = [z_0^2, z_0^1]^u \cup [z_0^1, z_0^3]^u$ and $\partial^u U_1 = [l_2, z_{-1}^1]^u \cup [z_{-1}^1, z_0^4]^u \subset [z_0^2, z_{-1}^1]^u \cup [z_{-1}^1, z_0^4]^u$. Since $[z_0^2, z_0^1]^u \cap Q^{(0)} \ne \emptyset$, $[z_0^1, z_0^3]^u \cap Q^{(0)} \ne \emptyset$, $[z_0^2, z_{-1}^1]^u \cap Q^{(0)} = \hat \theta' \ne \emptyset$ and $[z_{-1}^1, z_0^4]^u \cap Q^{(0)} = \hat \theta \ne \emptyset$, we have that $Q^{(0)}$ contains two components of $\C^{(1)}$. Also, $\mfc \cap \partial^u U_1 = \{ z_1, z_2 \}$ and $z_1 \in \hat \theta$, $z_2 \in \hat \theta'$. Since the subset of $Q^{(0)}$ bounded by $\hat \theta, \hat \theta'$ and the vertical boundary of $Q^{(0)}$ is a subset of $U_1$, and $U_1 \cap \Lambda_F = \emptyset$, there exists an arc $[z_1, z_2] \subset Q^{(0)}$ such that $(z_1, z_2) \cap \Lambda_F = \emptyset$. The arc $\zeta_1 = [z_1, z_2]$ is a connector.
		
		Let $k \in \N$, $k \ge 2$. If $Q^{(k-1)}$ contains two components of $\C^{(k)}$, $Q^{(k)}$ and $Q'^{(k)}$, then by (3) (ii) (a) \& (b) $\partial^u U_k \cap Q^{(k-1)}$ consists of two arc-components, $\hat \theta_k$ and $\hat \theta'_k$, and $U_k \cap Q^{(k-1)}$ is bounded by $\hat  \theta_k$, $\hat \theta'_k$ and vertical boundary of $Q^{(k-1)}$. Every arc-component of $\partial^u U_k \cap Q^{(k-1)}$ contains a basic critical point, $z_i \in \hat \theta_k$, $z_j \in \hat \theta'_k$. Since $U_k \cap \Lambda_F = \emptyset$, we connect them with an arc $[z_i, z_j] \subset Q^{(k-1)}$ such that $[z_i, z_j] \cap \hat \theta_k = \{ z_i \}$ and $[z_i, z_j] \cap \hat \theta'_k = \{ z_j \}$. The exact choice of such an arc $[z_i, z_j]$ is irrelevant, since $\diam \left(Q^{(k-1)}\right)\rightarrow_{k\to\infty}0$. Note that $(z_i, z_j) \cap \Lambda_F = \emptyset$. The arc $\zeta_i = [z_i, z_j]$ is a connector, see Figure \ref{fig:connectors}.
		
		Case 2. For $b < 0$ in the $k$th step we first consider all the quasi-critical points of $\partial^u U_k$ and connect certain pairs with arcs, that we also call \emph{connectors}. Similarly as in the first case, for $k = 1$, recall that $Q^{(0)} = \C^{(0)}$, $\partial^u U_1 = [l_2, {z'_1}^1]^u_Y \cup [{z'_1}^1, {z'_0}^3]^u_Y \subset [{z'_0}^2, {z'_1}^1]^u_Y \cup [{z'_1}^1, {z'_0}^3]^u_Y$. Since $[{z'_0}^2, {z'_1}^1]^u_Y \cap Q^{(0)} = \eta' \ne \emptyset$ and $[{z'_1}^1, {z'_0}^3]^u_Y \cap Q^{(0)} = \eta \ne \emptyset$, but ${z'_1}^1\notin Q^{(0)}$, we have that $Q^{(0)}$ contains two components of $\C^{(1)}$. Also, $\H \cap \partial^u U_1 = \{ z'_2, z'_3 \}$ and $z'_2 \in \eta'$, $z'_3 \in \eta$. Since the subset of $Q^{(0)}$ bounded by $\eta, \eta'$ and the vertical boundary of $Q^{(0)}$ is a subset of $U_1$, and $U_1 \cap \Lambda_F = \emptyset$, there exists an arc $[z'_2, z'_3] \subset Q^{(0)}$ such that $[z'_2, z'_3] \cap \Lambda_F = \emptyset$. The arc $\zeta_2 = [z'_2, z'_3]$ is a connector. 
		
		Again, similarly as in the first case, for $k \ge 2$ if $Q^{(k-1)}$ contains two components $Q^{(k)}$ and $Q'^{(k)}$ of $\C^{(k)}$, then $\partial^u U_k \cap Q^{(k-1)}$ consists of two arc-components, $\eta_k$ and $\eta'_k$, and $U_k \cap Q^{(k-1)}$ is bounded by $\eta_k$, $\eta'_k$ and vertical boundary of $Q^{(k-1)}$. Every arc-component of $\partial^u U_k \cap Q^{(k-1)}$ contains a quasi-critical point, say $z'_i \in \eta_k$, $z'_j \in \eta'_k$. Since $U_k \cap \Lambda_F = \emptyset$, we connect them with an arc $[z'_i, z'_j] \subset Q^{(k-1)}$ such that $[z'_i, z'_j] \cap \eta_k = \{ z'_i \}$ and $[z'_i, z'_j] \cap \eta'_k = \{ z'_j \}$. Note that $[z'_i, z'_j] \cap \Lambda_F = \emptyset$. The arc $\zeta_i = [z'_i, z'_j]$ is a connector. 
		
		Additionally, for every $z_i' $ there exist $k_i \in \N$ and $Q^{(k_i-1)}$ such that $Q^{(k_i-1)} \cap \beta_{k_i + 1} = \emptyset$ and $Q^{(k_i-1)} \cap \beta_{k_i + 2} \ne \emptyset$. In this case $Q^{(k_i-1)}$ contains only one component of $\C^{(k)}$, say $Q^{(k)}$, and an arc of $W^u_Y$ that contains the quasi-critical point $z'_i$. Since $\partial^u U_{k_i} \cap Q^{(k_i-1)}$ consists of only one arc-component, say $\eta'_{k_i}$, which contains a quasi-critical point $z'_j$, let $\eta_{k_i}$ be the arc of $W^u_Y \cap Q^{(k_i-1)}$ that contains the quasi-critical point $z'_i$. Note that $Q^{(k_i-1)} \cap \beta_{k_i + 2} \ne \emptyset$ implies $Q^{(k_i-1)} \cap \beta_{k_i + m} \ne \emptyset$ for every $m \ge 2$. Therefore, for the region $\hat U^{(k_i)}$ bounded by $\eta_{k_i}$, $\eta'_{k_i}$ and the vertical boundary of $Q^{(k_i-1)}$ we have $\hat U^{(k_i)} \cap \Lambda_F = \emptyset$. We connect $z'_i$ and $z'_j$ with an arc $[z'_i, z'_j] \subset Q^{(k-1)}$ such that $[z'_i, z'_j] \cap \eta_{k_i} = \{ z'_i \}$ and $[z'_i, z'_j] \cap \eta'_{k_i} = \{ z'_j \}$. Note that $[z'_i, z'_j] \cap \Lambda_F = \emptyset$. The exact choice of such an arc $[z'_i, z'_j]$ is irrelevant, since $\diam \left(Q^{(k-1)} \right) \rightarrow_{k\to\infty}0$. The arc $\zeta'_{k_i} = [z'_i, z'_j]$ is a connector, see Figure \ref{fig:connectors2}.
		
		For convenience, to unify notation, let $\hat Q^{(k)} = Q^{(k)}$, for $k \in \N$ and $b < 0$, and let $\zeta'_{k_i} = \emptyset$ for $k_i$, $i \in \N_0$, and $b > 0$.
		
		Now we are ready to define $$\K = \bigcap_{k \in \N} \left(\bigcup_{i=1}^{n_k}\hat Q^{(k)}_i \cup \bigcup_{i=1}^{n_k-1} \zeta^k_i \cup \bigcup_{k_i \le k} \zeta'_{k_i} \right),$$ 
		where $\{ \hat Q^{(k)}_i : i = 1, \ldots , n_k \}$ are the components of $\C^{(k)}$, and $\zeta^k_i$ is a connector joining $\hat Q^{(k)}_i$ with $\hat Q^{(k)}_{i+1}$. Note that $\diam(\hat Q^{(k)}_i) \longrightarrow_{k\to\infty} 0$ and $\K = \C \cup\bigcup_{i \in \Z \setminus \{ 0, -1 \}} \zeta_i$ for $b > 0$, and  $\K = \C \cup\bigcup_{i=2}^\infty \zeta_i \cup \bigcup_{i=0}^\infty \zeta'_{k_i}$ for $b < 0$.
		
		It is easy to see that all the connectors satisfy the following property:
		\begin{enumerate}[($*$)]
			\item For $b > 0$ ($b < 0$ respectively), after $k$th step every basic critical (quasi-critical, respectively) point of $\bigcup_{i=1}^{k} \partial^u U_i$ is connected with one other basic critical (quasi-critical, respectively) point by a connector. Note that by (3) (iii) (a) \& (b), if some component $Q^{(k-1)}$ of $\C^{(k-1)}$ contains only one component $Q^{(k)}$ of $\C^{(k)}$, then the basic critical (quasi-critical, respectively) points in $\partial^u \hat Q^{(k)}$ have already been visited and connected in the previous step. 
			
			For $b > 0$, for every $z_i \in \mfc$, $i \in \Z \setminus \{ 0, 1 \}$, there exists only one connector $\zeta$ with $z_i \in \zeta$, since by (2) (iii), we visit every basic critical point except $z_0$ and $z_{-1}$, and by (2) (ii), we visit every one of them only once. 
			
			For $b < 0$, for every quasi-critical point $z'_i \in \H$, $i \ge 2$, there exist two connectors $\zeta$ and $\zeta'$ with $\{z'_i\} = \zeta \cap \zeta'$, and each of them connects $z'_i$ with a different quasi-critical point. Each of the quasi-critical points $z'_0$ and $z'_1$ has only one connector that contains it, $z'_0 \in \zeta'_{k_0}$ and $z'_1 \in \zeta'_{k_1}$.
		\end{enumerate}  
		
		It is easy to see that $\K$ is a connected and compact set since it is the intersection of a descending family of connected and compact sets. To see that $\K$ is an arc, it is enough to show that $\K\setminus\{z\}$ is disconnected, for any $z \in \K \setminus \{ z_0, z_{-1} \}$ for $b > 0$, and for any $z \in \K \setminus \{ z'_0, z'_1 \}$ for $b < 0$. To that end we define an order $\tl$ on $\K$: If $z, z' \in \C \cup \H$, we define that $z \tl z'$ if $z$ is above $z'$ (if $b > 0$, then $\H = \emptyset$). Let $[z_i, z_j], [z_n, z_m]$ be two different connectors and $z_i \tl z_j$. If $P, Q \in [z_i, z_j]$, then $P \tl Q$ if $\length([z_i, P]) < \length([z_i, Q])$, where $[z_i, P], [z_i, Q] \subset [z_i, z_j]$. If $P \in [z_i, z_j)$ and $Q \in (z_n, z_m]$, then $P \tl Q$ if $z_j \tl z_n$ or $z_j = z_n$.  Note that by $(*)$ the order is well defined. Given $z \in \K \setminus \{ z_0, z_{-1} \}$ for $b > 0$, or $z \in \K \setminus \{ z'_0, z'_1 \}$ for $b < 0$, it is easy to see that the sets $\{z' \in \K : z \tl z'\}$ and $\{z' \in \K : z' \tl z\}$ are open in $\K$ in the topology inherited from the plane, and consequently $\K\setminus\{z\}$ is disconnected.
	\end{proof}
	
	Let $\gamma_X$ denote a unique (connected) component of $D \cap W^s$ that contains $X$, $X \in \gamma_X$. It has two endpoints. For $b > 0$, one is $X$, and we denote the other one by $E$, $\{ X , E \} = \partial^u D \cap \gamma_X$. Let $\Omega$ be the closed region bounded by two arcs: $\partial^u \Omega = [X, E^{-1}]^u \subset W^u$ and $\partial^s \Omega = [X, E^{-1}]^s \subset W^s$, see Figure \ref{fig:omega}. Note that $z_0 \in [X, E^{-1}]^u$ and $\K \subset \Omega$.
	
	For $b < 0$, let $E$ denote an endpoint of $\gamma_X$ that is closer to $Y$ on $W^u_Y$. Then the other one is $E^1$, $\gamma_X = [E, E^1]^s$. Let again $\Omega$ be the closed region bounded by two arcs: $\partial^u \Omega = [E, E^{-1}]^u \subset W^u_Y$ and $\partial^s \Omega = [E, E^{-1}]^s \subset W^s$, see Figure \ref{fig:omega2}. Note that $z'_0 \in [X, E^{-1}]^u$ and $\K \subset \Omega$.
	
	In what follows, we will work with the inverse images of the critical locus, $F^{-n}(\K)$, $n \in \N_0$. Since $\mathring \zeta_i$ are chosen with very mild restrictions, we do not have a lot of control over their preimages, but we have what we need. First, we are interested only in points in the attractor $\Lambda_F$ and $\mathring \zeta_i \cap \Lambda_F = \emptyset$. Second, for every $i$, there is $k \in \N$, such that $\mathring \zeta_i \subset U_k \cap \Omega$. Therefore, for every $n \in \N$ we have $F^{-n}(\K) \subset F^{-n}(\Omega)$ and moreover, $n < k$ implies $F^{-n}(\mathring \zeta_i) \subset F^{-n}(U_k \cap \Omega) \subset U_{k-n} \cap F^{-n}(\Omega)$, what is needed.

	\section{Coding of orbits on the attractor}\label{sec:coa}
	
	As we have already mentioned in the Introduction, the existence of coding for the H\'enon maps $F$ within the parameter set $\WY$ is proved in \cite[Theorem 1.6]{WY}. In this section, we provide an alternative proof. Recall that for a point $P \in D$ we let $P^j = F^j(P)$ for any $j \in \Z$.
	
	The critical locus $\K$ divides $D$ into two components. We denote by $D^l$ the one that lies to the left of $\K$, and by $D^r$ the one that lies to the right of $\K$. Both of them contain the critical locus. Also $F(\K)$ divides $F(D)$ into two other components, say $D^\uparrow \subset F(D)$ that lies above $F(\K)$, and $D^\downarrow \subset F(D)$ that lies below $F(\K)$. In \cite{BenedicksViana} it is proved that $W^s \cap D$ is dense in $D$. Since $\K \subset \Omega \cap D$, this implies that the set $\bigcup_{i \in \N_0}(F^{-i}(\K) \cap D)$ is dense in $D$.
	
	We code the points of $\Lambda_F$ in the following way. To a point $P \in \Lambda_F$ we assign a bi-infinite sequence $\bar p = \dots p_{-2} \, p_{-1} \cdot \, p_0 \, p_1 \, p_2 \dots$ such that
	$$
	p_n = \begin{cases} -, & \textrm{if  } F^n(P) \in D^l,\\
		+, & \textrm{if  } F^n(P) \in D^r.
	\end{cases} 
	$$
	The dot shows where the 0th coordinate is. 
	
	A bi-infinite symbol sequence $\bar q = \dots q_{-2} \, q_{-1} \cdot \, q_0 \, q_1 \, q_2 \dots$ is called \emph{admissible} if there is a point $Q \in \Lambda_F$ such that $\bar q$ is assigned to $Q$. This sequence is called an \emph{itinerary} of $Q$. Since both components that we use for coding, $D^l$ and $D^r$, contain the critical locus, some points of $\Lambda_F$ have more than one itinerary. We denote the set of all admissible sequences by $\Sigma_F$. It is a metrizable topological space with the usual product topology. Since $D^l$ and $D^r$ (with the boundary) intersected with $\Lambda_F$, are compact, the space $\Sigma_F$ is compact. 
	
	\begin{lem}\label{lem:unique}
		For every $\bar p \in \Sigma_F$ there exists only one point $P \in \Lambda_F$ with that itinerary.
	\end{lem}
\begin{proof}		
	Let $P, Q \in \Lambda_F$, $P \ne Q$. It was shown in \cite[Proof of Theorem 1.1]{BS} that there exists a $k$ and distinct $\alpha, \alpha' \in \mathcal{A}$ such that $P^{-k}\in \alpha$ and $Q^{-k}\in\alpha'$. By \cite[Proofs of Lemma 4.2 and Theorem 1.4]{BS} there exists an $n>0$ and $s^n$, a component of $F^{-n}(\Omega)\cap D$, such that $s^n$ separates $\alpha$ and $\alpha'$ in $D$. Since $\partial^s s^n$ is separated by a component of $F^{-n}(\K)$ we get that $P^{n-k}$ and $Q^{n-k}$ have different codings and in turn $P$ and $Q$ different itineraries, which completes the proof.
\end{proof}
	
	By Lemma~\ref{lem:unique} and the definition of $\Sigma_F$, the map $\iota : \Sigma_F \to \Lambda_F$ such that $\bar p \in \Sigma_F$ is an itinerary of $\iota(\bar p) \in \Lambda_F$ is well defined and is a surjection. Clearly, $F \circ \iota = \iota \circ \sigma$, where $\sigma : \Sigma_F \to \Sigma_F$ is the shift homeomorphism.
	
	\begin{lem}\label{lem:continuous}
		The map $\iota$ is continuous.
	\end{lem}
	\begin{proof}
		Let $\bar p \in \Sigma_F$ and let $(\bar p^n)_{n=1}^\infty$ be a sequence of elements of $\Sigma_F$ which converges to $\bar p$. Set $P =\iota(\bar p)$ and $P_n = \iota(\bar p^n)$. We will prove that $(P_n)_{n=1}^\infty$ converges to $P$.
		
		Let us suppose by contradiction, that $(P_n)_{n=1}^\infty$ does not converge to $P$. Since $\Lambda_F$ is compact, there exists a subsequence of the sequence $(P_n)_{n=1}^\infty$ convergent to some $Q \ne P$. Without loss of generality, we assume that this subsequence is the original sequence $(P_n)_{n=1}^\infty$.
		
		If for every $j \in \Z$ both $P^j$ and $Q^j$ belong to the same closed disc $D^l$ or $D^r$, then there is $\bar q \in \Sigma_F$ which is an itinerary of both $P$ and $Q$. This contradicts Lemma~\ref{lem:unique}. Therefore, there is $j\in \Z$ such that $P^j$ belongs to the one of the open discs $\Int_D D^l = D^l \setminus \K$ or $\Int_D D^r$, and $Q^j$ belongs to the other one. Since the points $P_n$ converge to $Q$, for all $n$ large enough the points $P_n^j$ and $P^j$ belong to the different open discs $\Int_D D^l$ and $\Int_D D^r$. This means that for all $n$ large enough we have $p^n_j\ne p_j$. Therefore, the sequence $(\bar p^n)_{n \in \N}$ cannot converge to $\bar p$, a contradiction. This completes the proof.
	\end{proof}
	
	For a sequence $\bar p = (p_i)_{i \in \Z} \in \Sigma_F$ and $n \in \Z$, we call the left-infinite sequence $\la p_{\h n} = \dots p_{n-2} \, p_{n-1} \, p_n$ a \emph{left tail} of $\bar p$ and the right-infinite sequence $\ra p_{\h n} = p_n \, p_{n+1} \, p_{n+2} \dots$ a \emph{right tail} of $\bar p$. We call a finite sequence $w = w_1 \dots w_k$ a \emph{word} and denote its length by $|w|$, $|w| = k$. We denote an infinite to the right (respectively, left) sequence of $+$s by $\rpinf$ (respectively, $\lpinf$).
	
	By \cite[Theorem 1.1 (2) (iii)]{WY}, for every $z \in \C$, we have $z^n \nin \C$ for every $n \in \N$, and hence $z^i \nin \C$ for every $i \in \Z \setminus \{ 0 \}$. This implies that every point $P \in \Lambda_F$ has at most two itineraries, and if $P$ has two itineraries, then they differ at one coordinate, say $k$, and $P^k$ is a critical point. In such a case, for simplicity, we write both itineraries of $P$ as one sequence $\bar p = (p_i)_{i \in \Z}$ such that its $k$th coordinate is $p_k = \pm$.
	
	Let us now consider itineraries of the points of $W^u$. Let $\Sigma_{W^u}$ denote the set of all itineraries of all points of $W^u$. The itinerary of the fixed point $X$ (which is in $D^r$) is $\bar x = \lpinf \cdot \rpinf$. Since $W^u$ is the unstable manifold of $X$, for every point $P \in W^u$ and its itinerary $\bar p$, there is $n \in \Z$ such that $\la p_{\h n} = \lpinf$. Therefore, if the orbit of $P$ intersects the critical set $\C$, there exists a unique integer $k > n$ such that $P^k \in \C$. Also $P^{k+1} \in F(\C)$ lies in $D^r$. By definition, the set $\Sigma_{W^u}$ is $\sigma$-invariant. If a sequence is an itinerary of a point of $W^u$, we call it \emph{$W^u$-admissible}.
	
	\begin{lem}\label{approx}
		Assume that $\bar p \in \Sigma_{W^u}$ and $n \in \N$. Then there is $\bar q \in \Sigma_{W^u}$ such that $p_{-n},\dots p_n = q_{-n}\dots q_n$ and $\bar q$ is the only itinerary of $\iota(\bar q)$.
	\end{lem}
	\begin{proof}
		Let $n \in \N$, $\bar p \in \Sigma_{W^u}$ and $\iota(\bar p) = P \in W^u$. If $P^i \nin \K$ for every $i \in \Z$, then $\bar p$ is the only itinerary of $P$ and the claim holds.
		
		Let us suppose that there exists $j \in \Z$ such that $P^j \in \K$. By \cite[Theorem 1.1 (2) (iii)]{WY}, $P^i \nin \K$ for every $i \in \Z$, $i \ne j$. Therefore, for every $i \in \{ -n, \dots , n \}$, $i \ne j$, there exists $\eps_i > 0$ such that for the (closed) arc $J_i = [L_i, R_i]^u \subset W^u$ with $d(L_i, P^i) = d(R_i, P^i) = \eps_i$ we have $J_i \cap \K = \emptyset$. Let $J_j = \bigcap_{i=-n, \, i \ne j}^n F^{j-i}(J_i)$. Then $J_j = [L_j, R_j]^u \subset W^u$ is a (closed non-degenerate) arc with a boundary point $L_j$ in $D^l$ and the other boundary point $R_j$ in $D^r$, $P^j \in (L_j, R_j)$, and for every point $Q \in J_j$, $Q \ne P^j$, the points $Q^i \in F^i(J_j) \subset J_{j+i}$ and $P^{j+i}$ lie in the same disc, $D^l$ or $D^r$, for all $i \in \{ -n-j, \dots , n-j \}$, $i \ne 0$. This implies that for $p_j = +$ (respectively $p_j = -$) and for every $\bar q$ such that $\iota(\bar q) = Q \in (P^j, R_j]$ (respectively $\iota(\bar q) = Q \in [L_j, P^j)$), we have $q_i = p_i$ for all $i \in \{ -n, \dots , n \}$. 
		
		Since the unstable manifold $W^u$ intersects the critical locus $\K$ only at countably many points, the set of points $R \in W^u$ such that $R^k$ belongs to $\K$ for some $k \in \Z$ is also countable. Thus, there are points $Q \in J = F^{-j}(J_j) \subset J_0$, on both sides of $P = P^0$, such that for every $k \in \Z$ the point $Q^k$ does not belong to $\K$. Such $Q$ has only one itinerary and by construction $p_{-n},\dots p_n = q_{-n}\dots q_n$. This completes the proof.
	\end{proof}
	
	The space $\Cl \Sigma_{W^u}$, as  the closure of a $\sigma$-invariant space, is also $\sigma$-invariant. We want to show that the sequences of $\Cl \Sigma_{W^u}$ suffice for the symbolic description of dynamics on $\Lambda_F$. We know that this is true for $\Sigma_{W^u}$, that is, we know that $\iota(\Sigma_{W^u}) = W^u$. We also want to show that the sequences of $\Cl \Sigma_{W^u}$ are essential, that is, we cannot remove any of them from our symbolic system. 
	
	\begin{lem}\label{essential}
		$\iota(\Cl \Sigma_{W^u}) = \Lambda_F$ and $\Cl \Sigma_{W^u}$ is the minimal set with this property, that is, each compact subset $\Sigma'$ of $\Sigma_F$ such that $\iota(\Sigma') = \Lambda_F$, contains $\Cl \Sigma_{W^u}$.
	\end{lem}
	\begin{proof}
		Since the set $\Cl \Sigma_{W^u}$ is compact, so is $\iota(\Cl \Sigma_{W^u})$. Also, $\iota(\Cl \Sigma_{W^u})$ contains $\iota(\Sigma_{W^u}) = W^u$, which is dense in $\Lambda_F$, so $\iota(\Cl \Sigma_{W^u})$ is equal to $\Lambda_F$.
		
		Now suppose that $\Sigma' \subseteq \Sigma_F$ is a compact set such that $\iota(\Sigma') = \Lambda_F$. The itineraries of all points of $W^u$ with unique itineraries clearly belong to $\Sigma'$. By Lemma~\ref{approx}, the set of those itineraries is dense in $\Sigma_{W^u}$. Since $\Sigma'$ is closed, we get $\Sigma_{W^u} \subset \Sigma'$, and then $\Cl \Sigma_{W^u} \subseteq \Sigma'$.
	\end{proof}
	By the definition of $\Sigma_F$ and since for every $z \in \C$ we have $z^i \nin \C$ for every $i \in \Z \setminus \{ 0 \}$, the following corollary holds.
	\begin{cor}
		$\Sigma_F = \Cl \Sigma_{W^u}$.
	\end{cor}

	\section{The pruning front conjecture}\label{sec:pfc}
	
	Recall that the set of the basic critical points is $\mfc = \C \cap W^u$, and the points in $\mfc^+ = \bigcup_{n \in \N} F^n(\mfc)$ are called the basic post-critical points. We call the points in $\mfc \cup \mfc^+$ the \emph{basic points}.
	
	We partition $W^u$ into basic arcs. A \emph{basic arc} is an arc in $W^u$ whose endpoints are two consecutive basic points. The basic arc $[z_0, z_0^1]^u$ contains the fixed point and for every point $P \in [z_0, z_0^1]^u$ we have $\la p_{\h 0} = \lpinf$. Its consecutive basic arc in $W^{u-}$ is $[z_0^2, z_0]^u$ and for every point $P \in [z_0^2, z_0]^u$ we have $\la p_{\h 0} = \lpinf-$. Note that $F([z_0^2, z_0]^u) = [z_0^1, z_0^3]^u$. In general, $z_0^3$ might lie in $D^r$, and in that case, $[z_0^1, z_0^3]^u$ would be a basic arc. However, for the Wang-Young set of parameters that we are considering in this paper, a lot of the initial iterations of $z_0^1$ lie in $D^l$. Therefore, $[z_0^1, z_0^3]^u$ contains the critical point $z_{-1}$, the consecutive basic arc of $[z_0, z_0^1]^u$ in $W^{u+}$ is $[z_0^1, z_{-1}]^u$ and for every point $P \in [z_0^1, z_{-1}]^u$ we have $\la p_{\h 0} = \lpinf-+$. Also, for the next consecutive arc $[z_{-1}, z_0^3]^u$ and for every $P \in [z_{-1}, z_0^3]^u$ we have $\la p_{\h 0} = \lpinf--$. In general, if $J$ is a basic arc, there exists a word $w = w_{-n+1} \dots w_{-1}w_0$ such that for every point $P \in J$ we have $\la p_{\h 0} = \lpinf w_{-n+1} \dots w_{-1} w_0 = \lpinf w$, $w_{-n+1} = -$, and any two different basic arcs have different related words. This allows us to code the basic arcs with the related words, $J = I_w$, and we call the word $w$ the \emph{arc-code} of the basic arc $I_w$. In particular, the arc-code of $[z_0, z_0^1]^u$ is $\emptyset$, so $[z_0, z_0^1]^u = I_\emptyset$. Observe that if $I_w$ and $I_{w'}$ are two consecutive basic arcs, then $\lpinf w$ and $\lpinf w'$ differ in only one coordinate. Also, for every $m \in \N$, all basic arcs of $F^{m-1}([z_0^2, z_0]^u)$ have arc-codes of length $m$. Moreover, if $m$ is even, $F^{m-1}([z_0^2, z_0]^u) \subset W^{u+}$, and if $m$ is odd, $F^{m-1}([z_0^2, z_0]^u) \subset W^{u-}$, see Figure \ref{fig:FP}.
	
	Let us recall the standard parity-lexicographical ordering $\preceq$ on a set of (one-sided) sequences of $+$s and $-$s: First, $- \prec +$. Let $\ra p = p_0 p_1 \dots p_n \dots$ and $\ra q = q_0 q_1 \dots q_n \dots$ be two different sequences, or $\ra p = p_0 p_1 \dots p_n$ and $\ra q = q_0 q_1 \dots q_n$ be two different finite words of the same length. Let $m \in \N_0$ be the smallest non-negative integer such that $p_m \ne q_m$ (in the case of finite words $m \le n$). Then $\ra p \prec \ra q$ if and only if the number of $+$s in $p_0 \dots p_{m-1}$ is even and $p_m \prec q_m$, or the number of $+$s in $p_0 \dots p_{m-1}$ is odd and $q_m \prec p_m$. Here, if $m = 0$, then $p_0 \dots p_{m-1}$ is the empty word. Also, if $p_m = \pm$, or $q_m = \pm$, then by convention $- \prec \pm \prec +$.
	
	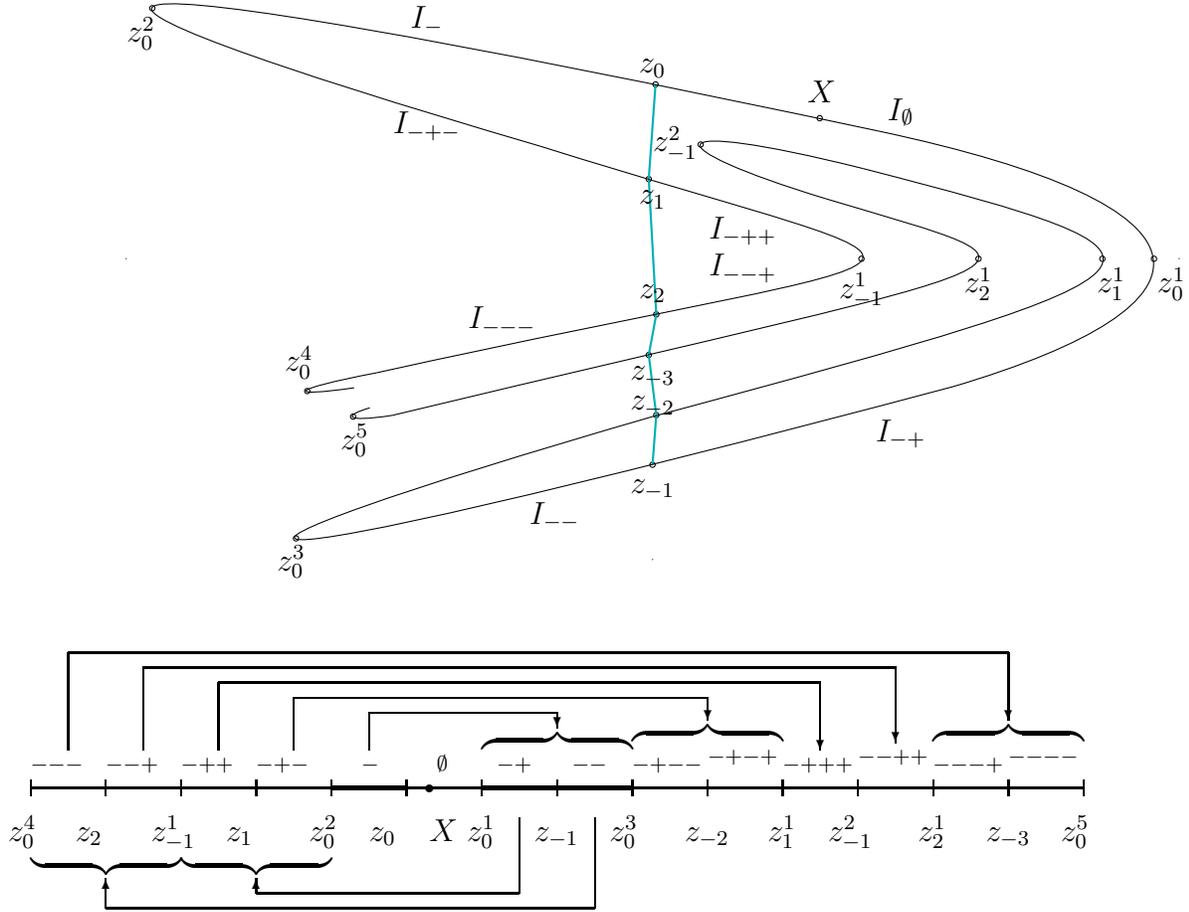
\begin{figure}[h!]
		\begin{adjustbox}{center}
			\begin{tikzpicture}
				\tikzstyle{every node}=[draw, circle, fill=white, minimum size=2pt, inner sep=0pt]
				\tikzstyle{dot}=[circle, fill=white, minimum size=0pt, inner sep=0pt, outer sep=-1pt]
				\node (n1) at (5*4/3,0) {};
				\node (n2) at (-6.65,5*2/3)  {};
				\node (n3) at (-4.74,-3.72)  {};
				\node (n4) at (5*5/9,0)  {};
				\node (n5) at (-4.59,-1.76) {};
				\node (n6) at (-3.98,-2.1)  {};
				\node (n7) at (0.64,1.52)  {};
				\node (n8) at (5*61/51,0)  {};
				\node (n9) at (5*65/75,0)  {};
				
				\node (n10) at (5*4/9,1.87)  {};
				
				\node (n21) at (0.04,2.32)  {};
				\node (n22) at (-0.05,1.06)  {};
				\node (n23) at (0.05,-0.74)  {};
				\node (n24) at (-0.05,-1.28)  {};
				\node (n25) at (0.05,-2.08)  {};
				\node (n26) at (0,-2.74)  {};
				
				\tikzstyle{every node}=[draw, circle, fill=white, minimum size=0.1pt, inner sep=0pt]
				\node (n16) at (-7,0) {};
				\node (n17) at (7,0)  {};
				\node (n18) at (0,4)  {};
				\node (n19) at (0,-4)  {};
				
				\draw[TealBlue,thick](0.04,2.32)--(-0.05,1.06);
				\draw[TealBlue,thick](-0.05,1.06)--(0.05,-0.74);
				\draw[TealBlue,thick](0.05,-0.74)--(-0.05,-1.28);
				\draw[TealBlue,thick](-0.05,-1.28)--(0.05,-2.08);
				\draw[TealBlue,thick](0.05,-2.08)--(0,-2.74);
				
				\draw (3,1.7) .. controls (7.7,0.7) and  (7.7,-0.6).. (4,-1.7);	
				\draw (3,1.7) .. controls (-9,4.25) and (-9,3.7) .. (-1.5,1.5);
				\draw (0,-2.1) .. controls (-6.9,-4.1) and (-6.9,-4.6) .. (4,-1.7);
				\draw (2.8,1.2) .. controls (7.44,0) and (7.44,-0.1) .. (0,-2.1);
				\draw (2.8,1.2) .. controls (0,1.9) and (0,1.5) .. (2.3,0.8);
				\draw (2.3,0.8) .. controls (5.63,-0.2) and (5.63,0.1) .. (-5*19/24+0.5,-5*5/12);
				\draw (-1.5,1.5) .. controls (4.55,-0.4) and (4.55,0.3) .. (-5*5/6+0.5,-5*1/3+0.15);
				\draw (-5*5/6+0.5,-5*1/3+0.15) .. controls (-5*5/6-0.7,-5*1/3-0.08) and (-5*5/6-0.7,-5*1/3-0.18) .. (-5*5/6+0.2,-5*1/3-0.05);
				\draw (-5*19/24+0.2,-5*5/12+0.1) .. controls (-5*19/24-0.15,-5*5/12) and (-5*19/24-0.15,-5*5/12-0.1) .. (-5*19/24+0.5,-5*5/12);
				
				\node[dot, draw=none, label=above: $z^1_0$] at (6.9,0) {};
				\node[dot, draw=none, label=above: $z_0$] at (0,2.8) {};
				\node[dot, draw=none, label=above: $z^2_0$] at (-6.8,5*2/3) {};
				\node[dot, draw=none, label=above: $z_1$] at (0,1.05) {};
				\node[dot, draw=none, label=above: $z^1_{-1}$] at (5*5/9,0) {};
				\node[dot, draw=none, label=above: $z_2$] at (0,-0.2) {};
				\node[dot, draw=none, label=above: $z^4_0$] at (-4.7,-1.05) {};
				\node[dot, draw=none, label=above: $z^5_0$] at (-5*19/24,-5*5/12) {};
				\node[dot, draw=none, label=above: $z_{-3}$] at (0,-1.2) {};
				\node[dot, draw=none, label=above: $z^1_2$] at (5*65/75,0) {};
				\node[dot, draw=none, label=above: $z^2_{-1}$] at (0.3,1.95) {};
				\node[dot, draw=none, label=above: $z^1_1$] at (6.1,0) {};
				\node[dot, draw=none, label=above: $z^3_0$] at (-4.8,-3.7) {};
				\node[dot, draw=none, label=above: $z_{-1}$] at (0,-2.7) {};
				\node[dot, draw=none, label=above: $z_{-2}$] at (0,-1.6) {};
				\node[dot, draw=none, label=above: $X$] at (2.25,2.5) {};
				\node[dot, draw=none, label=above: $I_{-}$] at (-3,3.5) {};
				\node[dot, draw=none, label=above: $I_{\emptyset}$] at (3.3,2.25) {};
				\node[dot, draw=none, label=above: $I_{-+-}$] at (-3,2.3) {};
				\node[dot, draw=none, label=above: $I_{-++}$] at (1.2,0.9) {};
				\node[dot, draw=none, label=above: $I_{--+}$] at (1.2,0.4) {};
				\node[dot, draw=none, label=above: $I_{---}$] at (-2,-0.3) {};
				\node[dot, draw=none, label=above: $I_{-+}$] at (3.3,-1.9) {};
				\node[dot, draw=none, label=above: $I_{--}$] at (-1.3,-3) {};
			\end{tikzpicture}
		\end{adjustbox}
		
		\unitlength=10mm
		
		\begin{picture}(20,4)(3,4.4)
			
			\put(8.3,6){\circle*{0.1}}
			\put(8.3,5.3){$X$}
			
			\put(3,6){\line(1,0){14}}
			
			\put(3,5.9){\line(0,1){0.2}}\put(2.7,5.3){$z_0^4$}\put(3,6.3){$_{---}$}
			\put(4,5.9){\line(0,1){0.2}}\put(3.6,5.3){$z_{2}$}\put(4,6.3){$_{--+}$}
			\put(5,5.9){\line(0,1){0.2}}\put(4.6,5.3){$z^1_{-1}$}\put(5,6.3){$_{-++}$}
			\put(6,5.9){\line(0,1){0.2}}\put(5.6,5.3){$z_{1}$}\put(6,6.3){$_{-+-}$}
			\put(7,5.9){\line(0,1){0.2}}\put(6.7,5.3){$z^2_0$}\put(7.4,6.3){$_{-}$}
			\put(8,5.9){\line(0,1){0.2}}\put(7.5,5.3){$z_0$}\put(8.4,6.3){$_{\emptyset}$}
			\put(9,5.9){\line(0,1){0.2}}\put(8.8,5.3){$z^1_0$}\put(9.2,6.3){$_{-+}$}
			\put(10,5.9){\line(0,1){0.2}}\put(9.7,5.3){$z_{-1}$}\put(10.2,6.3){$_{--}$}
			\put(11,5.9){\line(0,1){0.2}}\put(10.7,5.3){$z_0^3$}\put(11,6.3){$_{-+--}$}
			\put(12,5.9){\line(0,1){0.2}}\put(11.7,5.3){$z_{-2}$}\put(12,6.4){$_{-+-+}$}
			\put(13,5.9){\line(0,1){0.2}}\put(12.8,5.3){$z^1_1$}\put(13,6.3){$_{-+++}$}
			\put(14,5.9){\line(0,1){0.2}}\put(13.6,5.3){$z_{-1}^2$}\put(14,6.4){$_{--++}$}
			\put(15,5.9){\line(0,1){0.2}}\put(14.8,5.3){$z^1_2$}\put(15,6.3){$_{---+}$}
			\put(16,5.9){\line(0,1){0.2}}\put(15.7,5.3){$z_{-3}$}\put(16,6.4){$_{----}$}
			\put(17,5.9){\line(0,1){0.2}}\put(16.7,5.3){$z_0^5$}
			
			\put(9,6.5){$\overbrace{\qquad \qquad \ \ \, }$}
			\put(3,5.1){$\underbrace{\qquad \qquad \ \ \, }$}
			\put(5,5.1){$\underbrace{\qquad \qquad \ \  \,}$}
			\put(11,6.6){$\overbrace{\qquad \qquad \ \ \, }$}
			\put(15,6.6){$\overbrace{\qquad \qquad \ \ \, }$}
			\put(7.5,6.5){\line(0,1){0.5}}
			\put(7.5,7){\line(1,0){2.5}}
			\put(10,7){\vector(0,-1){0.2}}
			
			\put(9.5,5.6){\line(0,-1){1}}
			\put(9.5,4.6){\line(-1,0){3.5}}
			\put(6,4.6){\vector(0,1){0.2}}
			
			\put(10.5,5.6){\line(0,-1){1.2}}
			\put(10.5,4.4){\line(-1,0){6.5}}
			\put(4,4.4){\vector(0,1){0.4}}
			
			\put(6.5,6.5){\line(0,1){0.7}}
			\put(6.5,7.2){\line(1,0){5.5}}
			\put(12,7.2){\vector(0,-1){0.3}}
			
			\put(5.5,6.5){\line(0,1){0.9}}
			\put(5.5,7.4){\line(1,0){8}}
			\put(13.5,7.4){\vector(0,-1){0.9}}
			
			\put(4.5,6.5){\line(0,1){1.1}}
			\put(4.5,7.6){\line(1,0){10}}
			\put(14.5,7.6){\vector(0,-1){1}}
			
			\put(3.5,6.5){\line(0,1){1.3}}
			\put(3.5,7.8){\line(1,0){12.5}}
			\put(16,7.8){\vector(0,-1){0.9}}
			
			\thicklines
			
			\put(7,5.99){\line(1,0){1}}
			\put(7,6.01){\line(1,0){1}}
			\put(9,5.99){\line(1,0){2}}
			\put(9,6.01){\line(1,0){2}}
			
		\end{picture}
		
		\caption{Several basic points, basic arcs, and arc-codes. The critical locus $\K$ is teal.}
		\label{fig:FP}
	\end{figure}
	
	\begin{lem}\label{lem:arc-codes}
		Let $u$, $v$ be two different arc-codes, and let $I_u$, $I_v$ be the corresponding basic arcs. If $u$ and $v$ have different lengths, but $|u|$ and $|v|$ have the same parity, then $|u| >|v|$ if and only if the basic arc $I_u$ is farther from $X$ then the basic arc $I_v$ $(d(I_u, X) > d(I_v, X))$. If $u$ and $v$ have the same length, then $u \prec v$ if and only if $d(I_u, X) > d(I_v, X)$.
	\end{lem}
	\begin{proof}
		If $u$ and $v$ have different lengths, but $|u|$ and $|v|$ have the same parity, then we take $n$ of the same parity as $|u|$ and $|v|$ and such that $F^{-n}(I_u)$ and $F^{-n}(I_v)$ are contained in $[z_0, z^1_0]^u$. Choose $P \in F^{-n}(I_u)$ and $Q \in F^{-n}(I_v)$. Compare $\ra p_{\h 0}$ with $\ra q_{\h 0}$. By the parity assumptions, they both start with the odd number of $+$s. If $|u| >|v|$ then $\ra q_{\h 0}$ starts with more $+$s, so $\ra q_{\h 0} \prec \ra p_{\h 0}$, and therefore $Q < P$. This means that $I_{v}$ is closer to $X$ than $I_{u}$.
		
		If $u$ and $v$ have the same length, we make the same construction. Then $u \prec v$ is equivalent to $\ra q_{\h 0} \prec \ra p_{\h 0}$ (remember of the odd number of $+$s in front), and, as before, $I_{v}$ is closer to $X$ than $I_{u}$.
	\end{proof}
	
	We also partition $W^u$ into leaves. A leaf is an arc in $W^u$ whose endpoints are two consecutive post-critical points. Let $\ell^0$ be the leaf which contains the fixed point. We index leaves such that $\ell^i$ and $\ell^j$ are adjacent if and only if $|i - j| = 1$ for all $i, j \in \Z$, and that $\ell^k \subset W^{u+}$ and $\ell^{-k} \subset W^{u-}$ for all $k \in \N$. 
	
	For any two different points $P, Q \in \ell^{2i}$, $i \in \Z$, we say that $P$ is \emph{on the left of} $Q$ if $P < Q$. For any two different points $P, Q \in \ell^{2i+1}$, $i \in \Z$, we say that $P$ is \emph{on the left of} $Q$ if $Q < P$. In any case, if $P$ is on the left of $Q$, we also say that $Q$ is \emph{on the right of} $P$. 
	
	This definition is motivated by the fact that one endpoint of  $\ell^0$, $z_0^1$, lies in $D^r$, and the other one, $z_0^2$, lies in $D^l$, so we can say that $z_0^2$ is on the left of $z_0^1$, and we have $z_0^2 < z_0^1$. Also, any two adjacent leaves $\ell^j$ and $\ell^{j+1}$ have a common boundary point and it is on the same side, left or right, of the other two boundary points. But this common boundary point is between the other two boundary points, so it is less than one of them and greater than the other one. This fact also motivates the next definition.
	
	Let $P, Q, R, T \in W^u$ be points such that $P, Q \in \ell^i$ and $R, T \in \ell^j$, for some $i, j \in \Z$ (not necessarily different). We say that the points $P, Q$ have the \emph{opposite orientation than} the points $R, T$, respectively if $P$ is on the left of $Q$ and $R$ is on the right of $T$, or vice versa. We say that the points $P, Q$ have the \emph{same orientation as} the points $R, T$, respectively if $P$ is on the same side of $Q$ as $R$ with respect to $T$. 
	
	Let us define the \emph{generalized parity-lexicographical order} on the set $\Sigma_{W^u}$ in the following way. 
	
	\begin{df}\label{df:gplo}
		Let $\bar p, \bar q \in \Sigma_{W^u}$, $\bar p \ne \bar q$. Let $n \in \N$ be a positive integer such that $\la p_{\h -n} = \la q_{\h -n} = \lpinf$. Then $\bar p \prec \bar q$ if either
		\begin{enumerate}[(1)]
			\item $n$ is even and $\ra q_{\h -n+1} \prec \ra p_{\h -n+1}$, or
			\item $n$ is odd and $\ra p_{\h -n+1} \prec \ra q_{\h -n+1}$.
		\end{enumerate}
	\end{df}
	
	By the definition of the parity-lexicographical order and since $F$ reverses orientation on $W^u$, this order is well defined (it does not depend on the choice of $n$).
	
	\begin{lem}\label{lem:o1}
		Let $P, Q \in \ell^0$ and let $\bar p$, $\bar q$, respectively, be their itineraries. Then $P < Q$ if and only if $\bar p \prec \bar q$.
	\end{lem}
	\begin{proof}
		From $P, Q \in \ell^0$, it follows $\la p_{\h -1} = \la q_{\h -1} = \lpinf$. Also, $P < Q$ if and only if $P$ is to the left of $Q$. 
		
		Let $n \in \N_0$ be the smallest integer such that $p_n \ne q_n$. If $n = 0$, then $P < Q$ if and only if $P$ lies in $D^l$ and $Q$ lies in $D^r$, which is equivalent to $p_0 = -$, $q_0 = +$, and hence $\bar p \prec \bar q$.
		Let $n \in \N$. Since $F$ maps $D^r$ to $D^\uparrow$, if  $p_i = q_i = +$, for some $i \in \{ 0, \dots , n-1 \}$, then $P^{i+1}$ and $Q^{i+1}$ have the opposite orientation than $P^i$ and $Q^i$. Also, since $F$ maps $D^l$ to $D^\downarrow$, if $p_i = q_i = -$, for some $i \in \{ 0, \dots , n-1 \}$, then $P^{i+1}$ and $Q^{i+1}$ have the same orientation as $P^i$ and $Q^i$. Therefore, $P^n, Q^n$ have the same orientation as $P, Q$, respectively, if and only if $p_0 \dots p_{n-1}$ is even. In that case, $P < Q$ if and only if $P^n$ lies in $D^l$ and $Q^n$ lies in $D^r$, which is equivalent to $p_n = -$ and $q_n = +$, and hence $\bar p \prec \bar q$. If $p_0 \dots p_{n-1}$ is odd, $P < Q$ if and only if $P^n$ lies in $D^r$ and $Q^n$ lies in $D^l$, which is equivalent to $p_n = +$ and $q_n = -$, and again $\bar p \prec \bar q$.
	\end{proof}
	
	From Definition \ref{df:gplo} and Lemma~\ref{lem:o1} we get immediately the following result.
	
	\begin{lem}\label{lem:o2}
		Let $P, Q \in W^u$ be two different points and let $\bar p, \bar q$, respectively, be their itineraries. Then $P < Q$ if and only if $\bar p \prec \bar q$.
	\end{lem}
	
	Recall, we call the points in $F(\mfc) = \{ z^1_i : i \in \Z \}$ the \emph{turning points}. They play a special role in what follows. Their itineraries we call \emph{kneading sequences} and denote them $\bar k^i$. More precisely, for each $i \in \Z$ the itinerary $\bar k^i$ of the $i$th turning point $z^1_i$ is a kneading sequence. Let $$\mathfrak{K}_F = \{ \bar k^i : i \in \Z \}$$ be the set of all kneading sequences of $F$. Similarly, as for interval maps, $\mathfrak{K}_F$ contains the information about many dynamical properties of $F$. We call $\mathfrak{K}_F$ the \emph{kneading set} of $F$.
	
	Strictly speaking, a turning point $z^1_i$ has two itineraries. They are of the form $$\lpinf w^i \pm \cdot \ra k^i_{\h 0},$$ where $w^i$ is the arc-code of the basic arc containing $z_i^{-1}$. Here for $\pm$ you can substitute any of $+$ and $-$. Therefore we can think of this kneading sequence as a pair $(w^i, \ra k^i_{\h  0})$.
	
	While $\mathfrak{K}_F$ is only a set, by Lemma \ref{lem:arc-codes}, we can recover the order in it by looking at the arc-code parts of the kneading sequences. Moreover, $\bar k^0$ is the only kneading sequence with the arc-code part empty. Thus, given an element $\bar k$ of $\mfk_F$ we can determine $i$ such that $\bar k=\bar k^i$.

	\begin{lem}\label{lem:plek}
		Let $z_m^1$ be a turning point, $\ell^j \subset W^u$ a leaf that contains $z_m^1$, and $P \in \ell^j$ any point. Then $\ra p_{\h 0} \preceq \ra k^m_{\h 0}$.
	\end{lem}
	\begin{proof}
		First, note that $z_m^1$ lies in $D^r$ and it is a boundary point of $\ell^j$. Moreover, it is on the right of any other point $P \in \ell^j$. If $P$ lies in $D^l$, then $p_0 = -$, $k_0^m = +$ and hence $\ra p_{\h 0} \prec \ra k_{\h 0}^m$.
		
		Let us suppose that $P$ lies in $D^r$ and $P \ne z_m^1$. Then $P$ and $z_m^1$ lie in the same basic arc $I_w$ for some finite word $w$, and $\la p_{\h 0} = \la k^m_{\h 0} = \lpinf w$. Therefore, there exists the smallest $n \in \N$ such that $p_n \ne k^m_n$. If $p_0 \dots p_{n-1}$ is even, then $P^n$ is on the left of $z_m^{n+1}$ and hence $P^n$ lies in $D^l$, so $p_n = -$ and $k^m_n = +$. If $p_0 \dots p_{n-1}$ is odd, then $P^n$ is on the right of $z_m^{n+1}$ and hence $P^n$ lies in $D^r$, so $p_n = +$ and $k^m_n = -$. In any case $\ra p_{\h 0} \prec \ra k^m_{\h 0}$. 
		
		If $P = z_m^1$, the statement follows obviously.
	\end{proof}
	
	In the same way, one can prove a more general statement that for any two points $P, Q$ that lie on the same leaf, $P$ is on the left of $Q$ if and only if $\ra p_{\h 0} \prec \ra q_{\h 0}$.
	
	In order to mimic the kneading theory for unimodal interval maps, we would like to have a straightforward characterization of all admissible sequences by the kneading set.
	
	First, we characterize $W^u$-admissible sequences by the kneading set. Recall that itineraries of all points of $W^u$ start with $\lpinf$. The next thing that simplifies our task is that $W^u$ is invariant for $F$, so the set of all $W^u$-admissible sequences is invariant for $\sigma$. This means that apart from the sequence $\lpinf\cdot\rpinf$ (which is  $W^u$-admissible, because it is the itinerary of $X$), we only need a tool for checking $W^u$-admissibility of sequences of the form $\lpinf \cdot p_0p_1p_2 \dots = \lpinf \cdot \ra p_{\h 0}$, such that $p_0=-$.
	
	\begin{proof}[Proof of Theorem \ref{admis}]
		Let $\bar p = \lpinf \cdot \ra p_{\h 0}$ and $p_0 = -$. Let us assume that $\bar p$ is $W^u$-admissible. Then there is a point $P \in W^u$ such that $\iota(\bar p) = P$. Let $z^1$ be a turnig point whose itinerary $\bar k = \lpinf w \pm \cdot \ra k_{\h 0}$ satisfies $w = p_0 \dots p_m$ for some $m$. Then  $F^{m+2}(P)$ and $z^1$ lie in the same leaf which contains the basic arc $I_{p_0 \dots p_{m+2}}$ and by Lemma \ref{lem:plek} $\sigma^{m+2} (\ra p_{\h 0}) \preceq \ra k_{\h 0}$.
		
		Let us suppose now that a sequence $\bar p = \lpinf \cdot \ra p_{\h 0}$ with $p_0 = -$ satisfies that for every kneading sequence $\bar k = \lpinf w \pm \cdot \ra k_{\h 0}$ such that $w = p_0 \dots p_m$ for some $m$, we have $\sigma^{m+2} (\ra p_{\h 0}) \preceq \ra k_{\h 0}$. We want to prove that $\bar p$ is $W^u$-admissible. 
		
		Let us assume, by contradiction, that $\bar p$ is not $W^u$-admissible. Since $\la p_{\h 1} = \lpinf \cdot -$ is a $W^u$-admissible left tail, that is, there are $W^u$-admissible sequences $\bar q$ such that $\la q_{\h 1} = \la p_{\h 1}$, there exists $n \in \N$ such that $\la p_{\h n} = \lpinf \cdot p_0 \dots p_n$ is $W^u$-admissible and $\la p_{\h n+1} = \lpinf \cdot p_0 \dots p_np_{n+1}$ is not $W^u$-admissible.
		
		Let us denote by $\widehat p_{n+1}$ the sign opposite to $p_{n+1}$. The left tail $\lpinf \cdot p_0 \dots p_n \widehat p_{n+1}$ is $W^u$-admissible and $J = I_{p_0 \dots  p_n \widehat p_{n+1}}$ is a basic arc. Since $\lpinf \cdot p_0 \dots p_n p_{n+1}$ is not $W^u$-admissible, the basic arc $I_{p_0 \dots  p_n p_{n+1}}$ does not exist, so $J$ is a leaf and its endpoints are turning or post-turning points.
		
		Consider the endpoint which is closer to the critical curve. It is of the form $F^i(z^1)$, where $z^1$ is a turning point and $i \in \N$. The kneading sequence of $z^1$ is $\lpinf w \pm \cdot \ra k_{\h 0}$, with $w = p_0 \dots p_m$, where $m = n-i-1$. Thus, by the assumption, $\sigma^{n-i+1}(\ra p_{\h 0}) \preceq \ra k_{\h 0}$.
		
		The point $z^1$ is a turning point, so it is the right endpoint of $F^{-i}(J)$. If the number of $+$s among $p_{n-i+1}, p_{n-i+2}, \dots, p_n$ is even, then $F^i(z^1)$ is the right endpoint of $J$, so $J$ is in the left half-plane. This means that $\widehat p_{n+1} = -$, so $p_{n+1} = +$. Both sequences $\sigma^{n-i+1}(\ra p_{\h 0})$ and $\ra k_{\h 0}$ start with $p_{n-i+1} p_{n-i+2} \dots p_n$. Then in $\sigma^{n-i+1}(\ra p_{\h 0})$ we have $p_{n+1} = +$, while in $\ra k_{\h 0}$ we have $\widehat p_{n+1}=-$. But this means that $\ra k_{\h 0} \prec \sigma^{n-i+1}(\ra p_{\h 0})$, a contradiction.
		
		Similarly, if the number of $+$s among $p_{n-i+1}, p_{n-i+2}, \dots, p_n$ is odd, then $F^i(z^1)$ is the left endpoint of $J$, so $J$ is in the right half-plane. This means that $\widehat p_{n+1} = +$, so $p_{n+1} = -$. Both sequences $\sigma^{n-i+1}(\ra p_{\h 0})$ and $\ra k_{\h 0}$ start with $p_{n-i+1} p_{n-i+2} \dots p_n$. Then in $\sigma^{n-i+1}(\ra p_{\h 0})$ we have $p_{n+1} = -$, while in $\ra k_{\h 0}$ we have $\widehat p_{n+1} = +$. But this means that $\ra k_{\h 0} \prec \sigma^{n-i+1}(\ra p_{\h 0})$, a contradiction.
		
		In both cases, we got a contradiction, so $\lpinf \cdot \ra p_{\h 0}$ is $W^u$-admissible.
	\end{proof}
	
	Now, that we know which sequences are  $W^u$-admissible, and since the symbolic space is equipped with the product topology, Theorem \ref{esadmis} holds.
	
	As already stated in the Introduction, Theorems \ref{admis} and \ref{esadmis} prove the pruning front conjecture. Indeed, the kneading set $\mfk_F$ is the pruning front and by the mentioned theorems we have that all the other disallowed regions of the symbol plane are obtained by backward and forward iterations of the primary pruned region, which is the statement of the pruning front conjecture in \cite{CGP}.
	
	% \begin{theorem}\label{esadmis}
		% A sequence $\bar p$ is essential admissible if and only if for each $n \in \N$ there is a $W^u$-admissible sequence $\bar q$ such that $p_{-n} \dots p_n = q_{-n} \dots q_n$. 
		% \end{theorem}
	
	Let us consider two H\'enon maps, $F_1$ and $F_2$. Let $\Lambda_i$ denote the attractor of $F_i$, $i = 1, 2$. Analogously, let all 'items' related to the map $F_i$ be denoted by the index $i \in \{ 1, 2 \}$, unless stated otherwise. Note that by Theorems \ref{admis} and \ref{esadmis}, $\mfk_1 = \mfk_2$ implies $\Sigma_{W^u_1} = \Sigma_{W^u_2}$ and $\Sigma_{F_1} = \Sigma_{F_2}$. We want to prove that for any two itineraries $\bar p, \bar q \in \Sigma_{F_1}$, if $\iota_1 (\bar p) = \iota_1 (\bar q)$, then $\iota_2 (\bar p) = \iota_2 (\bar q)$. 
	
	\begin{prop}\label{lem:dif1}
		Let $F_1$ and $F_2$ be two H\'enon maps such that $\mfk_1 = \mfk_2$. Let $\bar p, \bar q \in \Sigma_{F_1}$ be two different elements such that $\iota_1(\bar p) = \iota_1(\bar q)$. Then $\iota_2(\bar p) = \iota_2(\bar q)$.
	\end{prop}
	\begin{proof}
		Recall that if $\bar p, \bar q$ are two different elements such that $\iota_1(\bar p) = \iota_1(\bar q)$, then $\bar p$ and $\bar q$ disagree at only one coordinate and they are itineraries of a (basic) critical point, or of some of its (pre)images, so there exists only one $j \in \Z$ with $p_j \ne q_j$ ($p_i = q_i$ for all integers $i \ne j$)
		
		If $\bar p, \bar q \in \Sigma_{W^u_1}$ , then $\mfk_1 = \mfk_2$ implies $\iota_2(\bar p) = \iota_2(\bar q)$.
		
		Let $\iota_1(\bar p) = P \in \Lambda_1 \setminus W^u_1$, and hence $\bar p, \bar q \in \Sigma_{F_1} \setminus \Sigma_{W^u_1}$.  Without loss of generality we assume that $j = 0$, $p_0 = -$, $q_0 = +$, so $P \in \C$.
		
		We want to choose two sequences $(\bar p^i)_i$, $(\bar q^i)_i$, $\bar p^i, \bar q^i \in \Sigma_{W^u_1}$ for every $i \in \N$, such that $(\bar p^i)_i$ converges to $\bar p$ and $(\bar q^i)_i$ converges to $\bar q$. Let $P_i = \iota_1(\bar p^i) \in W^u_1$ and $Q_i = \iota_1(\bar q^i) \in W^u_1$ for every $i \in \N$. Since $\iota_1$ is continuous, both sequences $(P_i)_i$ and $(Q_i)_i$ converge to $P$. Recall, $\iota_2(\bar p^i), \iota_2(\bar q^i) \in W^u_2$ for every $i \in \N$. We additionally want that our sequences $(\bar p^i)_i$, $(\bar q^i)_i$ are chosen so that we are able to show that both sequences $(\iota_2(\bar p^i))_i$ and $(\iota_2(\bar q^i))_i$ converge to the same point. We do that in a few steps. 
		
		Since $P \in \K$, there exists a sequence of basic critical points $(z_{n_i})_i$ that converges to $P$. Therefore, there are sequences $(\bar p^i)_{i=1}^\infty, (\bar q^i)_{i=1}^\infty\subset \Sigma_{W^u_1}$, such that $(\bar p^i)_i$ converges to $\bar p$, $(\bar q^i)_i$ converges to $\bar q$, $\la p^i_{\h -1} = \la q^i_{\h -1}$, $p^i_0 = -$ and $q^i_0 = +$ for every $i \in \N$. Recall, $p_0 = -$, $q_0 = +$. Also, $\la p^i_{\h -1} = \la q^i_{\h -1}$ implies that $P_i$ and $Q_i$ lie on the same leaf of $W^u_1$, and $p^i_0 = -$, $q^i_0 = +$ implies that $P_i$ is on the left of $\K$ and $Q_i$ is on the right of $\K$.
		
		Every basic critical point $z_n$ has two itineraries, $\lpinf w^n \cdot - \ra k^n_{\h 0}$ and $\lpinf w^n \cdot + \ra k^n_{\h 0}$. Let $z_{n_i}$ be a critical point between $P_i$ and $Q_i$. Since the sequence  $(d(P_i, z_{n_i}))_i$ converges to zero, and itineraries $\lpinf w^n \cdot - \ra k^n_{\h 0}$ are not isolated, there exists a subsequence of $(\lpinf w^{n_i} \cdot - \ra k^{n_i}_{\h 0})_i$ which converges to $\bar p$ and, in the same way, there exists a subsequence of $(\lpinf w^{n_i} \cdot + \ra k^{n_i}_{\h 0})_i$ which converges to $\bar q$. Since $\la p_{\h -1} = \la q_{\h -1}$ and $\ra p_{\h 1} = \ra q_{\h 1}$ we can choose these subsequences so that $\lpinf w^n \cdot - \ra k^n_{\h 0}$ is an element of one subsequence if and only if $\lpinf w^n \cdot + \ra k^n_{\h 0}$ is an element of the other subsequence.
		
		For simplicity, and without loss of generality, we assume that the sequence $(\lpinf w^{n_i} \cdot - \ra k^{n_i}_{\h 0})_i$ converges to $\bar p$ and the sequence $(\lpinf w^{n_i} \cdot + \ra k^{n_i}_{\h 0})_i$ converges to $\bar q$. Since $\mfk_1 = \mfk_2$ and $\sigma (\lpinf w^{n_i} \cdot - \ra k^{n_i}_{\h 0}), \, \sigma (\lpinf w^{n_i} \cdot + \ra k^{n_i}_{\h 0}) \in \mfk_1$ for every $i \in \N$, it follows that $\iota_2 (\lpinf w^{n_i} \cdot - \ra k^{n_i}_{\h 0}) = \iota_2 (\lpinf w^{n_i} \cdot + \ra k^{n_i}_{\h 0}) = z'_{n_i}$ for every $i \in \N$, where $z'_{n_i}$ is a critical point of $F_2$. Moreover, continuity of $\iota_2$ implies that  the sequence $(z'_{n_i})_i$ converges to $\iota_2(\bar p)$ and $\iota_2(\bar q)$, giving us the desired conclusion that $\iota_2(\bar p) = \iota_2(\bar q)$.
	\end{proof}
	
	For the H\'enon map $F$ let us define an equivalence relation $\sim_F$ on the space $\Sigma_F$ as follows: $\bar p, \bar q \in \Sigma_F$ are equivalent, $\bar p \sim_F \bar q$, if and only if $\iota_F(\bar p) = \iota_F(\bar q)$. Let $\tilde \iota_F : \Sigma_F/_{\sim_F} \to \Lambda_F$ be defined in a natural way, $\tilde \iota_F([\bar p]) = \iota_F(\bar p)$, for every $[\bar p] \in \Sigma_F/_{\sim_F}$. By Proposition \ref{lem:dif1}, $\tilde \iota_F$ is well defined and it is a conjugacy between $\tilde \sigma : \Sigma_F/_{\sim_F} \to \Sigma_F/_{\sim_F}$, where $\tilde \sigma ([\bar p]) = [\sigma \bar p]$, and $F : \Lambda_F \to \Lambda_F$, so Theorem \ref{oneway} holds.

	\section{Folding patterns and pruned trees}\label{sec:fp}
	
	\subsection{Folding pattern}
	
	In this subsection, we introduce a new notion for the H\'enon maps, a bi-infinite sequence of two symbols that is an invariant of the topological conjugacy classes of the H\'enon maps within the Wang-Young parameter set, and that characterizes the set of itineraries of such H\'enon maps.
	
	Recall, $\phi : \mathbb{R} \to W^u$ with $\phi(0) = X$ and $z_0 \in \phi((-\infty, 0))$ is a parametrization of $W^u$. We require additionally, let $\phi(\Z \setminus \{ 0 \}) = \bigcup_{i \in \N_0} F^i(\mfc)$, so $\phi(k)$ is a basic point for every $k \in \Z \setminus \{ 0 \}$, in particular $\phi(-1) = z_0$ and $\phi(1) = z_0^1$. 
	
	Let us consider the bi-infinite sequence $(\phi(k))_{k \in \Z}$. It contains all the information on how $W^u$ is folded in the plane. We will simplify this sequence while still keeping all the information that it provides: Replace each basic critical point $z_i$, $i \in \Z$, by the symbol $0$ and each basic post-critical point $z_i^j$, $i \in \Z$, $j \in \N$, by $1$. Replace also $X$ with the decimal point. We get a sequence like this:
	$$\dots 1 \, 0 \, 1 \, 0 \, 1 \, 0 \cdot  1 \, 0 \, 1 \, 0 \, 1 \, 1 \, 1 \, 0 \, 1 \dots.$$
	We call this sequence of two symbols, $0$ and $1$, the \emph{folding pattern} of $F$.
	
	Let us show that the folding pattern carries the same information as the sequence $(\phi(k))_{k \in \Z}$. We know that $F$ restricted to $W^u$ is an orientation-reversing homeomorphism that fixes $X$. Moreover, it maps the set of basic points bijectively onto the set of basic post-critical points. Thus, we know which symbol of the folding pattern is mapped to which one, see Figure~\ref{fig:fpmap}.
	\begin{figure}[h]
		\unitlength=10mm
		\begin{picture}(20,3.6)(3,4.4)
			
			\put(8.5,6){\circle*{0.1}}
			%\put(8.3,5.3){$X$}
			
			\put(3,6){\line(1,0){14}}
			
			\put(3,5.9){\line(0,1){0.2}}\put(2.8,5.3){$1$}
			\put(4,5.9){\line(0,1){0.2}}\put(3.8,5.3){$0$}
			\put(5,5.9){\line(0,1){0.2}}\put(4.8,5.3){$1$}
			\put(6,5.9){\line(0,1){0.2}}\put(5.8,5.3){$0$}
			\put(7,5.9){\line(0,1){0.2}}\put(6.8,5.3){$1$}
			\put(8,5.9){\line(0,1){0.2}}\put(7.8,5.3){$0$}
			\put(9,5.9){\line(0,1){0.2}}\put(8.8,5.3){$1$}
			\put(10,5.9){\line(0,1){0.2}}\put(9.8,5.3){$0$}
			\put(11,5.9){\line(0,1){0.2}}\put(10.8,5.3){$1$}
			\put(12,5.9){\line(0,1){0.2}}\put(11.8,5.3){$0$}
			\put(13,5.9){\line(0,1){0.2}}\put(12.8,5.3){$1$}
			\put(14,5.9){\line(0,1){0.2}}\put(13.8,5.3){$1$}
			\put(15,5.9){\line(0,1){0.2}}\put(14.8,5.3){$1$}
			\put(16,5.9){\line(0,1){0.2}}\put(15.8,5.3){$0$}
			\put(17,5.9){\line(0,1){0.2}}\put(16.8,5.3){$1$}
			
			\put(9,5){\line(0,-1){0.3}}
			\put(9,4.7){\line(-1,0){2}}
			\put(7,4.7){\vector(0,1){0.3}}
			
			\put(10,5){\line(0,-1){0.5}}
			\put(10,4.5){\line(-1,0){5}}
			\put(5,4.5){\vector(0,1){0.5}}
			
			\put(11,5){\line(0,-1){0.7}}
			\put(11,4.3){\line(-1,0){8}}
			\put(3,4.3){\vector(0,1){0.7}}
			
			\put(8,6.5){\line(0,1){0.3}}
			\put(8,6.8){\line(1,0){1}}
			\put(9,6.8){\vector(0,-1){0.3}}
			
			\put(7,6.5){\line(0,1){0.5}}
			\put(7,7){\line(1,0){4}}
			\put(11,7){\vector(0,-1){0.5}}
			
			\put(6,6.5){\line(0,1){0.7}}
			\put(6,7.2){\line(1,0){7}}
			\put(13,7.2){\vector(0,-1){0.7}}
			
			\put(5,6.5){\line(0,1){0.9}}
			\put(5,7.4){\line(1,0){9}}
			\put(14,7.4){\vector(0,-1){0.9}}
			
			\put(4,6.5){\line(0,1){1.1}}
			\put(4,7.6){\line(1,0){11}}
			\put(15,7.6){\vector(0,-1){1.1}}
			
			\put(3,6.5){\line(0,1){1.3}}
			\put(3,7.8){\line(1,0){14}}
			\put(17,7.8){\vector(0,-1){1.3}}
			
		\end{picture}
		
		\caption{The action of the map on the folding pattern.}
		\label{fig:fpmap}
	\end{figure}
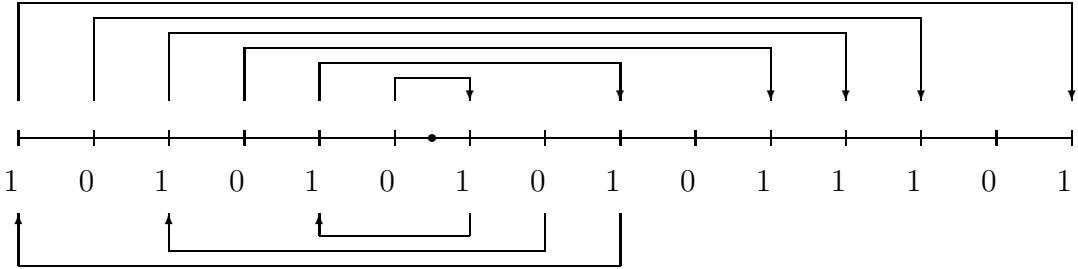
	We also know how to number the critical points (the first to the left of $X$ is $z_0$). This, plus the information about the action of the map, tells us which basic post-critical point corresponds to a given symbol $1$. Thus, we get $0$s and $1$s with subscripts and (some of them) superscripts, like in Figure \ref{fig:fpmapsigns}.
	
	Another information we can read from the folding pattern is which basic post-critical points and basic arcs are in $D^l$ or $D^r$. Namely, we know that the sign (which we use for the itineraries) changes at every symbol $0$. Thus, we can append our folding pattern with those signs and get a sequence like this:
	$$\dots -1-0+1+0-1-0+\cdot+1+0-1-0+1+1+1+0-1- \dots\ .$$
	For each symbol $1$ the signs adjacent to it from the left and right are the same, so we can say that this is the sign of this $1$. 
	
	Of course, we can put some of the additional information together, for instance, we can add to the folding pattern the map, the signs, the subscripts, and the superscripts (on the basic post-critical points), see Figure~\ref{fig:fpmapsigns}.
	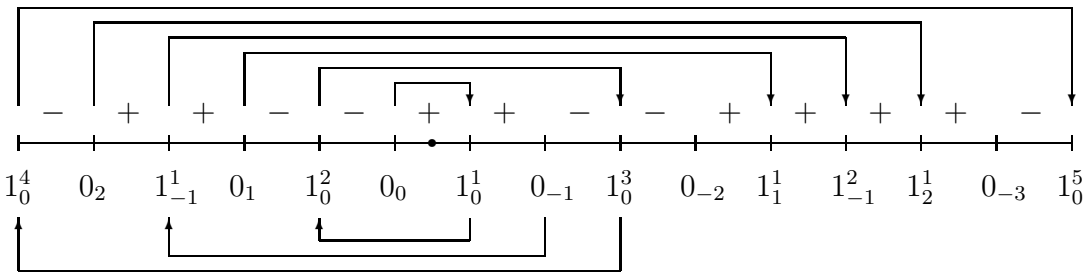
\begin{figure}[h]
		\unitlength=10mm
		\begin{picture}(20,3.6)(3,4.4)
			
			\put(8.5,6){\circle*{0.1}}
			%\put(8.3,5.3){$X$}
			
			\put(3,6){\line(1,0){14}}
			
			\put(3,5.9){\line(0,1){0.2}}\put(2.8,5.3){$1_0^4$}
			\put(4,5.9){\line(0,1){0.2}}\put(3.8,5.3){$0_2$}
			\put(5,5.9){\line(0,1){0.2}}\put(4.8,5.3){$1_{-1}^1$}
			\put(6,5.9){\line(0,1){0.2}}\put(5.8,5.3){$0_1$}
			\put(7,5.9){\line(0,1){0.2}}\put(6.8,5.3){$1_0^2$}
			\put(8,5.9){\line(0,1){0.2}}\put(7.8,5.3){$0_0$}
			\put(9,5.9){\line(0,1){0.2}}\put(8.8,5.3){$1_0^1$}
			\put(10,5.9){\line(0,1){0.2}}\put(9.8,5.3){$0_{-1}$}
			\put(11,5.9){\line(0,1){0.2}}\put(10.8,5.3){$1_0^3$}
			\put(12,5.9){\line(0,1){0.2}}\put(11.8,5.3){$0_{-2}$}
			\put(13,5.9){\line(0,1){0.2}}\put(12.8,5.3){$1_1^1$}
			\put(14,5.9){\line(0,1){0.2}}\put(13.8,5.3){$1_{-1}^2$}
			\put(15,5.9){\line(0,1){0.2}}\put(14.8,5.3){$1_2^1$}
			\put(16,5.9){\line(0,1){0.2}}\put(15.8,5.3){$0_{-3}$}
			\put(17,5.9){\line(0,1){0.2}}\put(16.8,5.3){$1_0^5$}
			
			\put(9,5){\line(0,-1){0.3}}
			\put(9,4.7){\line(-1,0){2}}
			\put(7,4.7){\vector(0,1){0.3}}
			
			\put(10,5){\line(0,-1){0.5}}
			\put(10,4.5){\line(-1,0){5}}
			\put(5,4.5){\vector(0,1){0.5}}
			
			\put(11,5){\line(0,-1){0.7}}
			\put(11,4.3){\line(-1,0){8}}
			\put(3,4.3){\vector(0,1){0.7}}
			
			\put(8,6.5){\line(0,1){0.3}}
			\put(8,6.8){\line(1,0){1}}
			\put(9,6.8){\vector(0,-1){0.3}}
			
			\put(7,6.5){\line(0,1){0.5}}
			\put(7,7){\line(1,0){4}}
			\put(11,7){\vector(0,-1){0.5}}
			
			\put(6,6.5){\line(0,1){0.7}}
			\put(6,7.2){\line(1,0){7}}
			\put(13,7.2){\vector(0,-1){0.7}}
			
			\put(5,6.5){\line(0,1){0.9}}
			\put(5,7.4){\line(1,0){9}}
			\put(14,7.4){\vector(0,-1){0.9}}
			
			\put(4,6.5){\line(0,1){1.1}}
			\put(4,7.6){\line(1,0){11}}
			\put(15,7.6){\vector(0,-1){1.1}}
			
			\put(3,6.5){\line(0,1){1.3}}
			\put(3,7.8){\line(1,0){14}}
			\put(17,7.8){\vector(0,-1){1.3}}
			
			\put(3.3,6.3){$-$}
			\put(4.3,6.3){$+$}
			\put(5.3,6.3){$+$}
			\put(6.3,6.3){$-$}
			\put(7.3,6.3){$-$}
			\put(8.3,6.3){$+$}
			\put(9.3,6.3){$+$}
			\put(10.3,6.3){$-$}
			\put(11.3,6.3){$-$}
			\put(12.3,6.3){$+$}
			\put(13.3,6.3){$+$}
			\put(14.3,6.3){$+$}
			\put(15.3,6.3){$+$}
			\put(16.3,6.3){$-$}
			
		\end{picture}
		
		\caption{The action of the map on the folding pattern with
			signs.}
		\label{fig:fpmapsigns}
	\end{figure}
	If we replace now in the sequence in Figure~\ref{fig:fpmapsigns} all $0$s and $1$s with $z$s and add $X$ in the place of the decimal point, we will get the sequence $(\phi(k))_{k \in \Z}$. Using $+$s and $-$s we can reconstruct $W^u$ and how it is folded in the plane, like in Figure \ref{fig:FP}.
	
	\begin{theorem}\label{equivalent}
		The set of kneading sequences and the folding pattern are equivalent, that is, given one of them, we can recover the other one.
	\end{theorem}
	\begin{proof} 
		Suppose we know the set of the kneading sequences and we want to recover the folding pattern. As we noticed, when we defined the kneading set, we know which kneading sequence is the itinerary of which point $z^1_n$. We proceed by induction. First, we know that in $[z_0^2, z_0^1]^u$ there are three basic points (and $X$), and that they should be marked from the left to the right $1 \, 0 \cdot 1$. We also know how they are mapped by $F$. Now suppose that we know the basic points in $F^n([z_0^2, z_0^1]^u) = [z_0^{n+2},z_0^{n+1}]^u$, how they are marked, and how they are mapped by $F$. Some of those points (on the left or on the right, depending on the parity of $n$) are not mapped to the points of this set. Then we map them to new points, remembering that $F(X) = X$ and that $F$ is a homeomorphism of $W^u$ reversing orientation. Those new points have to be marked as $1$ because $F$ maps the basic points onto the basic post-critical points. Now we use our information about the kneading sequences. They are the itineraries of the first images of the points marked $0$, and since we know the action of $F$ on the set of the basic points of $[z_0^{n+2},z_0^{n+1}]^u$, this determines the signs of all points marked $1$ in the picture that we have at this moment. We know that the signs change at each point marked $0$, so we insert such a point between every pair of $1$s with opposite signs (clearly, there cannot be two consecutive $0$s). In such a way we get the basic points in $[z_0^{n+2},z_0^{n+1}]^u$, and the information on how they are marked, and how they are mapped by $F$. This completes the induction step.
		
		Now suppose that we know the folding pattern and we want to recover the set of kneading sequences. As we observed, given a folding pattern, we can add to it the information about the signs, plus the information about the action of the map. Thus, we get $0$s and $1$s with subscripts and (some of them) superscripts. The turning points are the symbols $1$ with the superscript 1. Now for every $1$ which is a turning point, we follow the action of the map (the arrows), reading the signs of the symbols on this path. In such a way, we get the corresponding right tail $\ra k_0$ of the kneading sequence (the signs do not change at $1$s, so the sign immediately to the left and to the right of a given $1$ are the same, and moreover, this sign is the same as the sign of the component $D^l$ or $D^r$, where the corresponding post-critical point lies). Going back to recover the arc code $w$ of the kneading sequence is also simple. In the first step, we follow the arrow backward from the symbol $1$ with the superscript 1 to the corresponding symbol $0$. Every symbol $0$ is between two symbols $1$. To recover $w$ we follow the action of the inverse of the map (we follow the arrows of the adjacent $1$s backward), reading the signs of the symbols on this path, until we get to the initial part of the folding pattern $1 \, 0 \cdot 1$. 
	\end{proof}
	
	\subsection{Pruned tree}
	
	We can think of the folding pattern as a countable Markov partition for the map $F$ on  $W^u$. Thus, we can consider the corresponding Markov graph (the graph of transitions). The vertices of this graph are the basic arcs and there is an arrow from $I_u$ to $I_v$ if and only if $F(I_u)\supset I_v$. For simplicity we can index the basic arcs by integers: Let $I_0 = I_\emptyset$, $I_n \subset W^{u+}$, $I_{-n} \subset W^{u-}$, $n \in \N$, and $I_i \cap I_j \ne \emptyset$ if and only if $|i - j| \le 1$. Now, for the vertices, instead of writing the basic arcs $I_n$, we write just the corresponding index $n$ for them.
	
	From the folding pattern shown in Figure~\ref{fig:fpmap} we get the graph shown in Figure~\ref{fig:ftarcs} (of course this tree goes down and is infinite; we are showing only a part of it). This graph is almost a tree; except for $0$ and the arrows beginning at $0$, it is a sub-tree of the full binary tree, so we call it the \emph{pruned tree} of $F$.
	
	\begin{figure}[h]
		\unitlength=10mm
		\begin{picture}(20,6)(3.5,1.5)
			
			\put(10,6){$0$}
			\put(9.8,5){$-1$}
			\put(9,4){$1$}
			\put(11,4){$2$}
			\put(7.8,3){$-2$}
			\put(9.2,3){$-3$}
			\put(10.3,3){$-4$}
			\put(11.7,3){$-5$}
			\put(7,2){$3$}
			\put(8.1,2){$4$}
			\put(9.4,2){$5$}
			\put(10.6,2){$6$}
			\put(11.9,2){$7$}
			\put(13,2){$8$}
			
			\put(9.94,6.4){\line(0,1){0.4}}
			\put(9.94,6.8){\line(1,0){0.3}}
			\put(10.24,6.8){\vector(0,-1){0.4}}
			
			\put(10.1,5.8){\vector(0,-1){0.4}}
			\put(9.8,4.8){\vector(-1,-1){0.4}}
			\put(10.4,4.8){\vector(1,-1){0.4}}
			\put(8.8,3.8){\vector(-1,-1){0.4}}
			\put(11.4,3.8){\vector(1,-1){0.4}}
			\put(9.1,3.8){\vector(1,-1){0.4}}
			\put(11.1,3.8){\vector(-1,-1){0.4}}
			\put(7.8,2.8){\vector(-1,-1){0.4}}
			\put(12.4,2.8){\vector(1,-1){0.4}}
			\put(8.2,2.8){\vector(0,-1){0.4}}
			\put(9.5,2.8){\vector(0,-1){0.4}}
			\put(10.7,2.8){\vector(0,-1){0.4}}
			\put(12,2.8){\vector(0,-1){0.4}}
			
		\end{picture}
		
		\caption{A pruned tree with numbers of basic arcs.}
		\label{fig:ftarcs}
	\end{figure}
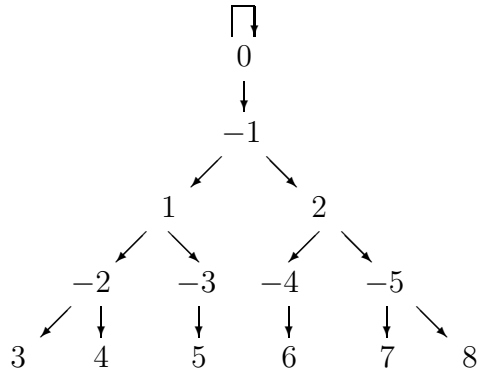
	
	This tree is in a natural way divided into \emph{levels}. The number $0$ is at level $0$, the number $-1$ is at level $1$, and in general, if the path from $-1$ to $i$ has $n$ arrows then $i$ is at level $n+1$. It is easy to see how the levels are arranged. Starting with level $1$, negative numbers are at odd levels, ordered with their moduli increasing from the left to the right. If level $n$ ends with $-i$ then level $n+2$ starts with $-(i+1)$. Positive numbers are at even levels, ordered in a similar way. Therefore, if we have the same tree without the numbers, like in Figure~\ref{fig:ftnaked}, we know where to put which number. Of course, we are talking about the tree embedded in the plane, so the order of the vertices at each level is given.
	
	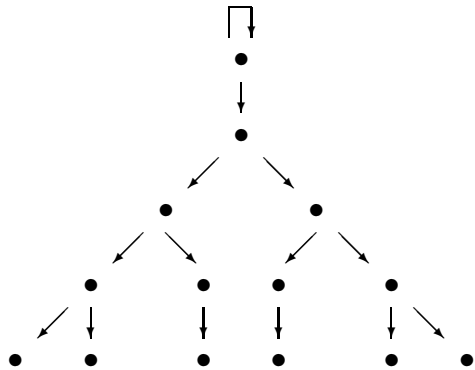
\begin{figure}[h!]
		\unitlength=10mm
		\begin{picture}(20,6)(3.5,1.5)
			
			\put(10,6){$\bullet$}
			\put(10,5){$\bullet$}
			\put(9,4){$\bullet$}
			\put(11,4){$\bullet$}
			\put(8,3){$\bullet$}
			\put(9.5,3){$\bullet$}
			\put(10.5,3){$\bullet$}
			\put(12,3){$\bullet$}
			\put(7,2){$\bullet$}
			\put(8,2){$\bullet$}
			\put(9.5,2){$\bullet$}
			\put(10.5,2){$\bullet$}
			\put(12,2){$\bullet$}
			\put(13,2){$\bullet$}
			
			\put(9.94,6.4){\line(0,1){0.4}}
			\put(9.94,6.8){\line(1,0){0.3}}
			\put(10.24,6.8){\vector(0,-1){0.4}}
			
			\put(10.1,5.8){\vector(0,-1){0.4}}
			\put(9.8,4.8){\vector(-1,-1){0.4}}
			\put(10.4,4.8){\vector(1,-1){0.4}}
			\put(8.8,3.8){\vector(-1,-1){0.4}}
			\put(11.4,3.8){\vector(1,-1){0.4}}
			\put(9.1,3.8){\vector(1,-1){0.4}}
			\put(11.1,3.8){\vector(-1,-1){0.4}}
			\put(7.8,2.8){\vector(-1,-1){0.4}}
			\put(12.4,2.8){\vector(1,-1){0.4}}
			\put(8.1,2.8){\vector(0,-1){0.4}}
			\put(9.6,2.8){\vector(0,-1){0.4}}
			\put(10.6,2.8){\vector(0,-1){0.4}}
			\put(12.1,2.8){\vector(0,-1){0.4}}
			
		\end{picture}
		
		\caption{A ``naked'' pruned tree.}
		\label{fig:ftnaked}
	\end{figure}
	
	In a similar way as for the folding pattern, we can add some information to the picture. The symbols $0$ and $1$ can be placed between the vertices of the tree. The ones that are between the last vertex of level $n$ and the first vertex of level $n+2$, will be placed to the right of the last vertex of level $n$. The only exception is $0_0$, which has to be placed to the left of the unique vertex of level 1, in order to avoid a collision with other symbols.
	
	We know which of the symbols are $0$s. By our construction, $0$s are those basic points that are in the interior of some $F(I_n)$. This means that they are exactly the ones that are between the siblings (vertices where the arrows from the common vertex end). And once we have $0$s and $1$s marked, we can recover the signs of the vertices because we know that the signs change exactly at $0$s. Then we get the pruned tree marked as in Figure \ref{fig:ftall}.
	
	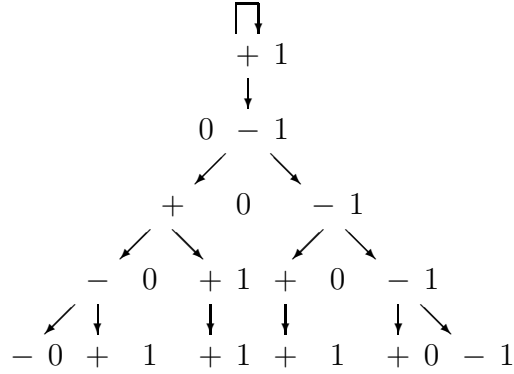
\begin{figure}[h!]
		\unitlength=10mm
		\begin{picture}(20,6)(3.5,1.5)
			
			\put(9.94,6){$+$}\put(10.44,6){$1$}
			\put(9.94,5){$-$}\put(10.44,5){$1$}\put(9.44,5){$0$}
			\put(8.94,4){$+$}\put(9.94,4){$0$}
			\put(10.94,4){$-$}\put(11.44,4){$1$}
			\put(7.94,3){$-$}\put(8.69,3){$0$}
			\put(9.44,3){$+$}\put(9.94,3){$1$}
			\put(10.44,3){$+$}\put(11.19,3){$0$}
			\put(11.94,3){$-$}\put(12.44,3){$1$}
			\put(6.94,2){$-$}\put(7.44,2){$0$}
			\put(7.94,2){$+$}\put(8.69,2){$1$}
			\put(9.44,2){$+$}\put(9.94,2){$1$}
			\put(10.44,2){$+$}\put(11.19,2){$1$}
			\put(11.94,2){$+$}\put(12.44,2){$0$}
			\put(12.94,2){$-$}\put(13.44,2){$1$}
			
			\put(9.94,6.4){\line(0,1){0.4}}
			\put(9.94,6.8){\line(1,0){0.3}}
			\put(10.24,6.8){\vector(0,-1){0.4}}
			
			\put(10.1,5.8){\vector(0,-1){0.4}}
			\put(9.8,4.8){\vector(-1,-1){0.4}}
			\put(10.4,4.8){\vector(1,-1){0.4}}
			\put(8.8,3.8){\vector(-1,-1){0.4}}
			\put(11.4,3.8){\vector(1,-1){0.4}}
			\put(9.1,3.8){\vector(1,-1){0.4}}
			\put(11.1,3.8){\vector(-1,-1){0.4}}
			\put(7.8,2.8){\vector(-1,-1){0.4}}
			\put(12.4,2.8){\vector(1,-1){0.4}}
			\put(8.1,2.8){\vector(0,-1){0.4}}
			\put(9.6,2.8){\vector(0,-1){0.4}}
			\put(10.6,2.8){\vector(0,-1){0.4}}
			\put(12.1,2.8){\vector(0,-1){0.4}}
			
		\end{picture}
		
		\caption{A pruned tree with $0$s, $1$s and signs.}
		\label{fig:ftall}
	\end{figure}
	
	Now that we have our three objects, the kneading sequences, the folding pattern, and the pruned tree, we prove that they carry
	the same information.
	
	\begin{theorem}\label{equivalent-fp-pt}
		The folding pattern and the pruned tree are equivalent, that is, given one of them, we can recover the other one.
	\end{theorem}
	
	\begin{proof} 
		{\it From the folding pattern to the pruned tree.} This is described when we were defining the folding tree.
		
		{\it From the pruned tree to the kneading set.} As we observed, given a pruned tree, we can add to it the information about the signs and the positions of the $0$ and $1$ symbols. The turning points are the symbols $1$ placed directly below $0$s, and additionally, $1_0$ is the only symbol in the zeroth row. Now for every $1$ which is a turning point, we go down along the tree, reading the signs immediately to the left of the symbols (see Figure~\ref{fig:ftall}). In such a way, we get the corresponding right tail of the kneading sequence (recall, the signs do not change at $1$s). Going back (up) is even simpler since in two steps we get to a vertex and just go up the tree along the edges.
		
		{\it From the kneading set to the folding pattern.} This part is proved in Theorem \ref{equivalent}.
	\end{proof}

	\section{The classification of H\'enon attractors}\label{sec:cla}
	
	In this section, we classify (up to topological conjugacy) the H\'enon maps for the Wang-Young parameters. Our proof relies on our recent work in \cite{BS}, where we showed that for each orientation reversing H\'enon-like map $F$ there exists a densely branching tree $\T$, such that $F$ on its strange attractor is conjugate to the shift homeomorphism on the inverse limit of $\T$. Here, for completeness, we briefly recapitulate the construction for the orientation reversing case ($b > 0$), and extend it to the orientation preserving case ($b < 0$). That construction uses the existence of the critical locus $\K$, proved in Section \ref{sec:cl}. 
	
	\subsection{The construction of densely branching trees}\label{ss:prelim2}
	
	In \cite{CP} the notion of mild dissipation was used to show that for $a\in(1,2)$ and $b\in (-\frac{1}{4},0)\cup (0,\frac{1}{4})$ there exists a metric tree $\T$, and a continuous map $f : \T \to \T$ such that the H\'enon map $F$ is semi-conjugate to $f$, on a domain of dissipation $\mathbb{D}=[-\frac{1}{2}-\frac{1}{a},\frac{1}{2}+\frac{1}{a}]\times [-\frac{1}{2}+\frac{a}{4},\frac{1}{2}-\frac{a}{4}]$. In \cite{BS} this result was strengthened for the Wang-Young parameters with $b > 0$ to show that the tree $\T$ is always densely branching and $F$ on its strange attractor $\Lambda_F$ is conjugate to the shift map $\sigma : \invlim (\T, f) \to \invlim (\T, f)$. 
	
	To recall that, we start with the set $\Gamma$, determined by pieces of the stable manifold of the hyperbolic fixed point $X$. Let $$\Gamma = \{ \gamma : \gamma \textrm{ is a connected component of } D \cap W^s \}.$$ It satisfies the following properties.
	
	(A') The elements of $\Gamma$ are pairwise disjoint or coincide.
	
	(B') Every $\gamma \in \Gamma$ is the limit of arcs in $\Gamma$ and is accumulated on both sides.
	
	(C') For $\gamma \in \Gamma$ the connected components of $F^{-1}(\gamma) \cap D$ are elements of $\Gamma$.
	
	One denotes by $\Xi$ the collection of sequences $(s_n)_{n=1}^\infty$ of connected surfaces $s_n$ in $D$ bounded by a finite number of elements in $\Gamma$, such that $\Cl(s_{n+1})\subset s_n$ for each $n$. Set $(s_n)_{n=1}^{\infty}\leq (s'_n)_{n=1}^{\infty}$ if for any $k$ there is $m$ such that 
	$\Cl(s_m)\subset s'_k$. Now let $\Xi_0$ be the collection of sequences that are minimal for the relation $\leq$. One defines $\T$ as the quotient of $\Xi_0$ by the relation $\equiv$ defined by 
	$$(s_n)_{n=1}^{\infty}\equiv(s'_n)_{n=1}^\infty\textrm{ if and only if }(s_n)_{n=1}^{\infty}\leq(s'_n)_{n=1}^\infty\textrm{  and }(s'_n)_{n=1}^{\infty}\leq (s_n)_{n=1}^\infty.$$ 
	With $2^D$ standing for the hyperspace of compact subsets of $D$ we define an injection $\psi : \T \to 2^D$ by 
	$$\psi([(s_n)_{n=1}^{\infty}])=\bigcap_{n=1}^\infty s_n,$$ 
	which assigns to each equivalence class $[(s_n)_{n=1}^{\infty}]$ in $\T$ a continuum $\bigcap_{n=1}^\infty s_n$, which is a geometric realization of that class. We let $\A = \psi(\T)$ and note that $\Gamma \subset \A$, but $\A \setminus \Gamma \neq \emptyset$. We let $\tilde{\pi} : \mathcal{A} \to \T$ as the inverse of $\psi$; i.e., $\tilde{\pi} := \psi^{-1}$. There exists a continuous map $f : \T \to \T$ such that the H\'enon map restricted to its attractor $F|_{\Lambda_F}$ is conjugate to the shift map $\sigma : \invlim (\T, f) \to \invlim (\T, f)$, which is also called the natural extension of $f$. That conjugacy $\Pi : \Lambda_F \to \invlim (\T, f)$ is given by the following formula: $\Pi(P) = (\pi(P), \pi(F^{-1}(P)), \dots , \pi(F^{-n}(P)), \dots )$, where $\pi : D \to \T$ is defined by $\pi(x) = t$ if $x \in \alpha \in \mathcal{A}$ and $\tilde{\pi}(\alpha) = t$.
	
	To obtain a similar result for $b < 0$, we slightly modify the region $D$ for that case, and for convenience, we keep the same notation. The boundary of $D$ consists again of an arc of $W^u_Y$, also denoted by $\partial^u D$, and a straight line segment, denoted by $l$, so again $\partial D = \partial^u D \cup l$, but this time we require that $Y \in l \cap \partial^u D$, see Figure \ref{fig:omega2}. It should be clear from the above construction for $b > 0$, that the analogous construction holds also for $b < 0$, but since $Y \in \partial D$ for $b < 0$, the shift map $\sigma : \invlim (\T, f) \to \invlim (\T, f)$ is conjugate to the H\'enon map restricted to the union of its attractor $\Lambda_F$ and $W^{u+}_Y$, i.e., $F|_{\Lambda_F \cup W^{u+}_Y}$. Note that if we did not require that $Y \in \partial D$, then the shift map $\sigma$ would be conjugate to the H\'enon map restricted to its attractor $\Lambda_F$, as in the orientation reversing case, but here we want to include the quasi-critical points since they will be useful for the classification.

	Recall that $\gamma_X \in \Gamma$ denotes a unique element with $X \in \gamma_X$. It has two endpoints. For $b > 0$, one is $X$, and we denote the other one by $E$, $\{ X , E \} = \partial^u D \cap \gamma_X$. Also, $\Omega$ is the closed region bounded by two arcs: $\partial^u \Omega = [X, E^{-1}]^u \subset W^u$ and $\partial^s \Omega = [X, E^{-1}]^s \subset W^s$, see Figure \ref{fig:omega}. Note that $z_0 \in [X, E^{-1}]^u$ and $\K \subset \Omega$. Note also that the closed subregion of $D$ bounded by the arc $[X, E]^u \subset W^u$ and $\gamma_X$ (containing $z_0^1$) is mapped by $F^{-1}$ onto the region $\Omega$. 
	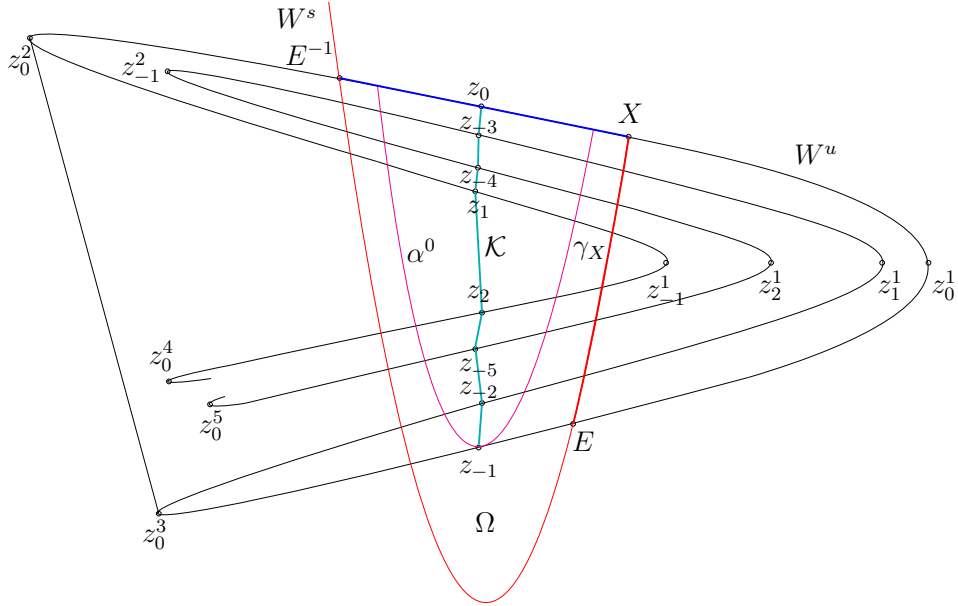
\begin{figure}[h]
		
		\begin{adjustbox}{center}
			\resizebox{1\textwidth}{!}{
				
				\begin{tikzpicture}
					
					\tikzstyle{every node}=[draw, circle, fill=white, minimum size=2pt, inner sep=0pt]
					\tikzstyle{dot}=[circle, fill=white, minimum size=0pt, inner sep=0pt, outer sep=-1pt]
					
					\node (n1) at (5*4/3,0) {};
					\node (n2) at (-6.65,5*2/3)  {};
					\node (n3) at (-4.74,-3.72)  {};
					\node (n4) at (5*5/9,0)  {};
					\node (n5) at (-4.59,-1.76) {};
					\node (n6) at (-3.98,-2.1)  {};
					\node (n7) at (-4.61,2.84)  {};
					\node (n8) at (5*61/51,0)  {};
					\node (n9) at (5*65/75,0)  {};
					
					\node (n10) at (5*4/9,1.87)  {};
					
					\node (n21) at (0.04,2.32)  {};
					\node (n22) at (-0.05,1.06)  {};
					\node (n23) at (0.05,-0.74)  {};
					\node (n24) at (-0.05,-1.28)  {};
					\node (n25) at (0.05,-2.08)  {};
					\node (n26) at (0,-2.74)  {};
					
					\node (n27) at (0,1.89)  {};
					\node (n28) at (-0.01,1.41)  {};
					
					\node (n11) at (-2.06,2.74)  {};
					\node (n12) at (1.4,-2.39)  {};
					
					\draw (-6.65,5*2/3)--(-4.74,-3.72);
					
					\draw[TealBlue,thick](-0.05,1.06)--(0.05,-0.74);
					\draw[TealBlue,thick](0.05,-0.74)--(-0.05,-1.25);
					\draw[TealBlue,thick](-0.05,-1.25)--(0.05,-2.08);
					\draw[TealBlue,thick](n25)--(n26);
					\draw[TealBlue,thick](n27)--(n28);
					\draw[TealBlue,thick](n21)--(n27);
					\draw[TealBlue,thick](n28)--(n22);
					
					\draw (3,1.7) .. controls (7.7,0.7) and  (7.7,-0.6).. (4,-1.7);	
					\draw (3,1.7) .. controls (-9,4.25) and (-9,3.7) .. (-1.5,1.5);
					\draw (0,-2.1) .. controls (-6.9,-4.1) and (-6.9,-4.6) .. (4,-1.7);
					\draw (2.8,1.2) .. controls (7.44,0) and (7.44,-0.1) .. (0,-2.1);
					\draw (2.8,1.2) .. controls (-7,3.7) and (-7,3.2) .. (2.3,0.8);
					\draw (2.3,0.8) .. controls (5.63,-0.2) and (5.63,0.1) .. (-5*19/24+0.5,-5*5/12);
					\draw (-1.5,1.5) .. controls (4.55,-0.4) and (4.55,0.3) .. (-5*5/6+0.5,-5*1/3+0.15);
					\draw (-5*5/6+0.5,-5*1/3+0.15) .. controls (-5*5/6-0.7,-5*1/3-0.08) and (-5*5/6-0.7,-5*1/3-0.18) .. (-5*5/6+0.2,-5*1/3-0.05);
					\draw (-5*19/24+0.2,-5*5/12+0.1) .. controls (-5*19/24-0.15,-5*5/12) and (-5*19/24-0.15,-5*5/12-0.1) .. (-5*19/24+0.5,-5*5/12);
					
					\draw[red] (5*4/9,1.87) .. controls (0.7,-7.7) and  (-0.7,-7.6) .. (-5*4/9,3.87);	
					\draw[red,thick] (5*4/9,1.87) .. controls (2.2,1.55) and  (1.62,-1.55) .. (1.4,-2.39);	
					
					\draw[magenta] (1.7,1.97) .. controls (0.5,-4.4) and  (-0.7,-4.4) .. (-1.5,2.64);	
					
					\draw[blue,thick] (5*4/9,1.87) -- (-2.06,2.74);
					
					\node[dot, draw=none, label=above: $z^1_0$] at (6.9,0) {};
					\node[dot, draw=none, label=above: $z_0$] at (0,2.8) {};
					\node[dot, draw=none, label=above: $z^2_0$] at (-6.8,5*2/3) {};
					\node[dot, draw=none, label=above: $z_1$] at (0,1.05) {};
					\node[dot, draw=none, label=above: $z^1_{-1}$] at (5*5/9,0) {};
					\node[dot, draw=none, label=above: $z_2$] at (0,-0.2) {};
					\node[dot, draw=none, label=above: $z^4_0$] at (-4.7,-1.05) {};
					\node[dot, draw=none, label=above: $z^5_0$] at (-5*19/24,-5*5/12) {};
					\node[dot, draw=none, label=above: $z_{-5}$] at (0,-1.2) {};
					\node[dot, draw=none, label=above: $z_{-3}$] at (0,2.4) {};
					\node[dot, draw=none, label=above: $z_{-4}$] at (0,1.6) {};
					\node[dot, draw=none, label=above: $z^1_2$] at (5*65/75,0) {};
					\node[dot, draw=none, label=above: $z^2_{-1}$] at (-5,3.3) {};
					\node[dot, draw=none, label=above: $z^1_1$] at (6.1,0) {};
					\node[dot, draw=none, label=above: $z^3_0$] at (-4.8,-3.7) {};
					\node[dot, draw=none, label=above: $z_{-1}$] at (0,-2.7) {};
					\node[dot, draw=none, label=above: $z_{-2}$] at (0,-1.53) {};
					\node[dot, draw=none, label=above: $X$] at (2.25,2.5) {};
					\node[dot, draw=none, label=above: $W^u$] at (5,2) {};
					\node[dot, draw=none, label=above: $\Omega$] at (0.1,-3.6) {};
					\node[dot, draw=none, label=above: $\K$] at (0.25,0.5) {};
					\node[dot, draw=none, label=above: $\gamma_X$] at (1.65,0.5) {};
					\node[dot, draw=none, label=above: $\alpha^0$] at (-0.85,0.5) {};
					\node[dot, draw=none, label=above: $W^s$] at (-2.7,4) {};
					\node[dot, draw=none, label=above: $E^{-1}$] at (-2.5,3.55)  {};
					\node[dot, draw=none, label=above: $E$] at (1.55,-2.39)  {};
					
				\end{tikzpicture}
			}
		\end{adjustbox}
		\vspace{-2cm}
		\caption{The region $\Omega$ for $b > 0$ with $\partial^u \Omega$ in blue and $\partial^s \Omega$ in red. The arc $\alpha^0$ is magenta. It can have various 'shapes', this figure shows only one of them. The critical locus $\K$ is teal.}
		\label{fig:omega}
	\end{figure}
	
	For $b < 0$, $E$ denotes an endpoint of $\gamma_X$ that is closer to $Y$ on $W^u_Y$, and the other one is $E^1$, $\gamma_X = [E, E^1]^s$. Again, $\Omega$ is the closed region bounded by two arcs: $\partial^u \Omega = [E, E^{-1}]^u_Y \subset W^u_Y$ and $\partial^s \Omega = [E, E^{-1}]^s \subset W^s$, see Figure \ref{fig:omega2}. Note that $\K \subset \Omega$, and the closed subregion of $D$ bounded by the arc $[E, E^1]^u_Y \subset W^u_Y$ and $\gamma_X$ is mapped by $F^{-1}$ onto the region $\Omega$. 
	
	\begin{figure}[h]
		
		\resizebox{1\textwidth}{!}{
			\begin{tikzpicture}
				
				\hspace{1.5cm}			
				
				\tikzstyle{every node}=[draw, circle, fill=white, minimum size=2pt, inner sep=0pt]
				\tikzstyle{dot}=[circle, fill=white, minimum size=0pt, inner sep=0pt, outer sep=-1pt]
				
				\node[label=right: $z_2$] at (0.05,3.81)  {};
				\node[label=right: $z_3$] at (-0.01,3.52)  {};
				\node[label=right: $z_{6}$] at (0.04,2.72)  {};
				\node[label=right: $z_{-1}$] at (0.04,2.32)  {};
				\node[label=right: $z_{-2}$] at (0,1.55)  {};
				\node[label=right: $z_{7}$] at (0,1.18)  {};					
				\node[label=right: $z_{-3}$] at (0.05,-1.91)  {};
				\node[label=right: $z_0$] at (0,-2.8)  {};
				\node[label=right: $z_5$] at (0.05,-3.34)  {};
				\node[label=right: $z_4$] at (-0.05,-4.29)  {};
				\node[label=right: $z_1$] at (0,-4.77)  {};
				\node[label=right: $z_{-4}$] at (0,-5.57)  {};							
				
				\node[label=right: $X$] at (1.41,-2.4)  {}; 
				\node[label=below left: $Y$] at (-7,6.88)  {}; 		
				
				\node[label=left: $z'_0$] at (0,6)  {};	
				\node[label=left: $z'_8$] at (0,4.05)  {};	
				\node[label=left: $z'_7$] at (0,3.3)  {};
				\node[label=left: $z'_4$] at (0,2.9)  {};
				\node[label=left: $z'_{11}$] at (0,2.17)  {};	
				\node[label=left: $z'_{12}$] at (0,1.7)  {};	
				\node[label=left: $z'_3$] at (0,0.9)  {};
				\node[label=left: $z'_2$] at (0,-1.25)  {};
				\node[label=left: $z'_{13}$] at (0,-2.1)  {};	
				\node[label=left: $z'_{10}$] at (0,-2.6)  {};	
				\node[label=left: $z'_5$] at (0,-3.6)  {};	
				\node[label=left: $z'_6$] at (0,-4.1)  {};	
				\node[label=left: $z'_9$] at (0,-5)  {};	
				\node[label=left: $z'_{14}$] at (0,-5.4)  {};	
				\node[label=left: $z'_1$] at (0,-6)  {};
				
				\node[label=right: $z^{'2}_0$] at (-5.98,-5.4)  {};
				\node[label=below left: $z^{'3}_0$] at (-3.58,2.8)  {};
				\node[label=below left: $z^{'4}_0$] at (-0.9,3.9)  {};
				\node[label=below left: $z^{'5}_0$] at (1.26,4.48)  {};	
				\node[label=right: $z^{'1}_0$] at (11.25,0.5)  {};
				\node[label=left: $z^{'1}_1$] at (3,-0.1)  {};	
				\node[label=right: $z^{'2}_1$] at (-2.62,-4.56)  {};
				\node[label=left: $z^{'3}_1$] at (-1.79,2.35)  {};
				
				\node[dot, draw=none, label=left: $l$] at (-6.9,0) {};
				\node[dot, draw=none, label=above: $W^u_Y$] at (7,5.2) {};
				\node[dot, draw=none, label=below right: $W^u$] at (2,-3.2) {};
				
				\node[label=above right: $E$] at (2.76,5.55) {};
				\node[label=above left: $E^{-1}$] at (-2.95,6.38) {};
				\node[label=below right: $E^{1}$] at (0.53,-5.9) {};
				
				\node[dot, draw=none, label=left: $\gamma_X$] at (1.2,0) {};
				\node[dot, draw=none, label=left: $\K$] at (-0.6,0) {};
				
				\node[dot, draw=none, label=above: $\Omega$] at (-0.3,-6.4) {};
				\node[dot, draw=none, label=above: $\partial^uD$] at (7,-3.9) {};
				
				\draw[TealBlue,thick] (0,6)--(0,4.05)--(0.05,3.81)--(-0.01,3.52)--(0,3.3)--(0,2.9);
				\draw[TealBlue,thick] (0,3.3)--(0,2.9)--(0.04,2.72)--(0.04,2.32)--(0,2.17);
				\draw[TealBlue,thick] (0,2.17)--(0,1.7)--(0,1.55)--(0,1.18)--(0,0.9);
				\draw[TealBlue,thick] (0,0.9)--(0,-1.25)--(0.05,-1.91)--(0,-2.1);
				\draw[TealBlue,thick] (0,-2.1)--(0,-2.6)--(0,-2.8)--(0.05,-3.34)--(0,-3.6);
				\draw[TealBlue,thick] (0,-3.6)--(0,-4.1)--(-0.05,-4.29)--(0,-4.77)--(0,-5) ;
				\draw[TealBlue,thick] (0,-5)--(0,-5.4)--(0,-5.57)--(0,-6);
				
				\draw[OliveGreen,thick] (-2.95,6.38) .. controls (-1,6.15) and (1,5.87) .. (2.76,5.55);	
				
				\draw (2,4.5) .. controls (1,4.7) and (1,4.3) .. (2,4.1);								
				
				\draw (0,4.05) .. controls (-1.2,4.2) and (-1.2,3.7) .. (0,3.3);	
				\draw[blue] (2.45,3.3) .. controls (-1.5,4.3) and (-1.5,3.8) .. (2.4,2.8);	
				
				\draw (0,0.9) .. controls (-4.8,2.1) and (-4.8,3.8) .. (0,2.9);	
				\draw[blue] (2.3,2.23) .. controls (-4.8,4) and (-4.8,2.1) .. (2,0.7);
				\draw[blue] (2.2,1.87) .. controls (-4,3.3) and (-4,2.4) .. (2.1,1.05);
				\draw (0,2.17) .. controls (-2.4,2.6) and (-2.4,2.35) .. (0,1.7);	
				
				\draw (0,-1.25) .. controls (-8,-3.9) and (-8,-7.7) .. (0,-6);	
				\draw[blue] (0.72,-5.4) .. controls (-7.5,-7.4) and (-7.5,-3.9) .. (1.63,-1.45);
				\draw (0,-2.1) .. controls (-6.8,-4.4) and (-6.8,-6.8) .. (0,-5.4);
				\draw[blue] (1,-4.5) .. controls (-6,-6.5) and (-6,-4.4) .. (1.4,-2.4);
				\draw (0,-2.6) .. controls (-6,-4.5) and (-6,-6.4) .. (0,-5);
				\draw[blue] (1,-4) .. controls (-4.5,-5.6) and (-4.5,-4.5) .. (1.3,-3);		
				\draw (0,-3.6) .. controls (-3.5,-4.6) and (-3.5,-5) .. (0,-4.1);			
				
				\draw (0,0.9) .. controls (4,-0.1) and (4,-0.1) .. (0,-1.25);
				\draw[blue] (2,0.7) .. controls (4.2,0.1) and (4.2,-0.1) .. (1.69,-1);
				\draw[blue] (2.1,1.05) .. controls (5.2,0.3) and (5.2,-0.3) .. (1.63,-1.45);
				\draw (0,1.7) .. controls (6.3,0.3) and (6.3,-0.3) .. (0,-2.1);
				\draw[blue] (2.2,1.87) .. controls (7.5,0.5) and  (7.5,-0.5).. (1.4,-2.4);
				\draw (0,2.17) .. controls (7.7,0.5) and  (7.7,-0.3).. (0,-2.6);
				\draw[blue] (2.3,2.23) .. controls (9,0.5) and  (9,-0.5).. (1.3,-3);
				\draw (0,2.9) .. controls (10,0.8) and  (10,-0.7).. (0,-3.6);	
				\draw (0,3.3) .. controls (11,0.4) and  (11,-0.8).. (0,-4.1);	
				\draw[blue] (2.4,2.8) .. controls (10.8,0.5) and  (10.8,-1).. (1,-4);
				\draw[blue] (2.45,3.3) .. controls (12,0.9) and  (12.4,-1.2).. (1,-4.5);
				\draw (0,4.05) .. controls (13.2,1.32) and  (13.2,-1.4).. (0,-5);
				\draw (2,4.1) .. controls (13.5,1.6) and  (13.5,-1.9).. (0,-5.4);
				\draw (0,6) .. controls (15,4) and  (15,-2.8).. (0,-6);
				
				\draw(-8,7)--(0,6);	
				\draw(2,4.5)--(4,4.05);	 
				
				\draw[blue](1.69,-1)--(0.7,-1.3);
				\draw[blue](0.72,-5.4)--(1.8,-5.1);	
				
				\draw[red] (2.9,6.5) .. controls (0.25,-12) and  (-1.4,-12) .. (-3,7);
				\draw[red,thick] (2.76,5.55) .. controls (2.4,3) and  (1.5,-3) .. (0.53,-5.9);	
				
				\draw (-7,6.88)--(-6,-5.2);
				
				\draw[magenta] (1.7,5.72) .. controls (0.5,-10.1) and  (-0.7,-9.85) .. (-1.5,6.18);	
				
			\end{tikzpicture}
			
		}	
		
		\vspace{-2.5cm}
		\caption{Regions $D$ and $\Omega$ for $b < 0$ with $\partial^u \Omega$ in olive-green and $\partial^s \Omega$ in red. The arc $\alpha^0$ is magenta. It can have various 'shapes', this figure shows only one of them. The critical locus $\K$ is teal.}
		\label{fig:omega2}
	\end{figure}
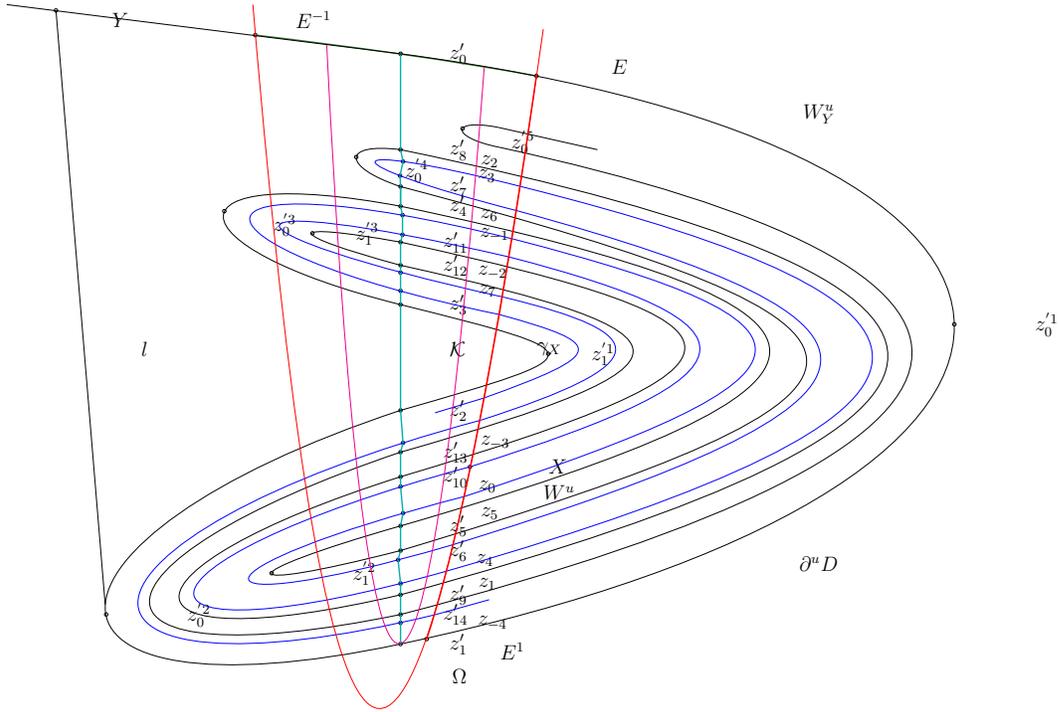

	\subsection{Stems}\label{ss:stems}
	
	In the proof of {\cite[Lemma 4.2]{BS}} the stems of $\T$ and their levels are defined inductively for $b > 0$. Below we recall that definition and parallelly 
	provide the analogous definition for $b < 0$. The only difference is that in the definition for $b > 0$ we use points from basic critical orbits, and for $b < 0$, from quasi-critical orbits.
	
	For $b > 0$, let $\A^0 := \{ \alpha \in \A : \alpha \textrm{ separates $z_0^1$ and $z_0^2$} \} \cup \{ z_0^1, z_0^2 \}$. For $b < 0$, let $\A^0 := \{ \alpha \in \A : \alpha \textrm{ separates $z'^1_0$ and $z'^2_0$} \} \cup \{ z'^1_0, z'^2_0 \}$. We call $B^0 := \pi(\A^0) \subset \T$ a {\it stem} of $\T$, and define the {\it level of the stem} $B^0$ to be zero. 
	
	Let $b > 0$ ($b < 0$, respectively). Suppose that $\alpha \in \A^0$ is such that one of its endpoints lies in $(X, z_0^1)^u$ ($(E, z'^1_0)^u_Y$, respectively), and the other one in $(E, z_0^1)^u$ ($(E', z'^1_0)^u_Y$, respectively). Then $F^{-1}(\alpha) \subset \Omega$, both points of $F^{-1}(\alpha) \cap \partial^u D$ lie in $(E^{-1}, X)^u$ ($(E^{-1}, E)^u_Y$, respectively), and so $F^{-1}(\alpha)$ separates $z_0$ ($z'_0$, respectively) and $\partial^s (\Omega \cap D)$, but $F^{-1}(\alpha)$ need not be in $\A^0$. Since $\Gamma$ is dense in $D$, there exists a unique $\alpha^0 \in \A$, such that $\alpha^0$ separates $z_0$ ($z'_0$, respectively) and $\partial^s(\Omega \cap D)$, and $\alpha^0$ is the maximal element of $\A$ with that property in the sense that every other element of $\A$ that separates $z_0$ ($z'_0$, respectively) and $\partial^s(\Omega \cap D)$ lies in the component of $D \setminus \alpha^0$ that contains $z_0$ ($z'_0$, respectively). Note that $\alpha^0 \in \A^0$, that is, $\alpha^0$ also separates two components of $\partial^s(\Omega \cap D)$, and hence $(\mathring \alpha^0) \cap \partial^u(\Omega \cap D) \ne \emptyset$. Let $b^0 := \pi(\alpha^0)$. Note that $b^0$ is a branch point. We define the level of the branch point $b^0$ to be zero.
	
	For $b > 0$, let $\A^1 := \{ \alpha \in \A : \alpha \textrm{ separates $z_0$ and $\alpha^0$} \} \cup \{ z_0, \alpha^0 \}$. For $b < 0$, let $\A^1 := \{ \alpha \in \A : \alpha \textrm{ separates $z'_0$ and $\alpha^0$} \} \cup \{ z'_0, \alpha^0 \}$. We call $B^1 := \pi(\A^1) \subset \T$ a {\it stem} of $\T$, and define the {\it level of the stem} $B^1$ to be one. Note that $b^0 = B^0 \cap B^1$. 
	
	Now we assume that we have already defined $\A^n$,	all stems of level $n$, and all branch	points of level $n-1$, and we proceed to define $\A^{n+1}$, stems of level $n+1$ and branch	points of level $n$, which depend on the components of $F^{-n}(\Omega)\cap D$.
	
	Consider the region $\Omega_n = F^{-n}(\Omega)$ and note that $\Omega_n \cap D$ consists of finitely many components $s^n_i$, $i = 1, \dots , m_n$, that is $\Omega_n \cap D = \bigcup_{i=1}^{m_n} s^n_i$, $s^n_i \cap s^n_j = \emptyset$ for $i \ne j$. In addition, $F^{-n}(\K)$ is an arc and $F^{-n}(\K) \cap \partial^u D$ consists of finitely many points. Let	$F^{-n}(\K) \cap \partial^u D = \{ Z^n_1, \dots , Z^n_{k_n}\}$, where $F^{-n}(z_0) = Z^n_1 < Z^n_2 < \cdots < Z^n_{k_n} = F^{-n}(z_{-1})$ for $b > 0$, and $F^{-n}(z'_0) = Z^n_1 < Z^n_2 < \cdots < Z^n_{k_n} = F^{-n}(z'_{1})$ for $b < 0$. Note that for every $Z^n_i$, $1 \le i \le k_n$, there exists $j \in \Z$ such that $Z^n_i = F^{-n}(z_j)$ for $b > 0$, and $Z^n_i = F^{-n}(z'_j)$ for $b < 0$, but here we need to index them in the other way and therefore we use capital letters for basic pre-critical or quasi-pre-critical points. Let $K^n_j\subset F^{-n}(\K)$ be an arc connecting $Z^n_{2j-1}$ and $Z^n_{2j}$. Note that $k_n/2 \ge m_n$, since it may happen that some $s^n_i$ contains more than one component $K^n_j$ of $F^{-n}(\K) \cap D$. 
	
	Fix $s^n_j$ for some $j \in \{ 1, \dots , m_n \}$. Recall that the stable boundary of $s^n_j$ is $\partial^s s^n_j = \partial s^n_j \cap W^s$ and the unstable boundary of $s^n_j$ is $\partial^u s^n_j = \partial s^n_j \cap \partial D$. There are finitely many $\gamma_1, \dots , \gamma_l \in \Gamma$ such that $\partial^s s^n_j =  \bigcup_{i= 1}^l \gamma_i$. The unstable boundary $\partial^u s^n_j$ has the same number of components $\delta_1, \dots , \delta_l$, $\delta_i \subset \partial D$ and $\partial^u s^n_j =  \bigcup_{i= 1}^l \delta_i$. 
	
	In order to define stems of level $n+1$ we have two cases to consider. First, let us suppose that there is a unique odd $p$, $1 \le p < k_n$, such that $Z^n_p, Z^n_{p+1} \in \partial^u s^n_j$. Since $\Gamma$ is dense in $D$, there is a unique element of $\A$, say $\alpha^n_p \in \A$, such that $\alpha^n_p$ separates $Z^n_p$ and $\partial^s s^n_j$, and $\alpha^n_p$ is the maximal element of $\A$ with that property in the sense that every other element of $\A$ that separates $Z^n_p$ and $\partial^s s^n_j$ lies in the component of $D \setminus \alpha^n_p$ that contains $Z^n_p$. Therefore, $\alpha^n_p \in \bigcup_{q=0}^n \A^q$ and $(\mathring \alpha^n_p) \cap \partial^us^n_j \ne \emptyset$. We let $\A^{n+1}_p := \{ \alpha \in \A : \alpha \textrm{ separates $Z^n_p$ and $\alpha^n_p$} \} \cup \{ Z^n_p, \alpha^n_p \}$. We call $B^{n+1}_p := \pi(\A^{n+1}_p) \subset \T$ a {\it stem} of $\T$ of {\it level} $n+1$. Let $b^n_p := \pi(\alpha^n_p)$, $b^n_p$ is a branch point and we define its {\it level} to be $n$.
	
	Second, let us suppose that $F^{-n}(\K) \cap \partial^u s^n_j = \{ Z^n_p, \dots , Z^n_{p+r} \}$ and $r > 1$. Note that $(r+1)/2 \le l$. Again, since $\Gamma$ is dense in $D$, there are elements of  $\bigcup_{q=0}^n \A^q$, say $\alpha^n_p, \alpha^n_{p+2}, \dots \alpha^n_{p+r-1} \in$  $\bigcup_{q=0}^n \A^q$, such that for every $i$, $0 \le i \le (r-1)/2$, $(\mathring \alpha^n_{p+2i}) \cap \partial^us^n_j \ne \emptyset$, and for every $k$, $0 \le k \le (r-1)/2$, $\bigcup_{i=0}^{k} \alpha^n_{p+2i}$	separates $Z^n_{p+2k}$ and $\partial^s s^n_j$.
	
	Obviously, $\pi(\alpha^n_{p+2i})$ is a branch point for every $i$. If additionally for some $i$, $\alpha^n_{p+2i}$ separates $Z^n_{p+2i}$ and $\partial^s s^n_j$, then $\pi(\alpha^n_{p+2i})$ is a branch point of {\it level} $n$. If $\alpha^n_{p+2i}$ does not separate $Z^n_{p+2i}$ and $\partial^s s^n_j$, then the branch point $\pi(\alpha^n_{p+2i})$ has level $m$ for some $m < n$. Note that $\alpha^n_p$ separates $Z^n_p$ and $\partial^s s^n_j$, and hence $\pi(\alpha^n_p)$ is a branch point of level $n$.
	
	For every $i \in \{  1, \dots , (r-2)/2 \}$ such that $\pi(\alpha^n_{p+2i})$ is a branch point of level $n$ let $\A^{n+1}_{p+2i} := \{ \alpha \in \A : \alpha \textrm{ separates $Z^n_{p+2i}$ and $\alpha^n_{p+2i}$} \} \cup \{ Z^n_{p+2i}, \alpha^n_{p+2i} \}$. We call $B^{n+1}_{p+2i} := \pi(\A^n_{p+2i}) \subset \T$ a {\it stem} of $\T$ of {\it level} $n+1$. 
	
	By the construction, between any two branch points of level $n \ge 1$ there is at least one branch point of some smaller level. Namely, the stable boundary of every $s^n_j \subset \Omega_n$ contains at least one component that belongs to $\partial^s \Omega_{n-1}$.
	
	Finally, let $\A^{n+1} = \bigcup_i \A^{n+1}_i$, where $i \in \{  1, \dots , (r-2)/2 \}$ such that $\pi(\alpha^n_{p+2i})$ is a branch point of level $n$.

	\subsection{Endpoints $E_X$ and branch points $B_X$ of $\mathbb{T}$}
	
	Recall that we call the points from the set $F(\mfc)$ the turning points. Let us call the points from $F(\H)$ the quasi-turning points. For convenience, let us denote by $S_{j}$, $j \in \Z$, the turning points for $b > 0$ and by $S_{j}$, $j \in \N_0$, the quasi-turning points for $b < 0$. Although we denote turning and quasi-turning points by the same letter, it will not make any confusion.
	
	For $b > 0$, the set of turning points $\{ S_j : j\in \mathbb{Z} \}$ is indexed in the following standard way (that follows from the indexing of the basic critical points defined in the last paragraph of Section \ref{sec:prelim}): $S_0$ is the first turning point on $W^u$ on the right of the fixed point $X$ (and $X$ is between $S_0$ and $S_0^1$). All the other turning points are indexed such that $S_k \subset W^{u+}$ and $S_{-k} \subset W^{u-}$ for all $k \in \N$, and $S_i$ and $S_j$ are consecutive (there are no other turning points in $W^u$ between them) if and only if $|i - j| = 1$ for all $i, j \in \Z$. For $b < 0$, the turning points $\{ S_j : j \in \N_0 \}$ are indexed in the following way (that follows from the indexing of the quasi-critical points defined in the second paragraph of (1) in Section \ref{sec:cl}): $S_0$ is the first quasi-turning point on $W^u_Y$ on the right of the fixed point $Y$. All the other quasi-turning points are indexed such that $S_k \subset W^{u+}_Y$ for all $k \in \N_0$, and $S_i$ and $S_j$ are consecutive (there are no other quasi-turning points in $W^u_Y$ between them) if and only if $|i - j| = 1$ for all $i, j \in \N_0$.
	
	Let us consider again the critical locus $\K$, and give it orientation opposite to the direction of the $y$-axis, and hence its initial point, say $\mf e$, has a positive $y$-coordinate, and its terminal point, say $\mf b$, has a negative $y$-coordinate. We will consider $F^{-n}(\K)$ as a directed arc, whose initial point is $F^{-n}(\mf{e})$, and terminal point is $F^{-n}(\mf{b})$. Moreover, we will consider every component of $F^{-n}(\K) \cap D$ as a directed arc, whose direction is given by the direction of $F^{-n}(\K)$, and hence every component has its initial point and its terminal point.
	
	\begin{lem}\label{lem:ebp}
		Let $x \in \T$, $x = [(s_n)_{n=1}^{\infty}]$. 
		\begin{enumerate}[(1)]
			\item
			If $\bigcap_{n=1}^\infty s_n = Z$, where $Z \in \partial^u D \cap F^{-n}(\K)$, for some $n \in \N_0$, is a point, or if $\bigcap_{n=1}^\infty s_n \in \{ S_0, S_0^1 \}$, then $x$ is an endpoint of $\T$. 
			\item
			If $\bigcap_{n=1}^\infty s_n = \alpha$, and $\alpha$ contains the terminal point of a component of $F^{-n}(\K) \cap D$ for some $n \in \N_0$, then $x$ is a branch point of $\T$.
		\end{enumerate}
	\end{lem}
	\begin{proof}
		The proof follows directly from the construction of stems in the proof of {\cite[Lemma 4.2]{BS}} and Subsection \ref{ss:stems}. For completeness, we briefly recall only the key idea. 
		\begin{enumerate}[(1)]
			\item 
			If there exists a stem $B$ such that $x$ is an endpoint of $B$, then $x$ is an endpoint of $\T$. Every stem contains only one endpoint, except the stem $B^0$, which contains two endpoints. If  $\bigcap_{n=1}^\infty s_n = Z$, where $Z \in \partial^u D \cap L^{-n}(\K)$, for some $n \in \N_0$, is a point, then by {\cite[Lemma 4.1]{BS}} and Subsection \ref{ss:stems}, $x$ is the endpoint of a stem $B \ne B^0$. Let us denote the endpoints of $B^0$ by $r_0$ and $r_0^1$. Then the geometric realizations of $r_0, r_0^1$ are $S_0, S_0^1$ respectively.
			\item
			If $\bigcap_{n=1}^\infty s_n = \alpha$, and $\alpha$ contains the terminal point, say $\mf b$, of a component of $F^{-n}(\K) \cap D$ for some $n \in \N_0$, then $\mf{b} \nin \partial \alpha$ and $\mf b \in \partial^u D$, implying by {\cite[Lemma 4.1]{BS}} and Subsection \ref{ss:stems} that $x$ is a branch point of $\T$.
		\end{enumerate}
	\end{proof}
	
	Below we define special classes of endpoints and branch points of trees given by the H\'enon maps, that will turn out to give rise to all turning points in the H\'enon attractors. 
	
	\begin{df}(Endpoints $E_X$ of $\mathbb{T}$)
		An endpoint of $\mathbb{T}$ is in $E_X$ if and only if it satisfies the condition of Lemma \ref{lem:ebp} (1).
	\end{df}
	
	\begin{df}(Branch points $B_X$ of $\mathbb{T}$)\label{df:BP}
		A branch point of $\mathbb{T}$ is in $B_X$ if and only if it satisfies the condition of Lemma \ref{lem:ebp} (2).
	\end{df}
	
	\begin{rem}
		Not all endpoints of $\T$ are in $E_X$.
	\end{rem}
	
	Note that the point $Z$ in Lemma \ref{lem:ebp} (1) is always an initial point of a component of $F^{-n}(\K) \cap D$, and it is the $n$th pre-image of a basic critical point for $b > 0$, and of a quasi-critical point for $b < 0$.
	
	\begin{lem}\label{lem:ep}
		For every endpoint $x \in E_X$ we have $f^{-1}(x)\in E_X$; i.e., $f^{-1}(x)$ is a singleton.
	\end{lem}
	\begin{proof}
		Let $x = [(s_n)_{n=1}^{\infty}] \in B \ne B^0$ be an endpoint that belongs to $E_X$. Then, by Lemma \ref{lem:ebp} (1), $\bigcap_{n=1}^\infty s_n = Z$, where $Z \in \partial^u D \cap F^{-n}(\K)$, for some $n \in \N_0$, is a point. Since $F$ is a homeomorphism, there exists $N \in \N$ such that for every $n \ge N$, $s_n' := F^{-1}(s_n)$ is a surface in $D$ bounded by a finite number of elements in $\Gamma$, $\Cl (s'_{n+1}) \subset s'_n$, and hence, $x' := [(s'_n)_{n=N}^{\infty}]$ is a point of $\T$. Moreover, $\bigcap_{n=N}^\infty (s'_n) =\bigcap_{n=N}^\infty F^{-1}(s_n) = F^{-1}(Z)$ is a point and $F^{-1}(Z) \in \partial^u D \cap F^{-n-1}(\K)$. Therefore, by Lemma \ref{lem:ebp}, $x'$ is an endpoint in $E_X$ and $f^{-1}(x) = x'$.
		
		Recall, $r_0$ and $r_0^1$ denote the endpoints of $B^0$. Since the geometric realizations of $r_0, r_0^1$ are $S_0, S_0^1$ respectively, we have $f^{-1}(r_0^1) = r_0$ and $f^{-1}(r_0) = e_1$, where $e_1$ is the endpoint of the stem $B^1$, and the proof follows by induction.
	\end{proof}
	
	Note that $Z \in \partial^u D$ implies that the sequence $(F^{-n}(Z))_{n=0}^{\infty}$ converges to the fixed point in $\partial^u D$ (that is $X$ for $b > 0$, and $Y$ for $b < 0$). Therefore, for every endpoint $x \in E_X$, the sequence $(f^{-n}(x))_{n=0}^{\infty}$ converges to the fixed point of $f$ that is related to the fixed point in $\partial^u D$. 
	
	In the following lemma, we show that the pre-image of any branch point from $B_X$ always contains
	another branch point from $B_X$. 
	
	\begin{lem}\label{lem:bp}
		For any branch point $x \in B_X$ we have $f^{-1}(x)\cap B_X\neq\emptyset$. 
	\end{lem}
	\begin{proof}
		Let $x = [(s_n)_{n=1}^{\infty}] \in B_X$. Then, by definition of the set $B_X$, $\bigcap_{n=1}^\infty s_n = \alpha$, and $\alpha$ contains the terminal point, say $\mf b$, of a component denoted by $K_{\mf b}$ of $F^{-n}(\K) \cap D$, for some $n \in \N_0$. Let us consider $F^{-1}(\alpha)$ and $F^{-1}(K_{\mf b})$.
		
		If $F^{-1}(K_{\mf b}) \subset D$ then $F^{-1}(\alpha) \subset D$ and hence $f^{-1}(x)$ contains only one point, denote it $x^{-1}$, and the geometric realization of $x^{-1} = f^{-1}(x)$ is $F^{-1}(\alpha)$. Also $F^{-1}(\mf{b}) \in F^{-1}(\alpha)$ and $F^{-1}(\mf{b})$ is the terminal point of $F^{-1}(K_{\mf b})$, so $f^{-1}(x)$ is a branch point.
		
		If $F^{-1}(K_{\mf b}) \setminus D \ne \emptyset$ then $F^{-1}(\alpha) \cap D$ has at least two components and one of them, say $\alpha'$, contains the point $F^{-1}(\mf{b})$ that is also the terminal point of a component of $F^{-1}(K_{\mf b}) \cap D$. Therefore, $\alpha'$ is the geometric realization of one of the preimages of $x$. Let us denote that preimage by $x^{-1}$. Then $x^{-1}$ is a branch point in $B_X$. We can inductively define $x^{-n}$, for every $n \in \N$, and the proof follows.
	\end{proof}
	
	Note that $\mf{b} \in \partial^u D$ implies that the sequence $(F^{-n}(\mf{b}))_{n=0}^{\infty}$ converges to the fixed point in $\partial^u D$. Therefore, for every branch point $x \in B_X$, the sequence $(x^{-n})_{n=0}^{\infty}$ converges to the fixed point of $f$ that is related to the fixed point in $\partial^u D$. . 
	
	\begin{lem}\label{lem:ebtoeb}
		For every $i\in\mathbb{Z}$ there exists $k \in \N_0$ such that $\pi (S_i^{-k-n})\in E_X \cup B_X$, for all $n \in\mathbb{N}_0$. Conversely, for every $r \in E_X \cup B_X$ there exist $i\in\mathbb{Z}$ and $k \in \N_0$ such that $\pi (S_{i}^{-k}) = r$.
	\end{lem}
	\begin{proof}
		Since for $b > 0$ we have $S_i^{-1} \in W^u \cap \K$, and for $b < 0$ we have $S_i^{-1} \in W^u_Y \cap \K$, there exists $N_i \in \N_0$ such that $S_i^{-k} \in \partial^u D \cap F^{-k}(\K)$ for every $k \ge N_i$ and by Lemmas \ref{lem:ebp}, \ref{lem:ep} and \ref{lem:bp}, $\pi (S_i^{-k}) \in E_X \cup B_X$, for every $k \ge N_i$.
		
		Let $r \in E_X \cup B_X$. By definitions of the sets $E_X$ and $B_X$, there exists $k \in \N_0$ such that the geometric realization of $r$ is either the initial point of a component of $F^{-k}(\K) \cap D$, or it contains the terminal point of a component of $F^{-k}(\K) \cap D$. In any case that point, denote it by $P$, belongs to $F^{-k}(\K) \cap \partial^u D$. Hence, for $b > 0$ we have $P^k \in F^{k}(F^{-k}(\K) \cap \partial^u D) \subset W^u \cap \K$ and $W^u$ intersects $\K$ at $P^k$, and for $b < 0$ we have $P^k \in F^{k}(F^{-k}(\K) \cap \partial^u D) \subset W^u_Y \cap \K$ and $W^u_Y$ intersects $\K$ at $P^k$. Thus $P^k = S_i^{-1}$ for some $i \in \Z$ and $\pi (S^{-1-k}) = r$.
	\end{proof}

	\subsection{Classification}
	
	Let $F_i := F_{a_i, b_i} : D_i \to D_i$, $i = 1, 2$, be two H\'enon maps within the Wang-Young parameter set, and let $\Lambda_i$ denote the attractor of $F_i$. Let $X_i, Y_i$ be the fixed points of $F_i$, where $X_i \in \Lambda_i$, and for $b < 0$ we have $Y_i \in \partial^u D_i$. Let $S_{i,j}$, $j \in \Z$ ($j \in \N_0$, respectively), be the turning points od $F_i$ for $b > 0$ (the qusi-turning points of $F_i$ for $b < 0$, respectively). We analogously denote all 'items' related to the map $F_i$ by an index $i$. Let also $f_1:\T_1\to \T_1$ and $f_2:\T_2\to\T_2$ be the corresponding tree maps, whose natural extensions are conjugate to $F_1$ and $F_2$ respectively. 
	
	\begin{lem}\label{lem:gamma}
		Suppose that $H : D_1 \to D_2$ is a conjugacy between $F_1$ and $F_2$. Let $\gamma_{X_i} \in \Gamma_i$ be such that $X_i \in \gamma_{X_i}$. Then there exists $k \in \Z$ such that $F_2^k(H(\partial \gamma_{X_1})) = \partial \gamma_{X_2}$.
	\end{lem}
	\begin{proof}
		Since $H$ is a conjugacy, $H(X_1) = X_2$, $H(W^u_1) = W^u_2$ and  $H(W^s_1) = W^s_2$. Since homoclinic points are invariant for a conjugacy, there exists $k \in \Z$ such that for every $i \ge k$, we have $F_2^i(H(\gamma_{X_1})) \subseteq \gamma_{X_2}$.
		
		Case 1. Let $b > 0$. Let $\{ E'_1 \} = \gamma_{X_1} \cap [S_{1,0}^2, S_{1,1}]^u$, $\{ E'_2 \} = \gamma_{X_2} \cap [S_{2,0}^2, S_{2,1}]^u$, and as before $E_i \in \partial \gamma_{X_i}$, see Figue \ref{fig:ee}. It is easy to see that there are only two homoclinic orbits of $F_i$, $\{ E_i^j : j \in \Z \}$ and $\{ E'^j_i : j \in \Z \}$, such that $(X_i, P_i)^u \cap (X_i, P_i)^s = \emptyset$ for $P_i \in \{ E^j_i, E'^j_i : j \in \Z \}$ (recall that $Q = Q^0$ for every point $Q$).
		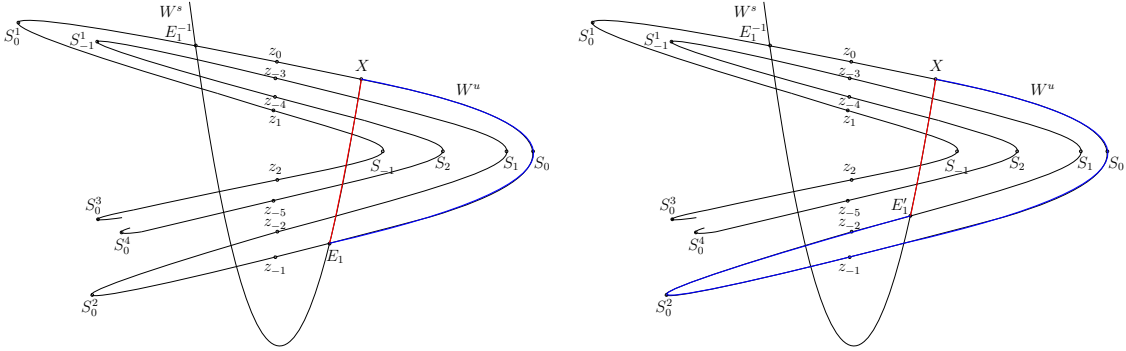
\begin{figure}[h]
			
			\begin{adjustbox}{center}
				\begin{tabular}{cc}
					\hspace{-0.5cm}
					\resizebox{0.6\textwidth}{!}{
						
						\begin{tikzpicture}
							
							\tikzstyle{every node}=[draw, circle, fill=white, minimum size=2pt, inner sep=0pt]
							\tikzstyle{dot}=[circle, fill=white, minimum size=0pt, inner sep=0pt, outer sep=-1pt]
							
							\node (n1) at (5*4/3,0) {};
							\node (n2) at (-6.65,5*2/3)  {};
							\node (n3) at (-4.74,-3.72)  {};
							\node (n4) at (5*5/9,0)  {};
							\node (n5) at (-4.59,-1.76) {};
							\node (n6) at (-3.98,-2.1)  {};
							\node (n7) at (-4.61,2.84)  {};
							\node (n8) at (5*61/51,0)  {};
							\node (n9) at (5*65/75,0)  {};
							\node (n10) at (5*4/9,1.87)  {};
							\node (n21) at (0.04,2.32)  {};
							\node (n22) at (-0.05,1.06)  {};
							\node (n23) at (0.05,-0.74)  {};
							\node (n24) at (-0.05,-1.28)  {};
							\node (n25) at (0.05,-2.08)  {};
							\node (n26) at (0,-2.74)  {};
							\node (n27) at (0,1.89)  {};
							\node (n28) at (-0.01,1.41)  {};
							\node (n11) at (-2.06,2.74)  {};
							\node (n12) at (1.4,-2.39)  {};
							%\node (n13) at (1.57,-1.67)  {};
							
							\draw (5*4/9,1.87) .. controls (0.7,-7.7) and  (-0.7,-7.6) .. (-5*4/9,3.87);	
							
							\draw (3,1.7) .. controls (7.7,0.7) and  (7.7,-0.6).. (4,-1.7);	
							\draw (3,1.7) .. controls (-9,4.25) and (-9,3.7) .. (-1.5,1.5);
							\draw (0,-2.1) .. controls (-6.9,-4.1) and (-6.9,-4.6) .. (4,-1.7);
							\draw (2.8,1.2) .. controls (7.44,0) and (7.44,-0.1) .. (0,-2.1);
							\draw (2.8,1.2) .. controls (-7,3.7) and (-7,3.2) .. (2.3,0.8);
							\draw (2.3,0.8) .. controls (5.63,-0.2) and (5.63,0.1) .. (-5*19/24+0.5,-5*5/12);
							\draw (-1.5,1.5) .. controls (4.55,-0.4) and (4.55,0.3) .. (-5*5/6+0.5,-5*1/3+0.15);
							\draw (-5*5/6+0.5,-5*1/3+0.15) .. controls (-5*5/6-0.7,-5*1/3-0.08) and (-5*5/6-0.7,-5*1/3-0.18) .. (-5*5/6+0.2,-5*1/3-0.05);
							\draw (-5*19/24+0.2,-5*5/12+0.1) .. controls (-5*19/24-0.15,-5*5/12) and (-5*19/24-0.15,-5*5/12-0.1) .. (-5*19/24+0.5,-5*5/12);
							
							\draw[red,thick] (5*4/9,1.87) .. controls (2.2,1.55) and  (1.62,-1.55) .. (1.4,-2.38);	
							
							\draw[blue,thick] (5*4/9,1.87) .. controls (8.1,0.7) and  (8.5,-0.76).. (1.4,-2.38);
							
							\node[dot, draw=none, label=above: $S_0$] at (6.9,0) {};
							\node[dot, draw=none, label=above: $z_0$] at (0,2.8) {};
							\node[dot, draw=none, label=above: $S^1_0$] at (-6.8,5*2/3) {};
							\node[dot, draw=none, label=above: $z_1$] at (0,1.05) {};
							\node[dot, draw=none, label=above: $S_{-1}$] at (5*5/9,0) {};
							\node[dot, draw=none, label=above: $z_2$] at (0,-0.2) {};
							\node[dot, draw=none, label=above: $S^3_0$] at (-4.7,-1.05) {};
							\node[dot, draw=none, label=above: $S^4_0$] at (-5*19/24,-5*5/12) {};
							\node[dot, draw=none, label=above: $z_{-5}$] at (0,-1.2) {};
							\node[dot, draw=none, label=above: $z_{-3}$] at (0,2.4) {};
							\node[dot, draw=none, label=above: $z_{-4}$] at (0,1.6) {};
							\node[dot, draw=none, label=above: $S_2$] at (5*65/75,0) {};
							\node[dot, draw=none, label=above: $S^1_{-1}$] at (-5,3.3) {};
							\node[dot, draw=none, label=above: $S_1$] at (6.1,0) {};
							\node[dot, draw=none, label=above: $S^2_0$] at (-4.8,-3.7) {};
							\node[dot, draw=none, label=above: $z_{-1}$] at (0,-2.7) {};
							\node[dot, draw=none, label=above: $z_{-2}$] at (0,-1.53) {};
							\node[dot, draw=none, label=above: $X$] at (2.25,2.5) {};
							\node[dot, draw=none, label=above: $W^u$] at (5,2) {};
							\node[dot, draw=none, label=above: $W^s$] at (-2.7,4) {};
							\node[dot, draw=none, label=above: $E_1^{-1}$] at (-2.5,3.55)  {};
							\node[dot, draw=none, label=above: $E_1$] at (1.55,-2.39)  {};
							
						\end{tikzpicture}
						
					}
					
					&
					\hspace{-2cm}
					
					\resizebox{0.6\textwidth}{!}{
						
						\begin{tikzpicture}
							
							\tikzstyle{every node}=[draw, circle, fill=white, minimum size=2pt, inner sep=0pt]
							\tikzstyle{dot}=[circle, fill=white, minimum size=0pt, inner sep=0pt, outer sep=-1pt]
							
							\node (n1) at (5*4/3,0) {};
							\node (n2) at (-6.65,5*2/3)  {};
							\node (n3) at (-4.74,-3.72)  {};
							\node (n4) at (5*5/9,0)  {};
							\node (n5) at (-4.59,-1.76) {};
							\node (n6) at (-3.98,-2.1)  {};
							\node (n7) at (-4.61,2.84)  {};
							\node (n8) at (5*61/51,0)  {};
							\node (n9) at (5*65/75,0)  {};
							
							\node (n10) at (5*4/9,1.87)  {};
							
							\node (n21) at (0.04,2.32)  {};
							\node (n22) at (-0.05,1.06)  {};
							\node (n23) at (0.05,-0.74)  {};
							\node (n24) at (-0.05,-1.28)  {};
							\node (n25) at (0.05,-2.08)  {};
							\node (n26) at (0,-2.74)  {};
							
							\node (n27) at (0,1.89)  {};
							\node (n28) at (-0.01,1.41)  {};
							
							\node (n11) at (-2.06,2.74)  {};
							%\node (n12) at (1.4,-2.39)  {};
							\node (n13) at (1.57,-1.67)  {};
							
							\draw (5*4/9,1.87) .. controls (0.7,-7.7) and  (-0.7,-7.6) .. (-5*4/9,3.87);	
							
							\draw (3,1.7) .. controls (7.7,0.7) and  (7.7,-0.6).. (4,-1.7);	
							\draw (3,1.7) .. controls (-9,4.25) and (-9,3.7) .. (-1.5,1.5);
							\draw (0,-2.1) .. controls (-6.9,-4.1) and (-6.9,-4.6) .. (4,-1.7);
							\draw (2.8,1.2) .. controls (7.44,0) and (7.44,-0.1) .. (0,-2.1);
							\draw (2.8,1.2) .. controls (-7,3.7) and (-7,3.2) .. (2.3,0.8);
							\draw (2.3,0.8) .. controls (5.63,-0.2) and (5.63,0.1) .. (-5*19/24+0.5,-5*5/12);
							\draw (-1.5,1.5) .. controls (4.55,-0.4) and (4.55,0.3) .. (-5*5/6+0.5,-5*1/3+0.15);
							\draw (-5*5/6+0.5,-5*1/3+0.15) .. controls (-5*5/6-0.7,-5*1/3-0.08) and (-5*5/6-0.7,-5*1/3-0.18) .. (-5*5/6+0.2,-5*1/3-0.05);
							\draw (-5*19/24+0.2,-5*5/12+0.1) .. controls (-5*19/24-0.15,-5*5/12) and (-5*19/24-0.15,-5*5/12-0.1) .. (-5*19/24+0.5,-5*5/12);
							
							\draw[red,thick] (5*4/9,1.87) .. controls (2.2,1.55) and  (1.62,-1.55) .. (1.57,-1.67);	
							
							\draw[blue,thick] (5*4/9,1.87) .. controls (8.1,0.7) and  (8.5,-0.76).. (1.4,-2.38);
							\draw[blue,thick] (1.57,-1.67) .. controls (-6.8,-4.0) and (-6.8,-4.5) .. (1.4,-2.38);
							
							\node[dot, draw=none, label=above: $S_0$] at (6.9,0) {};
							\node[dot, draw=none, label=above: $z_0$] at (0,2.8) {};
							\node[dot, draw=none, label=above: $S^1_0$] at (-6.8,5*2/3) {};
							\node[dot, draw=none, label=above: $z_1$] at (0,1.05) {};
							\node[dot, draw=none, label=above: $S_{-1}$] at (5*5/9,0) {};
							\node[dot, draw=none, label=above: $z_2$] at (0,-0.2) {};
							\node[dot, draw=none, label=above: $S^3_0$] at (-4.7,-1.05) {};
							\node[dot, draw=none, label=above: $S^4_0$] at (-5*19/24,-5*5/12) {};
							\node[dot, draw=none, label=above: $z_{-5}$] at (0,-1.2) {};
							\node[dot, draw=none, label=above: $z_{-3}$] at (0,2.4) {};
							\node[dot, draw=none, label=above: $z_{-4}$] at (0,1.6) {};
							\node[dot, draw=none, label=above: $S_2$] at (5*65/75,0) {};
							\node[dot, draw=none, label=above: $S^1_{-1}$] at (-5,3.3) {};
							\node[dot, draw=none, label=above: $S_1$] at (6.1,0) {};
							\node[dot, draw=none, label=above: $S^2_0$] at (-4.8,-3.7) {};
							\node[dot, draw=none, label=above: $z_{-1}$] at (0,-2.7) {};
							\node[dot, draw=none, label=above: $z_{-2}$] at (0,-1.53) {};
							\node[dot, draw=none, label=above: $X$] at (2.25,2.5) {};
							\node[dot, draw=none, label=above: $W^u$] at (5,2) {};
							\node[dot, draw=none, label=above: $W^s$] at (-2.7,4) {};
							\node[dot, draw=none, label=above: $E_1^{-1}$] at (-2.5,3.55)  {};
							\node[dot, draw=none, label=above: $E'_1$] at (1.3,-1.05)  {};
							
						\end{tikzpicture}
						
					}	
					
				\end{tabular}	
				
			\end{adjustbox}
			\vspace{-1cm}
			\caption{Positions of points $E_1$ and $E'_1$.}
			\label{fig:ee}
		\end{figure}
		Note that for the endpoint $E_i$ of $\gamma_{X_i}$ there exists a sequence of homoclinic points in $\gamma_{X_i}$ that converges to $E_i$. On the other hand, for the endpoint $E'_i$ of $[X_i, E'_i]^s$ there exists no sequence of homoclinic points in $[X_i, E'_i]^s$ converging to $E'_i$. Therefore, $H(\gamma_{X_1}) = [X_2, Q]^s$ implies $Q \in \{ E_2^j : j \in \Z \}$ and the claim follows. 
		
		Case 2. Let $b < 0$. Similarly as in Case 1, let $G_1 = \gamma_{X_1} \cap [z_{1,0}^3, z_{1,0}^1]^u$, $G_2 = \gamma_{X_2} \cap [z_{2,0}^3, z_{2,0}^1]^u$, $G'_1 = \gamma_{X_1} \cap [z_{1,0}^3, z_{1,1}^1]^u$ and $G'_2 = \gamma_{X_2} \cap [z_{2,0}^3, z_{2,1}^1]^u$, see Figure \ref{fig:gg}. It is easy to see that there are only two homoclinic orbits of $F_i$, $\{ G_i^j : j \in \Z \}$ and $\{ G'^j_i : j \in \Z \}$, such that $(P_i, P^1_i)^u \cap (P_i, P^1_i)^s = \{ X_i \}$, for $P_i \in \{ G^j_i, G'^j_i : j \in \Z \}$.
		\begin{figure}[h]
			
			\begin{adjustbox}{center}
				\begin{tabular}{cc}
					\hspace{-0.5cm}
					\resizebox{0.6\textwidth}{!}{
						
						\begin{tikzpicture}
							
							\tikzstyle{every node}=[draw, circle, fill=white, minimum size=2pt, inner sep=0pt]
							\tikzstyle{dot}=[circle, fill=white, minimum size=0pt, inner sep=0pt, outer sep=-1pt]
							
							\node[label=right: $z^1_{-1}$] at (9.59,-0.1) {};
							\node[label=right: $z^1_{-2}$] at (8.53,-0.15) {};
							\node[label=right: $z^1_{-3}$] at (7.2,0)  {};
							\node[label=right: $z^1_0$] at (6.08,0)  {};
							\node[label=right: $z^1_{1}$] at (4.36,0)  {};
							\node[label=left: $z^1_{-4}$] at (3.6,-0.05)  {};
							
							\node[label=left: $z_0^4$] at (-0.5,3.8)  {};
							
							\node[label=right: $z_0^3$] at (-2.46,2.51)  {};
							\node[label=left: $z_{-1}^3$] at (-3.06,2.65)  {};
							
							\node[label=left: $z_{-1}^2$] at (-5.3,-5.3)  {};	
							\node[label=left: $z_0^2$] at (-4.19,-5.05)  {};		
							\node[label=left: $z_{1}^2$] at (-3.05,-4.7) {};
							
							\node[label=above: $z_2$] at (0.05,3.81)  {};
							\node[label=below: $z_3$] at (-0.01,3.52)  {};
							\node[label=above: $z_{6}$] at (0.04,2.77)  {};
							\node[label=below left: $z_{-1}$] at (0.04,2.32)  {};
							\node[label=above right: $z_{-2}$] at (0,1.55)  {};
							\node[label=below: $z_{7}$] at (0,1.07)  {};					
							
							\node[label=above: $z_{-3}$] at (0.05,-1.9)  {};
							\node[label=above: $z_0$] at (0,-2.8)  {};
							\node[label=below: $z_5$] at (0.05,-3.34)  {};
							\node[label=above: $z_4$] at (-0.05,-4.29)  {};
							\node[label=below: $z_1$] at (0,-4.77)  {};
							\node[label=below: $z_{-4}$] at (0,-5.58)  {};						
							
							\node[label=below right: $G_1$] at (5*4/9,1.87)  {};
							\node[label=below right: $X_1$] at (1.41,-2.4)  {}; 
							\node[label=below right: $G_1^1$] at (0.96,-4.5)  {};				
							
							\draw (2.45,3.3) .. controls (-1.5,4.3) and (-1.5,3.8) .. (2.4,2.78);
							
							\draw (2.3,2.3) .. controls (-4.8,4) and (-4.8,2) .. (2,0.6);
							\draw (2.2,1.87) .. controls (-4,3.3) and (-4,2.4) .. (2.1,1.05);
							
							\draw (0.8,-5.4) .. controls (-7.5,-7.4) and (-7.5,-3.9) .. (1.63,-1.45);
							\draw[blue,thick] (1,-4.5) .. controls (-6,-6.5) and (-6,-4.4) .. (1.4,-2.4);
							\draw (1,-4) .. controls (-4.5,-5.6) and (-4.5,-4.5) .. (1.3,-3);
							
							\draw (2,0.6) .. controls (4.2,0.1) and (4.2,-0.1) .. (1.69,-1);						
							\draw (2.1,1.05) .. controls (5.2,0.3) and (5.2,-0.3) .. (1.63,-1.45);
							\draw[blue,thick] (2.2,1.87) .. controls (7.5,0.5) and  (7.5,-0.5).. (1.4,-2.4);
							\draw (2.3,2.3) .. controls (9,0.5) and  (9,-0.5).. (1.3,-3);
							\draw (2.4,2.78) .. controls (10.8,0.5) and  (10.8,-1).. (1,-4);
							\draw (2.45,3.3) .. controls (12,0.9) and  (12.4,-1.2).. (1,-4.5);
							
							\draw(1.69,-1)--(0.7,-1.3);
							\draw(0.8,-5.4)--(1.8,-5.1);								
							
							\draw (2.65,4.5) .. controls (0,-12) and  (-1.5,-12) .. (-2.15,4.5);
							%	\draw[red] (2.58,4) .. controls (2.1,1.5) and  (1.51,-3) .. (0.38,-6.5);
							\draw[red,thick] (5*4/9,1.87) .. controls (1.9,0) and  (1.35,-3) .. (0.96,-4.5);	
							
						\end{tikzpicture}
						
					}
					
					&
					\vspace{-1cm}
					\hspace{-2cm}
					
					\resizebox{0.6\textwidth}{!}{
						
						\begin{tikzpicture}
							
							\tikzstyle{every node}=[draw, circle, fill=white, minimum size=2pt, inner sep=0pt]
							\tikzstyle{dot}=[circle, fill=white, minimum size=0pt, inner sep=0pt, outer sep=-1pt]
							
							\node[label=right: $z^1_{-1}$] at (9.59,-0.1) {};
							\node[label=right: $z^1_{-2}$] at (8.53,-0.15) {};
							\node[label=right: $z^1_{-3}$] at (7.2,0)  {};
							\node[label=right: $z^1_0$] at (6.08,0)  {};
							\node[label=right: $z^1_{1}$] at (4.36,0)  {};
							\node[label=left: $z^1_{-4}$] at (3.6,-0.05)  {};
							
							\node[label=left: $z_0^4$] at (-0.5,3.8)  {};
							
							\node[label=right: $z_0^3$] at (-2.46,2.51)  {};
							\node[label=left: $z_{-1}^3$] at (-3.06,2.65)  {};
							
							\node[label=left: $z_{-1}^2$] at (-5.3,-5.3)  {};	
							\node[label=left: $z_0^2$] at (-4.19,-5.05)  {};		
							\node[label=left: $z_{1}^2$] at (-3.05,-4.7) {};
							
							\node[label=above: $z_2$] at (0.05,3.81)  {};
							\node[label=below: $z_3$] at (-0.01,3.52)  {};
							\node[label=above: $z_{6}$] at (0.04,2.77)  {};
							\node[label=below left: $z_{-1}$] at (0.04,2.32)  {};
							\node[label=above right: $z_{-2}$] at (0,1.55)  {};
							\node[label=below: $z_{7}$] at (0,1.07)  {};					
							
							\node[label=above: $z_{-3}$] at (0.05,-1.9)  {};
							\node[label=above: $z_0$] at (0,-2.8)  {};
							\node[label=below: $z_5$] at (0.05,-3.34)  {};
							\node[label=above: $z_4$] at (-0.05,-4.29)  {};
							\node[label=below: $z_1$] at (0,-4.77)  {};
							\node[label=below: $z_{-4}$] at (0,-5.58)  {};						
							
							\node[label=below right: $G'_1$] at (2.06,1.06)  {};
							\node[label=below right: $X_1$] at (1.41,-2.4)  {}; 
							\node[label=above right: $G'^1_1$] at (1.08,-3.97)  {};																	
							\draw[blue,thick] (2.45,3.3) .. controls (-1.5,4.3) and (-1.5,3.8) .. (2.4,2.78);
							
							\draw (2.3,2.3) .. controls (-4.8,4) and (-4.8,2) .. (2,0.6);
							\draw[blue,thick] (2.2,1.87) .. controls (-4,3.3) and (-4,2.4) .. (2.1,1.05);
							
							\draw (0.8,-5.4) .. controls (-7.5,-7.4) and (-7.5,-3.9) .. (1.63,-1.45);
							\draw[blue,thick] (1,-4.5) .. controls (-6,-6.5) and (-6,-4.4) .. (1.4,-2.4);
							\draw (1,-4) .. controls (-4.5,-5.6) and (-4.5,-4.5) .. (1.3,-3);
							
							\draw (2,0.6) .. controls (4.2,0.1) and (4.2,-0.1) .. (1.69,-1);
							\draw (2.1,1.05) .. controls (5.2,0.3) and (5.2,-0.3) .. (1.63,-1.45);
							\draw[blue,thick] (2.2,1.87) .. controls (7.5,0.5) and  (7.5,-0.5).. (1.4,-2.4);
							\draw (2.3,2.3) .. controls (9,0.5) and  (9,-0.5).. (1.3,-3);
							\draw[blue,thick] (2.4,2.78) .. controls (10.8,0.5) and  (10.8,-1).. (1,-4);
							\draw[blue,thick] (2.45,3.3) .. controls (12,0.9) and  (12.4,-1.2).. (1,-4.5);
							
							\draw(1.69,-1)--(0.7,-1.3);
							\draw(0.8,-5.4)--(1.8,-5.1);								
							
							\draw (2.65,4.5) .. controls (0,-12) and  (-1.5,-12) .. (-2.15,4.5);
							%	\draw[red] (2.58,4) .. controls (2.1,1.5) and  (1.51,-3) .. (0.38,-6.5);
							\draw[red,thick] (2.07,1.06) .. controls (1.84,0) and  (1.3,-3.3) .. (1.09,-3.97);	
							
						\end{tikzpicture}
						
					}	\\
					
					\hspace{-0.5cm}
					\resizebox{0.6\textwidth}{!}{
						
						\begin{tikzpicture}
							
							\tikzstyle{every node}=[draw, circle, fill=white, minimum size=2pt, inner sep=0pt]
							\tikzstyle{dot}=[circle, fill=white, minimum size=0pt, inner sep=0pt, outer sep=-1pt]
							
							\node[label=right: $z^1_{-1}$] at (9.59,-0.1) {};
							\node[label=right: $z^1_{-2}$] at (8.53,-0.15) {};
							\node[label=right: $z^1_{-3}$] at (7.2,0)  {};
							\node[label=right: $z^1_0$] at (6.08,0)  {};
							\node[label=right: $z^1_{1}$] at (4.36,0)  {};
							\node[label=left: $z^1_{-4}$] at (3.6,-0.05)  {};
							
							\node[label=left: $z_0^4$] at (-0.5,3.8)  {};
							
							\node[label=right: $z_0^3$] at (-2.46,2.51)  {};
							\node[label=left: $z_{-1}^3$] at (-3.06,2.65)  {};
							
							\node[label=left: $z_{-1}^2$] at (-5.3,-5.3)  {};	
							\node[label=left: $z_0^2$] at (-4.19,-5.05)  {};		
							\node[label=left: $z_{1}^2$] at (-3.05,-4.7) {};
							
							\node[label=above: $z_2$] at (0.05,3.81)  {};
							\node[label=below: $z_3$] at (-0.01,3.52)  {};
							\node[label=above: $z_{6}$] at (0.04,2.77)  {};
							\node[label=below left: $z_{-1}$] at (0.04,2.32)  {};
							\node[label=above right: $z_{-2}$] at (0,1.55)  {};
							\node[label=below: $z_{7}$] at (0,1.07)  {};					
							
							\node[label=above: $z_{-3}$] at (0.05,-1.9)  {};
							\node[label=above: $z_0$] at (0,-2.8)  {};
							\node[label=below: $z_5$] at (0.05,-3.34)  {};
							\node[label=above: $z_4$] at (-0.05,-4.29)  {};
							\node[label=below: $z_1$] at (0,-4.77)  {};
							\node[label=below: $z_{-4}$] at (0,-5.58)  {};						
							
							\node[label=below right: $G_1^2$] at (1.62,-1.47)  {};
							\node[label=below right: $X_1$] at (1.41,-2.4)  {}; 
							\node[label=below right: $G_1^3$] at (1.3,-3)  {};					
							
							\draw[blue,thick] (2.45,3.3) .. controls (-1.5,4.3) and (-1.5,3.8) .. (2.4,2.78);
							
							\draw (2.3,2.3) .. controls (-4.8,4) and (-4.8,2) .. (2,0.6);
							\draw[blue,thick] (2.2,1.87) .. controls (-4,3.3) and (-4,2.4) .. (2.1,1.05);
							
							\draw (2,0.6) .. controls (4.2,0.1) and (4.2,-0.1) .. (1.69,-1);
							\draw (0.8,-5.4) .. controls (-7.5,-7.4) and (-7.5,-3.9) .. (1.63,-1.45);
							\draw[blue,thick] (1,-4.5) .. controls (-6,-6.5) and (-6,-4.4) .. (1.4,-2.4);
							\draw[blue,thick] (1,-4) .. controls (-4.5,-5.6) and (-4.5,-4.5) .. (1.3,-3);
							
							\draw[blue,thick] (2.1,1.05) .. controls (5.2,0.3) and (5.2,-0.3) .. (1.63,-1.45);
							\draw[blue,thick] (2.2,1.87) .. controls (7.5,0.5) and  (7.5,-0.5).. (1.4,-2.4);
							\draw (2.3,2.3) .. controls (9,0.5) and  (9,-0.5).. (1.3,-3);
							\draw[blue,thick] (2.4,2.78) .. controls (10.8,0.5) and  (10.8,-1).. (1,-4);
							\draw[blue,thick] (2.45,3.3) .. controls (12,0.9) and  (12.4,-1.2).. (1,-4.5);
							
							\draw(1.69,-1)--(0.7,-1.3);
							\draw(0.8,-5.4)--(1.8,-5.1);								
							
							\draw (2.65,4.5) .. controls (0,-12) and  (-1.5,-12) .. (-2.15,4.5);
							%	\draw[red] (2.58,4) .. controls (2.1,1.5) and  (1.51,-3) .. (0.38,-6.5);
							\draw[red,thick] (1.62,-1.47) .. controls (1.49,-2) and  (1.4,-2.5) .. (1.3,-3);	
							
						\end{tikzpicture}
						
					}
					
					&
					\hspace{-2cm}
					
					\resizebox{0.6\textwidth}{!}{
						
						\begin{tikzpicture}
							
							\tikzstyle{every node}=[draw, circle, fill=white, minimum size=2pt, inner sep=0pt]
							\tikzstyle{dot}=[circle, fill=white, minimum size=0pt, inner sep=0pt, outer sep=-1pt]
							
							\node[label=right: $z^1_{-1}$] at (9.59,-0.1) {};
							\node[label=right: $z^1_{-2}$] at (8.53,-0.15) {};
							\node[label=right: $z^1_{-3}$] at (7.2,0)  {};
							\node[label=right: $z^1_0$] at (6.08,0)  {};
							\node[label=right: $z^1_{1}$] at (4.36,0)  {};
							\node[label=left: $z^1_{-4}$] at (3.6,-0.05)  {};
							
							\node[label=left: $z_0^4$] at (-0.5,3.8)  {};
							
							\node[label=right: $z_0^3$] at (-2.46,2.51)  {};
							\node[label=left: $z_{-1}^3$] at (-3.06,2.65)  {};
							
							\node[label=left: $z_{-1}^2$] at (-5.3,-5.3)  {};	
							\node[label=left: $z_0^2$] at (-4.19,-5.05)  {};		
							\node[label=left: $z_{1}^2$] at (-3.05,-4.7) {};
							
							\node[label=above: $z_2$] at (0.05,3.81)  {};
							\node[label=below: $z_3$] at (-0.01,3.52)  {};
							\node[label=above: $z_{6}$] at (0.04,2.77)  {};
							\node[label=below left: $z_{-1}$] at (0.04,2.32)  {};
							\node[label=above right: $z_{-2}$] at (0,1.55)  {};
							\node[label=below: $z_{7}$] at (0,1.07)  {}; 			
							
							\node[label=above: $z_{-3}$] at (0.05,-1.9)  {};
							\node[label=above: $z_0$] at (0,-2.8)  {};
							\node[label=below: $z_5$] at (0.05,-3.34)  {};
							\node[label=above: $z_4$] at (-0.05,-4.29)  {};
							\node[label=below: $z_1$] at (0,-4.77)  {};
							\node[label=below: $z_{-4}$] at (0,-5.58)  {};						
							
							\node[label=below right: $G''_1$] at (0.72,-5.4)  {};
							\node[label=below right: $X_1$] at (1.41,-2.4)  {}; 
							\node[label=above right: $G''^1_1$] at (1.69,-1)  {};				
							
							\draw[blue,thick] (2.45,3.3) .. controls (-1.5,4.3) and (-1.5,3.8) .. (2.4,2.78);
							
							\draw[blue,thick] (2.3,2.3) .. controls (-4.8,4) and (-4.8,2) .. (2,0.6);
							\draw[blue,thick] (2.2,1.87) .. controls (-4,3.3) and (-4,2.4) .. (2.1,1.05);
							
							\draw[blue,thick] (0.72,-5.4) .. controls (-7.5,-7.4) and (-7.5,-3.9) .. (1.63,-1.45);
							\draw[blue,thick] (1,-4.5) .. controls (-6,-6.5) and (-6,-4.4) .. (1.4,-2.4);
							\draw[blue,thick] (1,-4) .. controls (-4.5,-5.6) and (-4.5,-4.5) .. (1.3,-3);
							
							\draw[blue,thick] (2,0.6) .. controls (4.2,0.1) and (4.2,-0.1) .. (1.69,-1);
							\draw[blue,thick] (2.1,1.05) .. controls (5.2,0.3) and (5.2,-0.3) .. (1.63,-1.45);
							\draw[blue,thick] (2.2,1.87) .. controls (7.5,0.5) and  (7.5,-0.5).. (1.4,-2.4);
							\draw[blue,thick] (2.3,2.3) .. controls (9,0.5) and  (9,-0.5).. (1.3,-3);
							\draw[blue,thick] (2.4,2.78) .. controls (10.8,0.5) and  (10.8,-1).. (1,-4);
							\draw[blue,thick] (2.45,3.3) .. controls (12,0.9) and  (12.4,-1.2).. (1,-4.5);
							
							\draw(1.69,-1)--(0.7,-1.3);
							\draw(0.72,-5.4)--(1.8,-5.1);								
							
							\draw (2.65,4.5) .. controls (0,-12) and  (-1.5,-12) .. (-2.15,4.5);
							%	\draw[red] (2.58,4) .. controls (2.1,1.5) and  (1.51,-3) .. (0.38,-6.5);
							\draw[red,thick] (1.69,-1) .. controls (1.5,-2.3) and  (1.05,-4) .. (0.72,-5.4);	
							
						\end{tikzpicture}
						
					}	
					
				\end{tabular}	
				
			\end{adjustbox}
			\vspace{-1.5cm}
			\caption{Top: Positions of points $G_1$ and $G'_1$. Bottom left: Position of a point, $G_1^2$, from the orbit of $G_1$, and hence it satisfies the property $(G^2_1, G^{3}_1)^u \cap (G^2_1, G^{3}_1)^s = \{ X_1 \}$. Bottom right: Position of a point, $G''_1$, that is not from the orbit of $G_1$, and therefore it does not satisfy the property $(G''_1, G''^{1}_1)^u \cap (G''_1, G''^{1}_1)^s = \{ X_1 \}$.}
			\label{fig:gg}
		\end{figure}
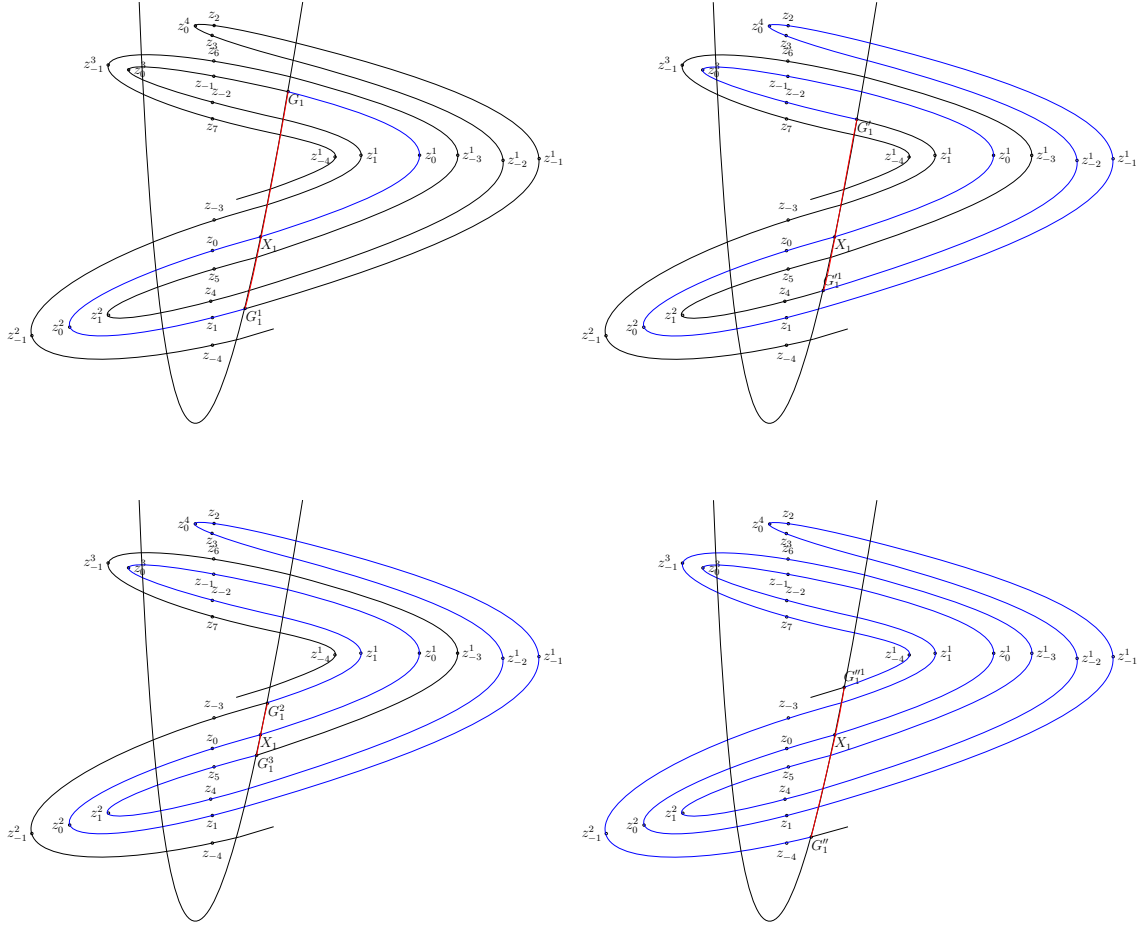	
		
		Let us denote by $A^j_i$, the region bounded by $[G^j_i, G^{j+1}_i]^u \cup [G^j_i, G^{j+1}_i]^s$, and by $A'^j_i$, the region bounded by $[G'^j_i, G'^{j+1}_i]^u \cup [G'^j_i, G'^{j+1}_i]^s$. Note that $\gamma_i \cap \Int A_i = \emptyset$, but $\gamma_i \cap \Int A'_i \neq \emptyset$ and in that intersection there exists a sequence of homoclinic points that converges to $G'_i$. Therefore, $H([G_1, G^1_1]^s) = [Q, Q^1]^s$ implies $Q \in \{ G_2^j : j \in \Z \}$ and the claim follows.
	\end{proof}
	
	\begin{prop}\label{prop:con}
		Suppose that $H : D_1 \to D_2$ is a conjugacy between $F_1$ and $F_2$. Then there exist $k \in \Z$, and a conjugacy $h : \T_1 \to \T_2$ between $f_1$ and $f_2$, such that $h\circ \pi_1 \circ F_1^{-n}(P) = \pi_2 \circ F_2^{-n+k} \circ H(P)$, for every $n \in \N$ and $P \in \Lambda_1$ for $b > 0$, and $P \in \Lambda_1 \cup W^u_{1,Y}$ for $b < 0$. Moreover, $h$ preserves stems and their levels; i.e., a stem of level $n \in \N_0$ of $\T_1$ is mapped to a stem of level $n$ of $\T_2$, and the natural order of stems is preserved.
	\end{prop}
	\begin{proof}
		By Lemma \ref{lem:gamma}, there exists $k \in \Z$ such that $F_2^k(H(\partial \gamma_{X_1})) = \partial \gamma_{X_2}$. Since $F_2^k \circ H$ is also a conjugacy between $F_1$ and $F_2$, without loss of generality we assume that $H(\gamma_{X_1}) = \gamma_{X_2}$, and moreover, the homoclinic (heteroclinic) points of $\gamma_{X_1}$ are mapped onto the homoclinic (heteroclinic) points of $\gamma_{X_2}$ and their order is preserved. Note that for $b > 0$,  $H([E_1^{-1}, X_1]^u) = [E_2^{-1}, X_2]^u$, the homoclinic points of $[E_1^{-1}, X_1]^u$ are mapped onto the homoclinic points of $[E_2^{-1}, X_2]^u$ and their order is preserved, as well as for $F_1^n([E_1^{-1}, X_1]^u)$ for every $n \in \N$. Also, for $b < 0$,  $H([E_1^{-1}, E_1]^u_{Y_1}) = [E_2^{-1}, E_2]^u_{Y_2}$, the heteroclinic points of $[E_1^{-1}, E_1]^u_{Y_1}$ are mapped onto the heteroclinic points of $[E_2^{-1}, E_2]^u_{Y_2}$ and their order is preserved, as well as for $F_1^n([E_1^{-1}, E_1]^u_{Y_1})$ for every $n \in \N$.
		
		Let $\gamma_1 \in \Gamma_1$. Since $\gamma_1 \subset W^s_1$ and $\partial \gamma_1$ are homoclinic points, there exists $n \in \N_0$ such that $\beta_1 := F_1^n(\gamma_1) \subseteq \gamma_{X_1}$. Let $\beta_2 := H(\beta_1) \subseteq \gamma_{X_2}$. Then $H(\gamma_1) =  F_2^{-n} \circ H \circ F_1^n(\gamma_1) = F_2^{-n}(\beta_2) =: \gamma_2 \in \Gamma_2$, and hence $H$ induces a homeomorphism $\tilde H : \Gamma_1 \to \Gamma_2$, $\tilde H(\gamma_1) = \gamma_2$, where $\gamma_2$ is defined as above. 
		
		Since $\Gamma_i$ is dense in $D_i$, $\tilde H$ induces a semi-conjugacy $h : \T_1 \to \T_2$ between $f_1$ and $f_2$ such that $\psi h \psi^{-1}|_{\Gamma_1} = \tilde H$. Recall, $\psi$ assigns to each equivalence class $[(s_n)_{n=1}^{\infty}]$ in $\T$ a continuum $\bigcap_{n=1}^\infty s_n$, which is a geometric realization of that class,
		$$\psi([(s_n)_{n=1}^{\infty}]) = \bigcap_{n=1}^\infty s_n.$$ 
		Let us prove that $h$ is an injection, and therefore a conjugacy. Let $x = [(s_n)_{n=1}^{\infty}], x' = [(s'_n)_{n=1}^{\infty}] \in \T_1$ and $x \ne x'$. Then there exists $N \in \N$ such that $s_k \cap s'_m = \emptyset$ for all $k, m \ge N$. Thus, $H(s_k) \cap H(s'_m) = \emptyset$, for all $k, m \ge N$, and $h(x) = [(H(s_n))_{n=1}^{\infty}] \ne [(H(s'_n))_{n=1}^{\infty}] = h(x')$. Therefore, $h$ is a conjugacy. Moreover, it follows from construction that $h$ preserves stems, that is a stem of level $n \in \N_0$ of $\T_1$ is mapped to a stem of level $n$ of $\T_2$, and the natural order of stems is preserved.
	\end{proof}
	
	Note that the conjugacy $h : (\T_1, f_1) \to (\T_2, f_2)$ induces a conjugacy 
	$$\tilde h : (\invlim (\T_1, f_1), \sigma_1) \to (\invlim (\T_2, f_2), \sigma_2),$$ given by $\tilde h (t_1, \dots , t_n, \dots ) = (h(t_1), \dots , h(t_n), \dots )$, such that $\tilde h \circ \Pi_1 = \Pi_2 \circ H$.
	
	\begin{theorem}\label{thm:tp}
		Suppose that $H : D_1 \to D_2$ is a conjugacy between two H\'enon maps $F_1$ and $F_2$. Then there exists $k \in \Z$ such that $H(S_{1,i}^n) = S_{2,i}^{k+n}$, for every $i, n \in \Z$.
	\end{theorem}
	\begin{proof}
		Let $H$ be a conjugacy between the H\'enon maps $F_1$ and $F_2$ such that  $H(\gamma_{X_1}) = \gamma_{X_2}$, as in the proof of Proposition \ref{prop:con}. Then $h$ maps every endpoint (every branch point) of $\T_1$ to the corresponding endpoint (branch point) of $\T_2$. 
		
		Recall that for $b > 0$ we have $H|_{\Lambda_1} = \Pi_2^{-1} \circ \tilde h \circ \Pi_1$, and for $b < 0$ we have $H|_{\Lambda_1 \cup W^u_{1,Y}} = \Pi_2^{-1} \circ \tilde h \circ \Pi_1$. Therefore, for every $i \in \Z$,
		$$H(S_{1,i}) = \Pi_2^{-1}(h \circ \pi_1(S_{1,i}), h \circ \pi_1 \circ F_1^{-1}(S_{1,i}), \dots , h \circ \pi_1 \circ F_1^{-n}(S_{1,i}), \dots ).$$
		By Lemma \ref{lem:ebtoeb} there exists $n \in \N$ such that $r_n := \pi_1 \circ F_1^{-n}(S_{1,i}) \in E_{X_1} \cup B_{X_1}$. If $r_n \in E_{X_1}$, by Lemma \ref{lem:ep}, $r_{n+j} := \pi_1 \circ F_1^{-n-j}(S_{1,i}) = f_1^{-j}(r_n) \in E_{X_1}$ for every $j \in \N$. If $r_n \in B_{X_1}$, by Lemma \ref{lem:bp}, $r_{n+j} := \pi_1 \circ F_1^{-n-j}(S_{1,i}) \in f_1^{-j}(r_n) \cap B_{X_1}$ for every $j \in \N$. By Proposition \ref{prop:con}, $h$ preserves stems and ordering of stems, and hence $h$ maps the endpoint of $E_{X_1}$ of a stem of $\T_1$ to the endpoint of $E_{X_2}$ of the corresponding stem of $\T_2$, and the branch points of $B_{X_1}$ of a stem of $\T_1$ to the corresponding branch points of $B_{X_2}$ of a corresponding stem of $\T_2$. Thus, $\pi_2 \circ F_2^{-n-j} \circ H(S_{1,i}) = h \circ \pi_1 \circ F_1^{-n-j}(S_{1,i}) \in E_{X_2} \cup B_{X_2}$, for every $j \in \N$. By Lemma \ref{lem:ebtoeb} and since we assumed that $H(\gamma_{X_1}) = \gamma_{X_2}$, it follows that $F_2^{-n-j} \circ H(S_{1,i}) = F_2^{-n-j}(S_{2,i})$ for every $j \in \N$. Therefore $H(S_{1,i}) = S_{2,i}$, for every $i \in \Z$. In general (without the assumption that $H(\gamma_{X_1}) = \gamma_{X_2}$) there exists $k \in \Z$ such that $H(S_{1,i}) = S_{2,i}^k$ and consequently $H(S_{1,i}^n) = S_{2,i}^{k+n}$, for every $i, n \in \Z$.
	\end{proof}
	The above result guarantees that for $b>0$ the orbits of basic critical points of $F_1$ are mapped onto the orbits of basic critical points of $F_2$. For $b<0$ this is guaranteed by the following corollary.
	\begin{cor}\label{cor:cp}
		Suppose that $H : D_1 \to D_2$ is a conjugacy between two orientation-preserving H\'enon maps $F_1$ and $F_2$. Then there exists $k \in \Z$ such that $H(z_{1,i}) = z_{2,i}^{k}$, for every $i \in \Z$. 
	\end{cor}
	\begin{proof}
		Without loss of generality we suppose that $H(\gamma_{X_1}) = \gamma_{X_2}$, and we want to prove that $H(z_{1,j}) = z_{2,j}$ for every $j \in \Z$. Recall that for $b < 0$ we have that $S_{i,j} = F_i(z'_{i,j})$ for $i = 1,2$ and $j \in \N_0$. Moreover, for every basic critical point $z_{i,j}$, $i = 1, 2$, $j \in \Z$, there exists a sequence of quasi-critical points $(z'_{i,j_n})_{n \in \N}$, $j_n \in \N_0$, that converges to $z_{i,j}$. Also, if a sequence of quasi-critical points converges, then its limit point is a critical point. Since by the proof of Theorem \ref{thm:tp} we have that $H(S_{1,j_n}) = S_{2,j_n}$, for every $n \in \N$, and $H(W^u_1) = W^u_2$, it follows that for every $j \in \Z$ there exists $k \in \Z$ such that $H(z_{1,j}) = z_{2,k}$.
		
		Let us denote by $\zeta_{i,j}$, $i = 1, 2$, $j \in \Z$, the (connected) components of $W^u_i \cap \Omega_i$, that are indexed in the following way: $X_i \in \zeta_{i,0}$, $\zeta_{i,j} \subset W^{u+}_i$, $\zeta_{i,-j} \subset W^{u-}_i$, for every $j \in \N$, and $\zeta_{i,j}, \zeta_{i,k}$ are consecutive (there are no other $\zeta_{i,m}$ between them) if and only if $|k - j| = 1$, $j, k \in \Z$. Note that $z_{i,-j} \in \zeta_{i,j}$ for $j = -1, 0, 1$, but in general if $z_{i,-j} \in \zeta_{i,k}$, then $|j| \le |k|$, since there are $\zeta_{i,m}$ that do not contain any basic critical point. Since $H$ is a conjugacy with $H(\gamma_{X_1}) = \gamma_{X_2}$, it follows that $H(\zeta_{1,j}) = \zeta_{2,j}$ for every $j \in \Z$. Therefore, $H(z_{1,0}) = z_{2,0}$ and consequently $H(z_{1,j}) = z_{2,j}$ for every $j \in \Z$.
	\end{proof}
	
	\begin{cor}\label{iff}
		Suppose that $H : D_1 \to D_2$ is a conjugacy between two H\'enon maps $F_1$ and $F_2$. Then the sets of kneading sequences of $F_1$ and $F_2$ coincide.
	\end{cor}
	\begin{proof}
		Without loss of generality, we assume that $H$ is a conjugacy such that $H(\gamma_{X_1}) = \gamma_{X_2}$. Then, by Theorem \ref{thm:tp} and Corollary \ref{cor:cp}, it follows that $H(z_{1,i}) = z_{2,i}$, $i \in \Z$. We want to prove that $z_{1,i}^n \in D_1^r$ if and only if $H(z_{1,i}^n) = z_{2,i}^n \in D_2^r$, for every $n, i \in \Z$, $n \ne 0$. 
		
		Let us first assume by contradiction that there exist $i \in \Z$ and $n \in \N$ such that $z_{1,i}^n \in D_1^l$, $z_{2,i}^n \in D_2^r$, and let $n$ and $i$ be such that $z_{1,i}^n$ is the closest point to $X_1$ with that property (that is, the arc $[X_1, z_{1,i}^n) \subset W^u_1$ does not contain any point $z_{1,j}^k$ with that property). Let $\ell_1 = [X_1, z_{1,i}^n] \subset W^u_1$ and $\ell_2 = [X_2, z_{2,i}^n] \subset W^u_2$ be arcs. Then $H(\ell_1) = \ell_2$, and all the turning and post-turning points of $\ell_1$ are mapped onto all the turning and post-turning points of $\ell_2$ and their order is preserved. Let $P \in \ell_1$ be the closest (post-)turning point to $z_{1,i}^n$, so $H(P) \in \ell_2$ is the closest (post-)turning point to $z_{2,i}^n$, and $P$ and $H(P)$ lie in the same component of $D_1$ and $D_2$, respectively. Then one of the arcs $[P, z_{1,i}^n] \subset W^u_1$, $[H(P), z_{2,i}^n] \subset W^u_2$ intersects, and the other one does not intersect the critical locus. Without loss of generality, we assume that $[P, z_{1,i}^n]$ intersects $\K_1$ (in the other case consider $H^{-1}$), and let denote by $Q$ the intersection point. Then $H \circ F_1(\ell_1) = F_2(\ell_2)$, and hence $H([F_1(P), z_{1,i}^{n+1}]) = [F_2(H(P)), z_{2,i}^{n+1}]$. Moreover, $[F_1(P), z_{1,i}^{n+1}]$ contains a turning point $F_1(Q)$, and $[F_2(H(P)), z_{2,i}^{n+1}]$ does not contain any turning point, a contradiction with $H(z_{1,i}) = z_{2,i}$, $i \in \Z$. Therefore, $z_{1,i}^n \in D_1^r$ if and only if $H(z_{1,i}^n) = z_{2,i}^n \in D_2^r$, for every $n \in \N$ and $i \in \Z$.
		
		Let us now consider $z_{1,i}^{-n}$ and $z_{2,i}^{-n}$ for every $n \in \N$ and $i \in \Z$. For every $i \in \Z$, $z_{1,i}$ and $z_{2,i}$ are basic critical points of $F_1$ and $F_2$ respectively, and lie in $\K_1$ and $\K_2$ respectively. Let $n \in \N$. Then $z_{1,i}^{-n}$ lies in the basic arc $[z_{1,j}^k, z_{1,m}^l]^u$ for some $j, m \in \Z$ and at least one of $k, l$ is positive, say $k > 0$. Since $z_{2,i}^{-n} \in [z_{2,j}^k, z_{2,m}^l]^u$ and, as we have already proved, $z_{1,j}^k \in D_1^r$ if and only if $z_{2,j}^k \in D_2^r$, it follows that $z_{1,i}^{-n} \in D_1^r$ if and only if $z_{2,i}^{-n} \in D_2^r$.
	\end{proof}
	
	The proof of Theorem \ref{thm:iff} follows now from Corollary \ref{iff} and Theorem \ref{oneway}.

	\section{The Lozi maps}\label{sec:lm}
	
	We note that all the results presented here for the H\'enon attractors with the Wang-Young parameters, hold also for the Lozi maps within the Misiurewicz parameter set $\mathcal{M} = \{ (a,b) \in \R^2 : b > 0, \ a\sqrt{2} - b > 2, \ 2a + b < 4 \}$. For the parameters in this set, Misiurewicz \cite{M} (see also \cite{MS2}) has shown the existence of strange attractors, with topological mixing. Also, the kneading theory (for the Lozi maps) has been developed by Misiurewicz and the second author in \cite{MS}, where analogues of Theorems \ref{admis} and \ref{esadmis} are proven. The analogue of Theorem \ref{thm:iff} is obtained as follows. One uses the following dictionary, which can be extracted from \cite{BS}.
	
	\begin{center}
		{\bf Dictionary}
		
		\begin{tabularx}{0.99\textwidth} { 
				| >{\raggedright\arraybackslash}X 
				|>{\raggedleft\arraybackslash}X | }
			\hline
			{\bf H\'enon maps from \cite{BC}} & {\bf Lozi maps from \cite{M} } \\
			\hline
			$D$  & $\Delta$  \\
			\hline
			$\K$ & $y$-axis$\, \cap\Delta$\\
			\hline
			$\Omega$ & $T_0$\\
			\hline
		\end{tabularx}
		
	\end{center}
	
	To prove that two Lozi maps $F_1$ and $F_2$ with the parameters in $\mathcal{M}$ are topologically conjugate if their sets of kneading sequences coincide, $\mfk_1 = \mfk_2$, small additional effort is needed, since for the Lozi map $F$ from $\mathcal{M}$, we do not know whether a point $P \in \Lambda_F \setminus W^u$ has at most two itineraries.
	
	First, using the entries from the dictionary, one can apply the proof of Proposition \ref{lem:dif1} to prove the following proposition:
	
	\begin{prop}\label{lem:dif1L}
		Let $F_1$ and $F_2$ be two Lozi maps such that $\mfk_1 = \mfk_2$. Let $\bar p, \bar q \in \Sigma_{F_1}$ be two different elements such that  $\pi_1(\bar p) = \pi_1(\bar q)$, and there exists only one $j \in \Z$ with $p_j \ne q_j$ ($p_i = q_i$ for all integers $i \ne j$). Then $\pi_2(\bar p) = \pi_2(\bar q)$.
	\end{prop}
	
	Now we consider the possibility that a point $P$ has more than two itineraries.
	
	\begin{prop}\label{lem:difnL}
		Let $F_1$ and $F_2$ be two Lozi maps such that $\mfk_1 = \mfk_2$. Let $n \in \N$. Suppose that $\bar p, \bar q \in \Sigma_{F_1}$ are such that $\pi_1(\bar p) = \pi_1(\bar q)$, and there exist $j_1, \dots , j_n \in \Z$ such that $p_{j_i} \ne q_{j_i}$ and $p_k = q_k$ for all integers $k \ne j_i$ and for every $i = 1, \dots , n$. Then $\pi_2(\bar p) = \pi_2(\bar q)$.
	\end{prop}
	\begin{proof}
		Let $P = \pi_1(\bar p)$. By the assumptions of the lemma, $F_1^{j_i}(P) \in y$-axis for $i = 1, \dots , n$. Therefore, the set $\pi^{-1}_1(P)$ contains all itineraries of the form 
		$$\la p_{\h j_1-1} \pm p_{j_1+1} \dots p_{j_2-1} \pm  p_{j_2+1}  \dots p_{j_n-1} \pm \ra p_{\h j_n+1},$$ where $p_k$, $k \ne j_i$, are coordinates of the itinerary $\bar p$, and for the $j_i$-th coordinates we can choose $+$ or $-$ in an arbitrary way.
		
		Now, the proof follows by induction and Lemma \ref{lem:dif1L}.
	\end{proof}
	
	\begin{prop}\label{prop:difL}
		Let $F_1$ and $F_2$ be two Lozi maps such that $\mfk_1 = \mfk_2$. Suppose that $\bar p, \bar q \in \Sigma_{F_1}$ are such that $\pi_1(\bar p) = \pi_1(\bar q)$, and there is a sequence $(j_i)_{i \in \Z} \subset \Z$, such that $p_{j_i} \ne q_{j_i}$ and $p_k = q_k$ for all integers $k \notin (j_i)_{i \in \Z}$, then $\pi_2(\bar p) = \pi_2(\bar q)$.
	\end{prop}
	\begin{proof}
		Let us assume that there exist $\bar p, \bar q \in \Sigma_{F_1}$ that satisfy the assumptions of the lemma. Let $P = \pi_1(\bar p)$. By the assumptions, $F_1^{j_i}(P) \in y$-axis for $i \in \Z$. Therefore, the set $\pi^{-1}_1(P)$ contains all itineraries of the form 
		$$\dots p_{j_{-1}-1} \pm p_{j_{-1}+1} \dots p_{j_0-1} \pm  p_{j_0+1}  \dots p_{j_1-1} \pm p_{j_1+1} \dots,$$ where $p_k$, $k \ne j_i$, are coordinates of the itinerary $\bar p$, and for the $j_i$-th coordinates we can choose $+$ or $-$ in an arbitrary way.
		
		Let $j_i$ be chosen such that the coordinate $j_0$ is the closest one to the decimal point (if there are two such indices take any of them), and $j_{-k-1} < j_{-k} < j_0 < j_k < j_{k+1}$ for every $k \in \N$ (as it is already shown in the sequence above). Let us consider a sequence $(\bar q^n)_{n \in \N_0}$, $\bar q^n \in \pi_1^{-1}(P)$, such that $\bar p$ and $\bar q^0$ disagree only at $j_0$ coordinate, and for every $n \in \N_0$, $\bar q^n$ and $\bar q^{n+1}$ disagree only at $j_{|n|}$ coordinates.
		By Proposition \ref{lem:difnL}, $\pi_2(\bar p) = \pi_2(\bar q^n)$ for every $n \in \N_0$. Note that the sequence $(\bar q^n)_n$ converges to $\bar q$. Since $\pi_2$ is continuous, $\pi_2(\bar p) = \pi_2(\bar q)$, as claimed in the proposition.
	\end{proof}
	
	Now again, using the entries from the dictionary, one introduces the equivalence relation on the space of symbolic sequences, to get the following theorem. 
	\begin{theorem}\label{Lozioneway}
		Two Lozi maps $F_1$ and $F_2$ with the parameters in $\mathcal{M}$ are topologically conjugate on their strange attractors if their sets of kneading sequences coincide, $\mfk_1 = \mfk_2$.
	\end{theorem}
	The results of Section \ref{sec:fp} for Lozi maps are proven in \cite{MS}. Finally, for the results of Section \ref{sec:cla} one uses the dictionary again, and the results in \cite{BS}, and almost verbatim copies the proofs of all the results in that section to get the following.
	\begin{theorem}\label{Lozithm:tp}
		Suppose that $H : \Delta_1 \to \Delta_2$ is a conjugacy between two Lozi maps $F_1$ and $F_2$ with parameters in $\mathcal{M}$, and with turning points $\{S_{1,i} : i \in \Z\}$ and $\{S_{2,i} : i \in \Z\}$ respectively. Then there exists $k \in \Z$ such that $H(S_{1,i}^n) = S_{2,i}^{k+n}$, for every $i, n \in \Z$.
	\end{theorem}
	That gives the following result for the Lozi maps.
	\begin{cor}\label{Loziiff}
		Suppose that $H : \Delta_1 \to \Delta_2$ is a conjugacy between two Lozi maps $F_1$ and $F_2$ with parameters in $\M$. Then the sets of kneading sequences of $F_1$ and $F_2$ coincide.
	\end{cor}

\appendix

\section{Infinitely many nonconjugate H\'enon maps in terms of the Wang-Young parameters.}

In this section we apply the classification results in terms of the kneading sets to show that there are infinitely many Wang-Young's parameters such that the corresponding H\'enon maps are not conjugate.

Let $q_a : [-1, 1] \to [-1, 1]$, $q_a(x):=1-ax^2$ for $x \in [-1, 1]$ and $a\in (0,2]$ be the quadratic family. Barge and Holte \cite{BH} showed that if $q_a$ has an attracting cycle, then there exists a $b_0 > 0$ such that for all $b \in [-b_0, b_0]$ the H\'enon maps $F_{a,b}$ are conjugate to the natural extension of $q_a$. Consequently, if $q_a$ and $q_{a'}$ have attracting periodic cycles and are nonconjugate, then there exists a $b_0 > 0$ such that for all $b, b' \in [-b_0, b_0]$ the H\'enon maps $F_{a,b}$ and $F_{a',b'}$ are nonconjugate, resulting in a countable collection of nonconjugate H\'enon maps. Below we extend this result to the case when $q_a$ and $q_{a'}$ are Misiurewicz maps. 

Recall that a map $q_a$ is called a \emph{Misiurewicz map} if:
\begin{enumerate}
	\item There is no $x \in [-1, 1]$ with $f^n(x) = x$ and $|(f^n)'(x)| \leq 1$,
	\item $\inf_{n>0}d(f^n(0), 0) > 0$.
\end{enumerate}
Note that for two quadratic Misiurewicz maps $q_{a}$ and $q_{a'}$ with parameters $a, a'\in [1.5,2]$, $a\neq a'$ implies that $q_{a}$ and $q_{a'}$ are not conjugate, and there are uncountable many nonconjugate quadratic Misiurewicz maps.

\begin{theorem}
	Let $a', a'' \in [1.5, 2]$ be distinct and such that $q_{a'}, q_{a''}$ are  Misiurewicz maps. There exist a $\xi>0$, and positive measure sets $\Delta_1\subset ([a' - \xi, a' + \xi]\times [-\xi,\xi])\cap \WY$ and $\Delta_2 \subset ([a'' - \xi, a'' + \xi]\times [-\xi,\xi])\cap \WY$, such that for any $(a_1, b_1)  \in \Delta_1$ and $(a_2, b_2) \in \Delta_2$ the H\'enon maps $F_{a_1,b_1}$ and $F_{a_2,b_2}$ are not conjugate. 
\end{theorem}
\begin{proof}
	Note that 
	\begin{enumerate}[(1)]
		\item $F_{a,b}$ converges uniformly to $F_{a_0,0}$ as $a \to a_0$ and $b \to 0$, and $F_{a_0,0}$ is conjugate to the quadratic map $q_{a_0}$ on its nonwandering set contained in $\{y=0\}$.
		\item For every  $\eps > 0$ and $k_0 \in \N$, there exists $\delta > 0$ such that, for every $(a,b)\in [1.5, 2] \times ([-1,1] \setminus \{ 0 \})$ and $S \subset \mathbb{R}^2$, if $\diam(S) < \delta$ then $\diam(F_{a,b}^k(S)) < \eps$, for every $k \le k_0$.
		\item By \cite[Theorem 1.1 (2)]{WY}, for every $\eps > 0$, there exists $m_\eps \in \N$ such that $d_\C (z^k) > \eps$ for every $(a, b) \in \WY$, $z \in \C$ and $k \le m_\eps$. Also,  $\eps \to 0$ implies $m_{\eps} \to \infty$. (Recall that $d_\C(p)$ denotes the distance between a point $p$ and the critical set $\C$ defined below Theorem \ref{thm:WY}.)
	\end{enumerate}
	
	Let $a', a'' \in [1.5, 2]$ be distinct and such that the quadratic maps $q_{a'}, q_{a''}$ are  Misiurewicz maps. Let $\eps > 0$ be arbitrary. Let $m_\eps \in \N$ be as in (3). Let $\delta > 0$ be as in (2) for $k_0 = m_\eps$. By (1), there exists $n_0 \in \N$ such that $[a' - \frac{1}{n_0}, a' + \frac{1}{n_0}]$ and $[a'' - \frac{1}{n_0}, a'' + \frac{1}{n_0}]$ are disjoint, and
	$$(a, b) \in \WY \cap ([a' - \frac{1}{n_0}, a' + \frac{1}{n_0}] \cup [a'' - \frac{1}{n_0}, a'' + \frac{1}{n_0}])  \times [-\frac{1}{n_0}, \frac{1}{n_0}]$$ implies $\diam(\C) < \delta$. 
	
	Suppose by contradiction that for each $n > n_0$ there exist $$(a_1^n, b_1^n) \in \WY \cap [a' - \frac1n, a' + \frac1n] \times [-\frac1n, \frac1n]$$ and $$(a_2^n, b_2^n) \in \WY \cap [a'' - \frac1n, a'' + \frac1n] \times [-\frac1n, \frac1n],$$ such that the H\'enon maps $F_{a^n_1,b^n_1}$ and $F_{a^n_2,b^n_2}$ are conjugate. Let $\C^n_i$ denote the set of critical points of $F_{a^n_i,b^n_i}$ for $i = 1, 2$ and $n \in \mathbb{N}$, $n > n_0$.
	
	By the choice of $n_0$, we have that $\diam(\C^n_i) < \delta$ for $i = 1, 2$ and $n \in \mathbb{N}$, $n > n_0$, and by (2), $\diam(F_{a^n_i,b^n_i}^k(\C^n_i)) < \eps$, for every $k \le m_\eps$. In addition by (3), $d_{\C^n_i} (z^k) > \eps$ for every $z \in \C^n_i$ and $k \le m_\eps$, implying that, for $i \in \{ 1, 2 \}$ and $n \in \mathbb{N}$, $n > n_0$, all points of $F_{a^n_i,b^n_i}^k(\C^n_i)$ lie on the same side of the critical locus, and consequently, every kneading sequence $\bar k^n_i$ of $F_{a^n_i,b^n_i}$ has its right tail $\ra k_{\h 0,i}^n$ agreeing with the kneading sequence of $q_{a_i}$ on the first $m_n$ coordinates. Note that $m_n \to \infty$ as $n \to \infty$. But this leads to a contradiction since the kneading sequences of $q_{a_1}$ and $q_{a_2}$ are distinct, and $\mathfrak{K}_{F_{a^n_1,b^n_1}} = \mathfrak{K}_{F_{a^n_2,b^n_2}}$. Therefore, there exist positive measure sets $$\Delta_1 \subset \WY \cap [a' - \frac{1}{n_0}, a' + \frac{1}{n_0}] \times [-\frac{1}{n_0}, \frac{1}{n_0}],$$ $$\Delta_2 \subset \WY \cap [a'' - \frac{1}{n_0}, a'' + \frac{1}{n_0}] \times [-\frac{1}{n_0}, \frac{1}{n_0}]$$ such that for any $(a_1, b_1)  \in \Delta_1$, $(a_2, b_2) \in \Delta_2$ the H\'enon maps $F_{a_1,b_1}$ and $F_{a_2,b_2}$ are not conjugate. 
\end{proof}

\section{All branch points of $\mathbb{T}$ come from tangencies and have order three.}

In this section we prove that every branch point of $\T$ has order three and represents a tangential intersection of the boundary of $D$ and a stable manifold. As a corollary, we obtain that $\T$ is a unique topological object, the same for all parameters in $\WY$, namely the universal dendrite of order three. 

Let $\d$ be a dendrite which is not a finite tree. Let us denote by $E_\d$ the set of all end points of $\d$, and by $B_\d$ the set of all branch points of $\d$. In Definition \ref{df:BP} we have defined a set of branch points $B_X$ of our dendrite $\T$ (assigned to the H\'enon map $F$) as the set of all those branch points that lie in stems. We want to prove that $B_X = B_\T$, and that every branch point has order three.

\begin{prop}\label{prop:order3}
	Every branch point in $B_X$ has order three.
\end{prop}
It is noteworthy that the above result is an analogy with the result of Williams for hyperbolic attractors \cite{Williams}, where the order of (finite) trees in the inverse limit representation is also three.
\begin{prop}\label{prop:AllBP}
	The set of branch points $B_X$ coincides with the set of all branch points of $\mathbb{T}$.
\end{prop}

Let us first recall some results from literature that we use in our proof of  Propositions \ref{prop:order3} and \ref{prop:AllBP}.	

\begin{df}\cite[bottom of p.\ 379 and a comment between Propositions 2.3 and 2.4]{BenedicksViana}
	A point $P$ is called {\bf expanding} if there exists $\lambda > e^{-20}$ such that $\| DF^j(P)(1,0) \| \ge \lambda^j$ for all $j \in \N$.
\end{df}

\begin{df}\cite[below Proposition 2.4]{BenedicksViana}\label{df:ll}
	Let $z$ be a point. A curve $\zeta = \zeta(z) = \{ (x(y), y) : |y| \le 1/10 \}$ is called a {\bf long stable leaf} if the following holds:
	\begin{enumerate}[(1)]
		\item $z \in \zeta$,
		\item $|x'| \le C\sqrt{|b|}$ and $|x''| \le C \sqrt{|b|}$, where $C > 1$ is a constant that does not depend on $b$,
		\item $d(F^n(P), F^n(Q)) \le (Cb)^n d(P, Q)$, for every $P, Q \in \zeta$ and $n \in \N$.
	\end{enumerate}
\end{df}

\begin{rem}\label{rem:ll}
	By \cite[Propositions 2.4 and 3.3]{BenedicksViana}, if $z$ is an expanding point then its stable manifold $W^s_z$ contains a long stable leaf $\zeta(z)$. Moreover, if $z_1, z_2$ are expanding points then $\operatorname{angle} (t(P_1), t(P_2)) \le C\sqrt{|b|} \, d(P_1, P_2)$, for every $P_1 \in \zeta(z_1), P_2 \in \zeta(z_2)$, where $t(P_i)$ denotes any norm 1 vector tangent to $\zeta(z_i)$ at $P_i$, $i = 1, 2$. Also, given any $z \in F(\mfc)$, there exists a sequence of long stable leaves $(\zeta_j)_{j \in \N}$ accumulating $W^s_z$ exponentially fast from the left.
\end{rem}

\begin{proof}[Proof of Proposition \ref{prop:order3}]
	Let $b > 0$. By Remark \ref{rem:ll}, every basic critical point $z \in \mfc$ lies in its stable manifold $W^s_z$. Moreover, for a long stable leaf $\zeta = \zeta(z^1)$ we have $z \in F^{-1}(\zeta) \subset W^s_z$. Since every  $\zeta \ne \zeta(z_0^1)$ intersects $\partial^u D$ transversely at two points that lie on the opposite sides of $z_0^1$, $F^{-1}(\zeta)$ intersects $\partial^u \Omega$ transversely at two points that lie on the opposite sides of $z_0$. Therefore, for $\alpha \in \A$ such that $z \in \alpha$, we have $\alpha \subset F^{-1}(\zeta(z^1))$ and $\alpha$ also intersects $\partial^u \Omega$ `transversely' at two points that lie on the opposite sides of $z_0$. Moreover, $\alpha \subset W^s_z$ intersects $W^u$ tangentially at $z$, and $W^s_{z^{-j}}$ and $W^u$ intersect tangentially at $z^{-j}$ for every $j \in \N_0$. 
	
	Let us now consider a special kind of basic critical points. Let $z \in \mfc$ and $\alpha \in \A$ be such that $z \in \alpha$ and for every small $\eps > 0$, 
	\begin{equation}\label{eqn:eps}
		(\alpha \setminus \{ z \}) \cap B_\eps(z) \cap \Lambda \ne \emptyset.
	\end{equation}  
	We want to prove now that no long stable leaf $\zeta$ contains any pre-critical point $z^{-j}$ that satisfies (\ref{eqn:eps}). Let us assume by contradiction that $z^{-j} \in \zeta$ for some $\zeta$, $z$ and $j$ as above. Let us consider $F^j(\zeta)$. Let $\alpha \in \A$ be such that $z \in \alpha$. The intersection of $\alpha$ and $F^j(\zeta)$ is an arc that contains $z$.  By \cite[Theorem 1.1]{WY}, arbitrarily close to $z$, $W^u$ intersects $\alpha$ in at least two points $Q_1, Q_2$ on the opposite sides of $z$, $z \in [Q_1, Q_2]^s_z \subset \alpha \subset W^s_z$, and such that $[Q_1, Q_2]^u \cap F^n(\C) = \emptyset$ for every $n \in \N$ (see \cite[Theorem 1.1.(1)(ii)-(iii)]{WY}). Therefore, arbitrarily close to $z^{-j}$, there exists an arc $[Q'_1, Q'_2]^u \subset W^u$ that intersects $\zeta \subset W^s_{z^{-j}}$ at points $Q'_1, Q'_2$ on the opposite sides of $z^{-j}$ and such that $[Q'_1, Q'_2]^u \cap F^n(\C) = \emptyset$ for every $n \in \N$. This contradicts Definition \ref{df:ll} (2), and so, $z^{-n} \nin \zeta$ for every long stable leaf $\zeta$ and $n \in \N_0$ and hence $z^{-n} \nin F^{-1}(\zeta)$ for every long stable leaf $\zeta$ and $n \in \N$.
	
	Note that, by definition of $B_X$, if $z^{-k} \in \alpha \cap \partial^u D$ for some non-degenerate $\alpha \in \A$, $z \in \mfc$ and $k \in \N_0$, then $\pi(\alpha) \in B_X$. Moreover, $z^{-j}$ satisfies (\ref{eqn:eps}) for every $j \in \N_0$. If $z^{-k}$ is the only point of tangential intersection of $\alpha$ and $\partial^u D$, then order of $\pi(\alpha)$ is three. For $\pi(\alpha)$ to have order greater than three, it is necessary that $\alpha$ intersects $\partial^u D$ tangentially at more than one point. An immediate consequence of everything above is that if $\alpha \in \A$ contains $z_{-1}$, then $\pi(\alpha) \in B_X$ is a branch point of order three (recall that $z_{-1} \in \partial^u D$). 
	
	Let $t \in B_X$ and $\alpha_t := \psi(t)$. We want to prove that $t$ has order three. Let us assume by contradiction that the order of $t$ is greater than three. Then $\alpha_t$ intersects $\partial^u D$ tangentially at two different points $z_i^{-k}$ and $z_j^{-d}$. If $k = d$, then $i \ne j$ and $z_i, z_j \in F^k(\alpha_t)$, a contradiction since, by the first paragraph of this proof, there exist $\alpha', \alpha'' \in \A$, $z_i \in \alpha'$, $z_j \in \alpha''$ and $\alpha' \cap \alpha'' = \emptyset$. If $d < k$ then $z_i, z_j^{d-k} \in F^k(\alpha_t) \subset \alpha'$ and $d - k < 0$, a contradiction, since $\alpha'$ does not contain a basic pre-critical point that satisfies (\ref{eqn:eps}). Therefore, $\alpha_t$ intersects tangentially $\partial^u D$ at only one basic (pre-)critical point and $t$ has order three.
	
	Let $b < 0$. Note that the results in \cite{BenedicksViana} mentioned in Remark \ref{rem:ll} apply to this case as well when we replace basic (pre-)critical points with quasi (pre-)critical points, and the first paragraph of the proof also holds for the quasi (pre-)critical points. The only difference is that $W^u_Y$ does not accumulate on itself, and hence (1) does not hold for the quasi-critical points. Nevertheless, we will prove that also in this case, no long stable leaf $\zeta$ contains any quasi pre-critical point $z'^{-j}$. We again assume by contradiction that $z'^{-j} \in \zeta$ for some $\zeta$, $z'$ and $j$. Let $\alpha \in \A$ be such that $z' \in \alpha$. Since $z' \in F^j(\zeta)$, intersection of $\alpha$ and $F^j(\zeta)$ is an arc that contains $z'$. Let $\Omega_\alpha \subset \Omega$ be the region bounded by $\alpha$ and $\partial^u \Omega$. In this case, also by \cite[Theorem 1.1.(1)(ii)-(iii)]{WY}, there exist arcs in $W^u_Y \cap \Omega_\alpha$ that have boundary points in $\alpha$ on the opposite sides of $z'$, and among them there exists one arc that is the closest to the quasi-critical point $z'$. If we denote the boundary points of that arc by $Q_1, Q_2$, and let $Q'_1 = F^{-j}(Q_1)$, $Q'_2 = F^{-j}(Q_2)$, then the rest of the proof for the case $b < 0$ is analogous to the case $b > 0$, with replacing basic (pre-)critical points with quasi (pre-)critical points.
\end{proof}

\begin{theorem}\cite[Lemma 4.2 and Theorem 1.2]{BS}\label{BS}
	The set of branch points $B_X$ is a dense subset of $\T$.
\end{theorem}

If an arc $A$ has the end points $a, b$, we write $A = ab$. A variant of the following result was proved in \cite[Proposition 3.2 and Corollary 3.5]{BMS}. For completeness, we include a proof. 

\begin{prop}\label{prop:Ai}
	Let $\d$ be a dendrite, $E_\d$ be the set of its end points, and $S$ be a dense subset of $\d$. Suppose that there exists a sequence of arcs $(A_n)_{n \in \N}$ such that 
	\begin{itemize}
		\item[(i)] $S \subset \bigcup_{n=0}^\infty A_n$, and
		\item[(ii)] $\bigcup_{n=0}^k A_n$ is connected for any $k \in \mathbb{N}$.
	\end{itemize}
	Then 
	$\d = E_\d \cup \bigcup_{n=0}^\infty A_n$.
\end{prop}	
\begin{proof}
	Suppose there exists an $a \in \d \setminus E_\d$ such that $a \notin \bigcup_{n=0}^\infty A_n$. Let $b \in \bigcup_{n=0}^\infty A_n \setminus E_\d$. 
	Since $a, b \notin E_\d$ the arc $ab$ can be extended from both ends to an arc $a'b' \subset \d \setminus E_\d$
	so that $a'a \cap ab = \{a\}$ and $ab\cap bb' = \{b\}$. Let $\d_a$ and $\d_b$ be dendrites contained in
	$\d \setminus ab$ such that $a' \in \Int(\d_a)$ and $b' \in \Int(\d_b)$. Note that each point of $ab$ separates
	$\d$ between $\d_a$ and $\d_b$. Since $S$ is dense in $\d$, there are points $s_a \in \d_a$ and $s_b \in \d_b$. By condition (i) there exists an $m$ such that $s_a, s_b \in \bigcup_{n=0}^m A_n$. By condition (ii) $\bigcup_{n=0}^m A_n$ is connected, hence a tree, and so $s_as_b \subset \bigcup_{n=0}^m A_n$. Since $ab \subset s_as_b$ it follows that $a \in \bigcup_{n=0}^m A_n$, leading to a contradiction. 
\end{proof}

\begin{df}
	\begin{enumerate}
		\item Let $k \in \N_0$. We say that $\alpha$ has {\bf separation type} $k$ if $\alpha \in \A^k$. 
		\item Let $K$ be a component of $F^{-n}(\K) \cap D$. We say that $K$ has the {\bf separation type} $k$ if $K$ separates two different elements of $\A^k_i$ for some $i \in \{ 1, \dots , m_k \}$, where $m_k$ is the number of stems of level $k$ (see Subsection \ref{ss:stems}).
	\end{enumerate}
\end{df}

\begin{rem}
	Note that every $\alpha$ such that $\pi(\alpha)$ is a branch point has two separation types, since every branch point lies in the intersection of two stems. If $\pi(\alpha)$ is not a branch point then $\alpha$ has a unique separation type. This also holds for any component of $F^{-n}(\K) \cap D$.
\end{rem}

\begin{proof}[Proof of Proposition \ref{prop:AllBP}]
	Since $B_X$ is dense in $\T$, and each element of $B_X$ belongs to a stem, it is enough to show that the union of all stems of the level at most $n$ is connected and apply Proposition \ref{prop:Ai}. 
	
	Recall that $\pi(\A^0) = B^0$, $\pi(\A^1) = B^1$. Let $\alpha \in \A$ be such that $z_{-1} \in \alpha$. Then by the proof of Proposition \ref{prop:order3} $\{ \alpha \} = \A^0 \cap \A^1$, so $\pi(\A^0 \cup \A^1) = B^0 \cup B^1$ is connected. Let us suppose that $\pi(\bigcup_{i=0}^{n-1} \A^i)$ is connected. We want to prove that $\pi(\bigcup_{i=0}^n \A^i)$ is connected.
	
	Recall that $\A^n = \bigcup_{j=1}^{m_n} \A^n_j$, where $m_n$ is the number of stems of level $n$. Also, each stem $B^n_j = \pi(\A^n_j)$ has two end points, one of them is an end point of $\T$ and the other one is a branch point of $\T$. Denote that branch point as $b^n_j = \pi(\alpha_{b^n_j})$. Note that $F^{-n}(\A^0) \cap D = \bigcup_{i=0}^n \A^i$. Also, $(\A^n \setminus \bigcup_{j=1}^{m_n} \alpha_{b^n_j}) \cap \bigcup_{i=0}^{n-1} \A^i = \emptyset$, so $\pi(\bigcup_{i=0}^n \A^i)$ is connected if $\bigcup_{j=1}^{m_n} \alpha_{b^n_j} \subset \bigcup_{i=0}^{n-1} \A^i$.
	
	Note also that $\K$ has the separation type zero, as well as each component of $F^{-1}(\K) \cap D$. Therefore, if $K$ is a component of $F^{-n}(\K) \cap D$, then the separation type of $K$ is less than $n$.
	
	To simplify notation, let $t := b^n_j$, for some $j \in \{ 1, \dots , m_n \}$, and $\alpha_t := \psi(t)$. By Proposition \ref{prop:order3}, $t$ has order three. For $b>0$, we have that $\alpha_t \cap W^u$ contains a unique basic pre-critical point. Since $\alpha_t \in \A^n$, we have that $\alpha_t \cap W^u = \{ z_k^{-n} \}$, for some $k \in \Z$, and there exists a component $K$ of $\K^{-n} \cap D$ such that $z_k^{-n}$ is the terminal point of $K$ (see the paragraph above Lemma \ref{lem:ebp}). This implies that one separation type of $\alpha_t$ is the same as the separation type of $K$, which is less than $n$. Therefore, $\alpha_t$ lies in a stem with level smaller than $n$, which completes the proof for $b>0$. For $b<0$ the proof is analogous, since then $\alpha_t \cap W^u_Y$ contains a unique quasi pre-critical point.
\end{proof}	

\begin{cor}(All branch points come from tangencies)
	Suppose $t \in \T$ is such that $\alpha_t = \psi(t)$ is nondegenerate. Then $t$ is a branch point if and only if there exist $z \in \mfc$ (for $b>0$), or $z\in \H$ (for $b<0$), and $k \in \N_0$ such that $z^{-k} \in \alpha_t \cap \partial^uD$. %where $\alpha_t := \psi(t) \in \A \subset D$. 
	In that case $\alpha_t \subset W^s_{z^{-k}}$ is an arc, and $z^{-k}$ is the only (pre-)critical point in $\alpha_t$ (or quasi (pre-)critical point).
\end{cor}	
\begin{cor}(No letter $Y$)
	Let $\alpha \in \A$. There is no point $P \in \alpha$ such that $\alpha \setminus P$ consists of three or more components, each intersecting $\partial D$.
\end{cor}
Finally we note the following corollary.
\begin{cor}
	The dendrite $\T$ is homeomorphic to the universal dendrite of order three, for any $(a,b)\in\WY$.
\end{cor}
\begin{proof}
	This follows from  Theorem \ref{BS}, Proposition \ref{prop:order3} and \cite[Theorem 6.2]{ChD}.
\end{proof}
Note that, in particular, $\T$ is homeomorphic to the continuum self-similar tree, whose geometric properties were studied by Bonk and Meyer in \cite{BoMe}.

\noindent
	Jan P. Boro\'nski\\	
	Faculty of Mathematics and Computer Science\\
	Jagiellonian University in Krak\'ow\\
	ul. Łojasiewicza 6, 30-348 Kraków, Poland\\
	-- and --\\
	National Supercomputing Centre IT4Innovations\\ 
	IRAFM, University of Ostrava\\
	30. dubna 22, 70103 Ostrava, Czech Republic
	\\
	\href{mailto:jan.boronski@uj.edu.pl}{jan.boronski@uj.edu.pl} \\
	\url{https://matinf.uj.edu.pl/en_GB/pracownicy/wizytowka?person_id=Jan_Boronski}\\

\noindent
	Sonja \v Stimac\\
	Department of Mathematics\\
	Faculty of Science, University of Zagreb\\
	Bijeni\v cka 30, 10\,000 Zagreb, Croatia
	\\
	\href{mailto:sonja@math.hr}{sonja@math.hr}\\
	\url{https://web.math.pmf.unizg.hr/~sonja/}
	
\end{document}